\documentclass[12pt]{amsart}
\usepackage{latexsym,amsmath,amsfonts,amsthm,amssymb,color}
\usepackage{fullpage}

\newtheorem{theorem}{Theorem}%[section]
\newtheorem{lemma}[theorem]{Lemma}

\newtheorem{proposition}[theorem]{Proposition}

\newcommand{\R}{{\mathbb R }}

\newcommand{\C}{{\mathbb C }}

\renewcommand{\S}{{\mathbb S }}

\newcommand{\la}{\langle}
\newcommand{\ra}{\rangle}

\renewcommand{\la}{\langle}
\renewcommand{\ra}{\rangle}

\newtheorem{thm}{Theorem}[section]

\newtheorem{lem}[thm]{Lemma}
\newtheorem{prop}[thm]{Proposition}

\newcommand{\case}[2]{\noindent\textbf{Case #1:}(\emph{#2})}
\newcommand{\step}[2]{\noindent\textbf{Step #1:}(\emph{#2})}

\newcommand{\lp}[2]{\Vert \, #1 \, \Vert_{#2}}

\newcommand{\td}{\widetilde}

\newcommand{\dint}{ {\int\!\!\!\!\int} }
\newcommand{\tint}{ {\int\!\!\!\!\int\!\!\!\!\int} }
\newcommand{\snabla}{ {\slash\!\!\!\!\nabla} }

\newcommand{\Diff}{\text{Diff}}

\newcommand{\ret}{\vspace{.1cm}}

\begin{document}
\title[The MKG equation]{ Global well-posedness for the Maxwell-Klein
  Gordon equation in $4+1$ dimensions. Small energy.}

% \author{ Daniel Tataru } \address{Department of Mathematics,
% University of California, Berkeley} \curraddr{}
% \email{tataru@math.berkeley.edu} \thanks{}

\author{Joachim\ Krieger} \address{B\^{a}timent des Math\'ematiques,
  EPFL, Station 8, CH-1015 Lausanne, Switzerland}
\email{joachim.krieger@epfl.ch}

\author{Jacob\ Sterbenz}
% Address of record for the research reported here
\address{Department of Mathematics, University of California San
  Diego, La Jolla, CA 92093-0112} \email{jsterben@math.ucsd.edu}

\author{Daniel\ Tataru} \address{Department of Mathematics, The
  University of California at Berkeley, Evans Hall, Berkeley, CA
  94720, U.S.A.}  \email{tataru@math.berkeley.edu}

\thanks{The first author was partially supported by the Swiss National
  Science Foundation. The second  author was partially
  supported by the NSF grant DMS-1001675 . 
The third author  was partially supported by the NSF grant 
DMS-0801261 and by the Simons Foundation.
The first author would
  like to express his thanks to the University of California,
  Berkeley, for hosting him during the summer 2011 and 2012.}

\begin{abstract}
  We prove that the critical Maxwell-Klein Gordon equation on
  $\R^{4+1}$ is globally well-posed for smooth initial data which are
  small in the energy. This reduces the problem of global regularity
  for large, smooth initial data to precluding concentration of
  energy.
\end{abstract}

\maketitle

%%%%%%%%%%%%%%%%%%%%%%%%%%%%%
% ------------------------------------------------------------------------
% ------------------------------------------------------------------------
%%%%%%%%%%%%%%%%%%%%%%%%%%%%%

\section{Introduction}

Let $\R^{4+1}$ be the five dimensional Minkowski space equipped with
the standard Lorentzian metric $g =
\text{diag}(1,-1,-1,-1,-1)$. Denote by $\phi: \R^{4+1}\rightarrow \C$
a scalar function, and by $A_{\alpha}: \R^{4+1}\rightarrow \R$,
$\alpha = 1\ldots,4$, a real valued connection form, interpreted as
taking values in $isu(1)$. Introducing the curvature tensor
\[
F_{\alpha\beta}: = \partial_{\alpha}A_{\beta}
- \partial_{\beta}A_{\alpha}
\]
as well as the covariant derivative
\[
D_{\alpha}\phi:= (\partial_{\alpha} + iA_{\alpha})\phi
\]
the {\it{Maxwell-Klein Gordon system}} are the Euler-Lagrange
equations associated with the formal Lagrangian action functional\footnote{
The Lagrangian below should also contain a mass term, which we choose
to neglect here; thus the system under consideration might be 
more aptly named ``the Maxwell-scalar field system''. For historical reasons
we have chosen to retain the  ``Maxwell-Klein Gordon'' terminology.}
\[
\mathcal{L}(A_{\alpha}, \phi): =
\frac{1}{2}\int_{\R^{4+1}}\big(\frac{1}{2}D_{\alpha}\phi\overline{D^{\alpha}\phi}
+ \frac{1}{4}F_{\alpha\beta} F^{\alpha\beta} \big)\,dxdt;
\]
here we are using the standard convention for raising indices.
Introducing the covariant wave operator
\[
\Box_{A}: = D^{\alpha}D_{\alpha}
\]
we can write the Maxwell-Klein Gordon system in the following form
\begin{equation}\label{eq:MKGgeneral}
\begin{split}
  & \partial^{\beta}F_{\alpha\beta} = -J_\alpha:= \Im(\phi\overline{D_\alpha \phi}) ,
\\
& \Box_A\phi = 0
\end{split}
\end{equation}

One key feature of this system is the underlying {\it{Gauge
    invariance}}, which is manifested by the fact that if
$(A_{\alpha}, \phi)$ is a solution, then so is
$(A_{\alpha} - \partial_{\alpha}\chi, e^{i\chi}\phi)$. This allows us
to impose an additional Gauge condition, and we shall henceforth
impose the {\it{Coulomb Gauge condition}} which requires
\begin{equation}\label{eq:Coulomb}
  \sum_{j=1}^4\partial_j A_j = 0
\end{equation}
The MKG-CG system can be written explicitly in the following form
\begin{subequations}
  \begin{align}
    \Box A_i \ &= \ \mathcal{P}_i J_x 
     \label{MKG10}\\
    \Box_A \phi \ &= \ 0   \label{MKG20}\\
    \Delta A_0 \ &= \  J_0  \label{MKG30}\\
    \Delta \partial_t A_0 \ &= \ \nabla^i J_i
    , \label{MKG40}
  \end{align}
\end{subequations}
where the operator $\mathcal{P}$ denotes the Leray projection onto divergence
free vector fields,
\[
 \mathcal{P} \ = \ I -\nabla\Delta^{-1}\nabla
\]
The last equation \eqref{MKG40} is a consequence of \eqref{MKG30} due
to the divergence free condition on the moments $\partial^\alpha
J_\alpha = 0$, which in turn follows from \eqref{MKG20}. However, it
plays a role in the sequel so it is added here for convenience.

 Here it is assumed that the data $A_j(0, \cdot)$ satisfy the
 vanishing divergence relation \eqref{eq:Coulomb}.  We remark that
 given an arbitrary finite energy data set for the MKG problem, one
 can find a gauge equivalent data set of comparable size which
 satisfies the Coulomb gauge condition.  This argument only involves
 solving linear elliptic pde's, and is omitted. The key
question now is to decide whether a family of data
\begin{align*}
  (A, \phi)[0]: = (A_{\alpha}(0, \cdot), \partial_t
  A_{\alpha}(0, \cdot), \phi(0, \cdot), \partial_t\phi(0, \cdot))
\end{align*}
which satisfy the compatibility conditions required by the two
equations for $A_0$ above can be extended to a global-in-time solution
for the Maxwell-Klein-Gordon system. The key for deciding this
question is the criticality character of the system. Note that the
{\it{energy}}
\[
E(A,\phi): = \int_{\R^4}\big(\frac{1}{4}\sum_{\alpha,\beta}F_{\alpha\beta}^2 +
\frac{1}{2}\sum_{\alpha}|D_{\alpha}\phi|^2\big)\,dx
\]
is preserved under the flow \eqref{eq:MKGgeneral}, and in our
$4+1$-dimensional setting it is also invariant under the {\it{natural
    scaling}}
\[
\phi(t, x) \rightarrow \lambda \phi(\lambda t, \lambda
x),\,A_{\alpha}(t, x)\rightarrow \lambda A(\lambda t, \lambda x)
\]
This means the $4+1$-MKG system is {\it{energy critical}}, and in
recent years a general approach to the large data Cauchy problem
associated with energy critical wave equations has emerged. The first
key step in this approach consists in establishing an essentially
optimal global well-posedness result for data which are small in the
energy norm, which is usually the optimal small-data global
well-posedness result achievable. In this paper, we set out to prove
this for the $4+1$-dimensional Maxwell-Klein-Gordon system.

\begin{theorem}\label{thm:Main}
a) Let $({A}, \phi)[0]$ be a   $C^\infty$ Coulomb data set satisfying
  \begin{equation}
  E({A}, \phi)<\epsilon_{*}
 \label{small-energy} \end{equation}
  for a sufficiently small universal constant $\epsilon_*>0$. Then the
  system \eqref{eq:MKGgeneral} admits a unique global smooth solution
  $({A}, \phi)$ on $\R^{4+1}$ with these data. 

  b) In addition, the data to solution operator extends continuously\footnote{here the continuity is locally in time}
on the set \eqref{small-energy}  to a map 
\[
\dot H^1(\R^4) \times L^2(\R^4) \ni ({A}, \phi)[0] \to ({A}, \phi)
\in C(\R; \dot H^1(\R^4)) \cap \dot C^1(\R, L^2(\R^4))
\]
\end{theorem}
We remark that the same result holds in all higher dimensions for small 
data in the scale invariant space $\dot H^{\frac{n}2-1} \times \dot H^{\frac{n}2-2}$.
This has already been known in dimensions $n\geq 6$, see \cite{RT_MKG}.
We have chosen to restrict our exposition to the more difficult case
$n = 4$ in order to keep the notations simple, but our analysis easily carries 
over to dimension $n =5$. On the other hand, we do not know whether
a similar result holds in dimension $n = 3$.

Before explaining some more details of our approach, we recall here
earlier developments on this problem, and how our approach relates to
these. Considering the case of general spatial dimension $n$ and
denoting the {\it{critical Sobolev exponent}} by $s_c = \frac{n}{2}-1$
(thus $s_c = 1$ for $n=4$, corresponding to the energy), a global
regularity result for data which are smooth and small in
$\dot{H}^{s_c}$ was established in dimensions $n\geq 6$ in the work
\cite{RT_MKG} which served as inspiration to the present work. We note
in passing that a result analogous to \cite{RT_MKG} was established in
\cite{KrSte} for the Yang-Mills system in dimensions $n\geq 6$, and
the present work most likely also admits a corresponding analogue for
the $4+1$-dimensional Yang-Mills problem.  The global regularity
question for the physically relevant $n=3$ case of the Yang-Mills
problem had been established earlier in the groundbreaking work
\cite{EarMon}. Observe that this problem is {\it{energy
    sub-critical}}.

In the context of MKG, the result \cite{RT_MKG} had been preceded by a
number of works which aimed at improving the {\it{local
    wellposedness}} of MKG in the $n = 3$ case, beginning with
\cite{KM0}, followed by \cite{Cu}, and more recently \cite{MaSte}; the
latter in particular established an essentially optimal local
well-posedness result by exploiting a subtle cancellation feature of
MKG, which also plays a role in the present work.
We also mention the recent result \cite{KeRoTao} which establishes
global regularity for energy sub-critical data in the $n=3$ case, in
the spirit of earlier work by Bourgain \cite{Bo}.

In the higher dimensional case $n\geq 4$, an essentially optimal local
well-posedness result for a model problem closely related to MKG was
obtained in \cite{KlainTat}. This model problem does not display the
crucial cancellation feature of the precise MKG-system which enable us
here to go all the way to the critical exponent and global regularity.
We mention also that essentially optimal local well-posedness for the exact MKG-system was obtained in \cite{Se}. 
Finally, the recent work \cite{Ste} established global regularity of
equations of MKG-type with data small in a weighted but scaling
invariant Besov type space, in the case $n = 4$.

The present paper follows a similar strategy as \cite{RT_MKG} : one
observes that the spatial Gauge connection components $A_{j}$,
$j=1,2,3$, which are governed by the first equation \eqref{MKG10}, may
in fact be decomposed into a free wave part and an inhomogeneous term
(the second term in the Duhamel formula for $A$) which in fact obeys a
better $l^1$-Besov type bound (while energy corresponds to $l^2$)
\[
A_j = A^{\text{free}}_j + A^{\text{nonlin}}_j
\]
This is important for handling the key difficulty of the MKG-system ,
which is the equation for $\phi$, i. e. the second equation
\eqref{MKG20}. In fact, we shall verify that the contribution of the
term $ A^{\text{nonlin}}$ to the difficult magnetic interaction term
$2iA_j\partial_j\phi$ {\it{in the low-high frequency interaction
    case}} can be suitably bounded when combined with the term
$2iA_0\partial_t\phi$, an observation coming from \cite{MaSte}.
However, the contribution of the free term $A^{\text{free}} $ to the
magnetic interaction term is nonperturbative and cannot be handled in
this manner. Thus, following the example set in \cite{RT_MKG}, we
retain this term into the covariant wave operator\footnote{Here we set
  $A^{\text{free}}_0=0$.}  $\Box_{A^{\text{free}}}$.  More precisely,
we shall define a suitable paradifferential wave operator $\Box_A^p$
which incorporates the 'leading part' of $\Box_{A^{\text{free}}}$
while relegating the rest to the source terms on the right.

 The key novelty of this paper then is the development of a functional
 calculus, involving in particular $X^{s,b}$-type as well as atomic
 null-frame spaces developed in other contexts, for solutions of the
 general inhomogeneous 'covariant' wave equation
\[
\Box_A^p = f
\]
This refined functional calculus is necessary to control the nonlinear
interaction terms, which become significantly more delicate in the
critical dimension than in the setting studied in \cite{RT_MKG}. In
particular, the Strichartz norms themselves appear far from sufficient
to handle the present situation.  The above covariant wave equation
will be solved by means of a suitable approximate parametrix, and we
show that this parametrix satisfies many of the same bounds as the
usual free wave parametrix, in particular encompassing refined square
sum type microlocalized Strichartz norms as well as null-frame
spaces. We expect the calculus developed here to be of fundamental
importance in other contexts, such as the regularity question of the
critical Yang-Mills system and related problems from mathematical
physics.

\section{Technical preliminaries}

Throughout the sequel we shall rely on Littlewood-Paley calculus, both
in space as well as space-time. In particular, we constantly invoke
the standard Littlewood-Paley localizers $P_k$, $k\in \mathbf{Z}$,
which are defined by
\[
\widehat{P_k f} = \chi(\frac{|\xi|}{2^k})\hat{f}(\xi)
\]
for functions $f$ defined on $\R^4$. Here $\chi$ is a smooth bump
function, supported on $[\frac{1}{4},4]$, which satisfies the key
condition $\sum_{k\in \mathbf{Z}}\chi(\frac{\xi}{2^k}) = 1$ if
$\xi>0$. To measure proximity of the space-time Fourier support to the
light cone, we use the concept of {\it{modulation}}. Thus we introduce
the multipliers $Q_j$, $j\in \mathbf{Z}$, via
\[
\widehat{Q_j f}(\tau, \xi) = \chi(\frac{\big||\tau| -
  |\xi|\big|}{2^j})\hat{f}(\tau, \xi)
\]
with the same $\chi$ as before, where $\hat{}$ in this context denotes
the space-time Fourier transform. We then refer to $2^j$ as the
modulation of the function.  On occasion we shall also use multipliers $S_l$, which restrict the {\it{space-time frequency}} to size $\sim 2^l$. 
These multipliers allow us to introduce a
variety of norms. In particular, for any $p\in [1,\infty)$, we set for
any norm $\|\cdot\|_{S}$
\[
\big\| F\big\|_{l^p S}: = \big(\sum_{k\in \mathbf{Z}}\big\|P_k
F\big\|_{S}^p\big)^{\frac{1}{p}}
\]
We also have the following $X^{s,b}$-type norms, applied to functions
localized to spatial frequency $\sim 2^k$:
\[
\big\|F\|_{X_p^{s,r}}: = 2^{sk}\big(\sum_{j\in
  \mathbf{Z}}\big[2^{rj}\big\|Q_j
F\big\|_{L_{t,x}^2}\big]^p\big)^{\frac{1}{p}},\,p\in [1,\infty),
\]
with the obvious analogue
\[
\big\|F\|_{X_\infty^{s,r}}: = 2^{sk}\sup_{k\in
  \mathbf{Z}}2^{rj}\big\|Q_j F\big\|_{L_{t,x}^2}
\]
For more refined norms, we shall also have to use multipliers
$P^{\omega}_{l}$, which localize the homogeneous variable
$\frac{\xi}{|\xi|}$ to caps $\omega\subset S^3$ of diameter $2^l$, by
means of smooth cutoffs. In these situations, we shall assume that for
each $l$ a uniformly (in $l$) finitely overlapping covering of $S^3$
by caps $\omega$ has been chosen with appropriate cutoffs subordinate
to these caps. Similar comments apply to the multipliers
$P_{\mathcal{C}_{k'}^{l'}}$ which localize to rectangular boxes and
will be defined below.\\
Given a norm $\|\cdot\|_{S}$ with corresponding space $S$, we denote by $S_k$ the space of functions in $S$ which are localized to frequency $\sim 2^k$. Furthermore, we denote by 
\[
S_{k,\pm}
\]
the subspace of functions in $S_k$ with Fourier support in the half-space $\tau><0$, with $\tau$ the Fourier variable dual to $t$. 

\section{Function Spaces}

There are three function spaces we work with: $N$, $N^*$, and
$S$. These are set up to that their dyadic subspaces $N_k$, $N^*_k$, and
$S_k$ satisfy the following relations:
\begin{equation}
  N_k \ = \ L^1(L^2) +  X_1^{0,-\frac{1}{2}} \ , 
  \qquad X_1^{0,\frac{1}{2}}
  \subseteq S_k\subseteq N^*_k \ , \label{basic_norms}
\end{equation}
Then define:
\begin{equation}
  \lp{F}{N}^2 \ = \ \sum_k \lp{P_k F}{N_k}^2 \ . \notag
\end{equation}

We also define $S_k^\sharp$ by
\[
\| u \|_{S_k^\sharp} = \| \Box u\|_{N_k} + \|\nabla u\|_{L^\infty L^2}
\]
On occasion we need to separate the two characteristic cones
$\{ \tau = \pm |\xi|\}$. Thus we
define
\[
N_{k,\pm}, \qquad N_k = N_{k,+} \cap  N_{k,-}
\]
\[
S_{k,\pm}^\sharp, \qquad S_k^\sharp = S_{k,+}^\sharp + S_{k,-}^\sharp
\]
\[
N^*_{k,\pm}, \qquad N^*_k = N^*_{k,+} + N_{k,-}^*
\]

Our space $S_k$ scales like $L^2$ free waves, and is defined by:
\begin{equation}
  \lp{\phi}{S_k}^2 \ = \ \lp{\phi}{S^{str}_k}^2 + \lp{\phi}{S^{ang}_k}^2
  +  \lp{\phi}{X_\infty^{0,\frac{1}{2}}}^2  \ , \notag
\end{equation}
where:
\begin{equation}
  S^{str}_k \ = \ \cap_{\frac{1}{q}+ \frac{3/2}{r}\leqslant \frac{3}{4}} 
  2^{(\frac{1}{q}+\frac{4}{r}-2)k}L^q(L^r) \ , \qquad
  \lp{\phi}{S^{ang}_k}^2 \ = \ 
  \sup_{l<0} \sum_{\omega}\lp{P^\omega_l
    Q_{<k+2l}\phi}{S_k^\omega(l)}^2 \ , \label{str_and_defn}
\end{equation}
The angular sector norms $S_k^\omega(l)$ are essentially the same as
in the wave maps context, see e.g.
\cite{T2}, and defined shortly.  Our space of solutions scales like
free waves with $\dot{H}^1$ data, with a high modulational gain
as in \cite{StTat}. Thus  we set:
\begin{equation}
  \lp{\phi}{S^1}^2 \ = \ \sum_k \lp{\nabla_{t,x}P_k\phi }{S_k}^2 
  + \lp{\Box\phi}{\ell^1 L^2(\dot{H}^{-\frac{1}{2}})}^2 
  \ . \label{S1-def}
\end{equation}
For later reference, we shall also use the norms
\[
\|\phi\|_{S^N}: = \|\nabla_{t,x}^{N-1}\phi\|_{S^1},\,N\geq 2
\]
Returning to $S_k^\omega(l)$, we define these as usual except that we
need to add additional square information over smaller radially
directed blocks $\mathcal{C}_{k'}(l')$ dimensions
$2^{k'}\times(2^{k'+l'})^3$ with appropriate dyadic gains.  First
define:
\begin{align}
  \lp{\phi}{P\!W^\pm_\omega(l)} \ &=\ \inf_{\phi=\int \!\!
    \phi^{\omega'} } \int_{|\omega-\omega'|\leqslant 2^{l}}
  \lp{\phi^{\omega'} }{L^2_{\pm\omega' }(L^\infty_{(\pm\omega')^\perp}
    )} d\omega' \ , \notag\\
  \lp{\phi}{N\!E} \ &= \ \sup_\omega \lp{\snabla_\omega
    \phi}{L^\infty_{ \omega} (L^2_{\omega^\perp})} \ , \notag
\end{align}
where the norms are with respect to $\ell_\omega^\pm = t\pm
\omega\cdot x$ and the transverse variable, while $\snabla_\omega$
denotes spatial differentiation in the $(\ell^+_\omega)^\perp$ plane.
Now set:
\begin{multline}
  \lp{\phi}{S_k^\omega(l)}^2 \ = \ \lp{ \phi}{S_k^{str}}^2 +
  2^{-2k}\lp{\phi}{N\!E}^2 + 2^{-3k}\sum_\pm
  \lp{Q^\pm  \phi}{P\!W^\mp_\omega(l) }^2 \\
  + \sup_{\substack{k'\leqslant k ,   l'\leqslant 0\\
      k+2l\leqslant k'+l'\leqslant k+l }} \sum_{\mathcal{C}_{k'}(l') }
  \Big( \lp{P_{\mathcal{C}_{k'}(l')} \phi}{S_k^{str}}^2
  + 2^{-2k}\lp{P_{\mathcal{C}_{k'}(l')} \phi}{N\!E}^2\\
  + 2^{-2k'-k}\lp{P_{\mathcal{C}_{k'}(l')} \phi}{L^2(L^\infty)}^2 +
  2^{-3(k'+l')}\sum_\pm \lp{Q^\pm P_{\mathcal{C}_{k'}(l')}
    \phi}{P\!W^\mp_\omega(l) }^2 \Big) \ . \label{Sl_def}
\end{multline}
We remark that an important feature of these norms is that the
time-like oriented $L^2(L^\infty)$ block norm gains dyadically from
the length of $\mathcal{C}_{k'}(l')$ in the radial direction, while
the null oriented $L^2_\omega(L^\infty_{\omega^\perp})$ norms gain
from the size of $\mathcal{C}_{k'}(l')$ in the angular direction. We
also remark that a useful feature of this setup is:
\begin{equation}
  \big(\sum_{\mathcal{C}_{k'}} \lp{P_{\mathcal{C}_{k'}} P_k \phi }{L^2(L^\infty)}^2\big)^\frac{1}{2}
  \ \lesssim \ 
  2^{k'} 2^{\frac{1}{2}k} \lp{P_k \phi}{S^{ang}_k\cap X_\infty^{0,\frac{1}{2}}} \ , \label{L2Linfty_sqsum}
\end{equation}
where $\mathcal{C}_{k'}$ are a finitely overlapping set of cubes of
side length $2^{k'}$.  This follows by splitting
$P_k\phi=Q_{<k'}P_k\phi+Q_{\geqslant k'}P_k\phi$ and using the
$S_k^\omega\big(\frac{1}{2}(k'-k)\big)$ norm for the first term and
the $X_{\infty}^{0,\frac{1}{2}}$ norm and Bernstein's inequality for
the second.

Next we describe some auxiliary spaces of $L^1(L^\infty)$ which
will be useful for decomposing the non-linearity. The first is for the
hyperbolic part of the solution:
\begin{equation}
  \lp{\phi}{Z^{hyp}} \ =\  \sum_k  \lp{P_k\phi }{Z_k^{hyp}} \ , \ \	
  \lp{\phi}{Z_k^{hyp}}^2 \ =\  \sup_{l<C} 
  \sum_{\omega }2^{l}\lp{P^\omega_lQ_{k+2l} \phi }{L^1(L^\infty)}^2
  \ . \notag
\end{equation}
Note that as defined this space already scales like $\dot{H}^1$ free
waves. In addition, note the following useful embedding which is a direct
consequence of Bernstein's inequality:
\begin{equation}
  \Box^{-1} \ell^1 L^1(L^2) \ \subseteq \ Z^{hyp} \ . \label{B_embed}
\end{equation}
The second is for the elliptic part of the solution:
\begin{equation}
  \lp{A}{Z^{ell}_k} \ = \ \sum_{ j<k+C} \lp{Q_j A}{L^1(L^\infty)}
  \ , \qquad \lp{A}{Z^{ell}} \ = \  \sum_k \lp{P_k A}{Z^{ell}_k} \ . \notag
\end{equation}
We remark that the only purpose of this last norm is to handle sums
over $Q_j$ in product estimates involving $A$ in $L^1(L^\infty)$.

Finally, the function spaces for $A_0$ are much easier to describe 
as the $A_0$ equation is elliptic:
\begin{equation}
\|A_0\|_{Y^1}^2 = \| \nabla_{x,t} A_0\|_{L^\infty L^2}^2 + \| A_0\|_{L^2_t 
\dot H^{\frac32}_x}^2 + \| \partial_t A_0\|_{ L^2_t 
\dot H^{\frac12}_x}^2  \ . \notag
\end{equation}
We also have the derivative norms
\[
\|A_0\|_{Y^N}: = \|\nabla_{t,x}^{N-1}A_0\|_{Y^1},\,N\geq 2
\]

% -------------------------------------------------------------------------
%%%%%%%%%%%%%%%%%%%%%%%%%%%%%
% -------------------------------------------------------------------------

\section{Decomposing the non-linearity;
 Statement and use of the core 
multilinear estimates}

Recalling the definition of the currents $J_\alpha = - \Im
(\phi\overline{D_{\alpha} \phi})$ we write the MKG-CG system again
here as:
\begin{subequations}
  \begin{align}
    \Box A_i \ &= \ \mathcal{P}_i J_x \ ,
  \label{MKG1}\\
    \Box_A \phi \ &= \ 0  \ , \label{MKG2}\\
    \Delta A_0 \ &= \  J_0 \ , \label{MKG3}
  \end{align}
\end{subequations}
This system will be solved iteratively using the following scheme.
We initialize 
\[
A_i^{(1)} = A_i^{\text{free}},  \qquad 
A_0^{(1)} = 0, \qquad \phi^{(1)} = \phi^{\text{free}}, 
\]
where $A_i^{\text{free}}$ and $\phi^{\text{free}}$ solve the flat wave 
equation with initial data $A_i[0]$, respectively $\phi[0]$. 
Given $A_\alpha^{(m)}$, $\phi_\alpha^{(m)}$ and their 
associated currents $J_\alpha^{(m)}$, we define the  next iteration
via the equations
  \begin{subequations}
  \begin{align}
    \Box A_i^{(m+1)}  \ &= \ \mathcal{P}_i J_x^{(m)} \ ,
  \label{MKG1-it}\\ 
  \Box_{A^{(m)}} \phi^{(m+1)} \ &= \ 0  \ , \label{MKG2-it} \\
 \Delta A_0^{(m+1)} \ &= \  J_0^{(m)} \ , \label{MKG3-it} 
  \end{align}
\end{subequations}
with the same initial data $A_i[0]$, respectively $\phi[0]$.  Assuming
small energy for the initial data $A_i[0]$ and  $\phi[0]$ 
as in \eqref{small-energy}, we 
will prove that this Picard type iteration converges in the space
$S^1$.  Our starting point is the linear bound
\begin{equation}\label{picard:1}
\| A_x^{(1)}\|_{S^1} + \|\phi^{(1)}\|_{S^1} \leq C_0 \epsilon_* 
\end{equation}
Then we will inductively establish the bound
\begin{equation} \label{picard}
\| A^{(m+1)}_x -A^{(m)}_x \|_{l^1 S^1} + \|\phi^{(m+1)} -\phi^{(m)} \|_{S^1}
+ \|A_0 ^{(m+1)} -A^{(m)}_0 \|_{Y^1}
\leq (C \epsilon_*)^m 
\end{equation}
for a universal constant $C \geq 2 C_0$.  

Assuming this holds, passing to the limit as $n \to \infty$ we obtain
a Coulomb solution $(A,\phi)$ to the MKG equation which satisfies
the bound 
\begin{equation} 
\| A^{\text{nonlin}}_x  \|_{l^1 S^1} + \|\phi \|_{S^1}
+ \|A_0  \|_{Y^1}
\lesssim \epsilon_* 
\end{equation}
The same argument proves uniqueness. If two solutions
$(A^{(0)},\phi^{(0)})$ and $(A^{(1)},\phi^{(1)})$
have the same Cauchy data for $A_x$, then the
same $A_x^{\text{free}}$ is used for both.  Thus applying the same
series of estimates in this section to the difference of the
two exact solutions rather than two approximate solutions
we obtain the bound
\begin{equation} 
\| A^{(0)}_x- A^{(1)}_x \|_{l^1 S^1} + \|\phi^{(0)}-\phi^{(1)} \|_{S^1}
+ \|A_0^{(0)}-A_0^{(1)}\|_{Y^1} \lesssim \|  \phi^{(0)}[0]-\phi^{(1)}[0]\|_{\dot H^1 \times L^2} 
\end{equation}
This bound proves both uniqueness and Lipschitz dependence 
of the solution with respect to $\phi[0]$. The continuous dependence 
with respect to $A_x[0]$ is a more delicate issue and will be explained in section~\ref{Sec:reg}.

The estimate \eqref{picard}$(m)$ will follow from \eqref{picard}$(<m)$. Summing up 
\eqref{picard}$(<m)$ and \eqref{picard:1}, we can easily add to our
induction hypothesis the bound
\begin{equation} \label{picard-sum}
\| A^{(n)}_x - A^{\text{free}}_x \|_{l^1 S^1} + \|\phi^{(n)} \|_{S^1}
+ \|A_0 ^{(n)}\|_{Y^1}
\leq  2C_0 \epsilon_*, \qquad n \leq m 
\end{equation}
provided that $\epsilon_*$ is small enough.

A simple but very useful observation is that, since all iterated
$\phi^{(m)}$ solve covariant wave equations, it follows that the
associated moments are divergence free, $D^\alpha J_\alpha^{(m)} = 0$.
Then, differentiating the equation \eqref{MKG3-it}, we obtain 
\begin{equation} \label{MKG4-it} 
 \Delta \partial_t A_0^{(m+1)} \ = \  \partial^i J_i^{(m)} 
\end{equation}
which will be used to estimate the high modulations of $A_0$.
This is the reason why we have completely 
avoided using $\phi^{(m)}$ in \eqref{MKG2-it}, even though 
some of the terms in there will be treated perturbatively.

A second fact to keep in mind is that while in terms of 
formulas this is a one step iteration, in terms of estimates this 
is really a two step iteration. Precisely, in order to obtain 
good bounds for $\phi^{(m+1)}$ we will need to reiterate 
``bad'' portions of $A_\alpha$ in terms of $\phi^{(m-1)}$ (but not  $A^{(m-1)}$).
All this is explained in detail below.

To decompose the non-linearity, the following device will be useful.
If $\mathcal{M}(D_{t,x},D_{t,y})$ is any bilinear operator we set:
\begin{align}
  \mathcal{H}_k \mathcal{M}(\phi,\psi) \ &= \ \sum_{j<k+C}
  Q_{j} P_k \mathcal{M}(Q_{<j-C}\phi, Q_{<j-C}\psi ) \ , \notag\\
  \mathcal{H}^*_k \mathcal{M}(\phi,\psi) \ &= \ \sum_{j<k+C} Q_{<j-C}
  \mathcal{M}(Q_{j}P_k \phi, Q_{<j-C}\psi ) \ , \notag
\end{align}
Now we describe the function spaces and multilinear estimates for each
iterative piece of the non-linearity.  In this Section we only state
the main estimates, and reduce them to appropriate 
dyadic bounds. These will be  proved later in the paper.

% -------------------------------------------------------------------------

\subsection{Estimates for the $A_i$}

We split the spatial potentials  into homogeneous and inhomogeneous
parts with respect to $t=0$ as follows, $A_i^{(m)}=A_i^{\text{free}}+
A_i^{\text{nonlin},(m)}$ where the second part solves the linear equation
\begin{equation}\label{am-nonlin}
\Box A_i^{\text{nonlin},(m+1)} = - \mathcal{P}_i \big(\phi^{(m)}\nabla_x\phi^{(m)}+
  |\phi^{(m)}|^2 A^{(m)}_x \big), \qquad  A_i^{nonlin}[0] \ = \ 0 \ 
\end{equation}
In order to establish the $l^1S^1$ bound in the first term of
\eqref{picard} it suffices to work directly with this equation.
However, $A^{(m)}$ also appears in the $\phi^{(m+1)}$ equation
\eqref{MKG2-it}, and in order to be able to treat its contribution
perturbatively there we also need to control its $Z^{hyp}$ norm (see
the subsection below devoted to $\phi$). This works for the most part,
but there is a part of $A_i^{\text{nonlin},(m)}$ which requires a more
delicate treatment, which is:
\begin{equation}
  \mathcal{H} A_i^{\text{nonlin},(m)} \ =\ -\sum_{\substack{k,k_i:\\
      k<\min\{k_1,k_2\}-C}} \Box^{-1} \mathcal{H}_k
  \mathcal{P}_i  (\phi^{(m-1)}_{k_1}{\nabla_x \phi^{(m-1)}_{k_2}}) \ . \label{HA_def} \end{equation}
For the good part of  $A_i^{\text{nonlin},(m)}$ we will establish the
difference bound
\begin{equation}\label{good-aj-diff}
\| (A_i^{\text{nonlin},(n)} -  \mathcal{H} A_i^{\text{nonlin},(n)} ) 
-  (A_i^{\text{nonlin},(n-1)} -  \mathcal{H} A_i^{\text{nonlin},(n-1)} )\|_{Z^{hyp}}  \leq
(C \epsilon_*)^{n-1}, \qquad 3 \leq n \leq m+1
\end{equation}
as well as the summed estimate
\begin{equation}\label{good-aj}
\| A_i^{\text{nonlin},(n)} -  \mathcal{H} A_i^{\text{nonlin},(n)}  \|_{Z^{hyp}}  \lesssim
 C_0^2 \epsilon_*^2, \qquad 2 \leq n \leq m+1
\end{equation}
On the other hand the bad part $\mathcal{H} A_i^{\text{nonlin},(m)}$ will be 
considered later in combination with the similar part of $A_0$.

As seen above,  we need estimates not only for the equation
\eqref{am-nonlin}, but also for differences of two consecutive
iterations. Thus we are led to consider a more general equation of the
form
\begin{equation}
  \Box B_i \ = \ \mathcal{P}_i \big(\phi^{(1)}\nabla_x\phi^{(2)}+
  \phi^{(3)}\phi^{(4)} A^{(1)}_i \big) \ , 
  \quad B_i^{nonlin}[0] \ = \ 0 \ . \label{Ai_it1}
  % \\ \Box A_i^{cubic} \ &= \ \mathcal{P}_i (\phi^{(3)}\phi^{(4)}
  % A^{(1)}) \ , \quad &A_i^{cubic}[0] \ &= \ 0 \ . \label{Ai_it2}
\end{equation}
and its corresponding bad part,
\begin{equation}
  \mathcal{H} B_i \ =\ -\sum_{\substack{k,k_i:\\
      k<\min\{k_1,k_2\}-C}} \Box^{-1} \mathcal{H}_k
  \mathcal{P}_i  (\phi^{(1)}_{k_1}{\nabla_x \phi^{(2)}_{k_2}}) \ . \label{HAa_def}
\end{equation}

Then we will show:
\begin{prop}
  One has the following iterative estimates for the functions 
$B_i$ and $  \mathcal{H} B_i $ defined above:
  \begin{align}
    \lp{B_i}{\ell^1 S^1} \ &\lesssim \
    \lp{\phi^{(1)}}{S^1}\lp{\phi^{(2)}}{S^1} + \lp{A^{(1)}_i}{S^1}
    \lp{\phi^{(3)}}{S^1}\lp{\phi^{(4)}}{S^1}
    \ , \label{AiS_est1}\\
    % \lp{A_i^{cubic}}{\ell^1 S^1} \ &\lesssim \
    % \lp{A^{(1)}}{S^1} \lp{\phi^{(3)}}{S^1}\lp{\phi^{(4)}}{S^1} \ , \label{AiS_est2}\\
    \lp{B_i-\mathcal{H} B_i}{Z^{hyp}} \ &\lesssim \
    \lp{\phi^{(1)}}{S^1}\lp{\phi^{(2)}}{S^1} \ . \label{AiB_est1}
    % \lp{A_i^{cubic}}{B} \ &\lesssim \
    % \lp{A^{(1)}}{S^1}\lp{\phi^{(3)}}{S^1}\lp{\phi^{(4)}}{S^1} \
    % , \label{AiB_est2}
  \end{align}
\end{prop}
In combination with our induction hypothesis, the estimate
\eqref{AiS_est1} leads to the $l^1S^1$ bound for the first term in
\eqref{picard}. Similarly, the estimate \eqref{AiB_est1} yields the $Z^{hyp}$
bounds for the good part of $A_i^{\text{nonlin},(m)}$ in \eqref{good-aj-diff}
and \eqref{good-aj}.

\begin{proof}[Decomposition of the proof of estimates \eqref{AiS_est1}
  and \eqref{AiB_est1}]
  For estimate \eqref{AiS_est1} we begin with the high modulational
  bounds:
  \begin{align}
    \lp{\phi^{(1)}\nabla_{t,x}\phi^{(2)}}{\ell^1L^2(\dot{H}^{-\frac{1}{2}})}
    \ &\lesssim \ \lp{\phi^{(1)}}{S^1}\lp{\phi^{(2)}}{S^1} \ , \label{hi_mod1}\\
    \lp{\phi^{(3)}\phi^{(4)}\phi^{(5)}
    }{\ell^1L^2(\dot{H}^{-\frac{1}{2}})} \ &\lesssim \
    \lp{\phi^{(3)}}{S^1}\lp{\phi^{(4)}}{S^1}\lp{\phi^{(5)}}{S^1} \
    . \label{hi_mod2}
  \end{align}

  Next, for the cubic terms in $A_i^{nonlin}$ both \eqref{AiS_est1}
  and \eqref{AiB_est1} follow once we can show:
  \begin{equation}
    \lp{\phi^{(3)}\phi^{(4)}\phi^{(5)} }{\ell^1L^1(L^2) }
    \ \lesssim \ \lp{\phi^{(3)}}{S^1}\lp{\phi^{(4)}}{S^1}\lp{\phi^{(5)}}{S^1} \ . \label{cubic_str}
  \end{equation}

  Returning to the quadratic part, we employ the notation:
  \begin{equation}
    \mathcal{N}_{ij}(\phi^{(1)},\phi^{(2)}) \ =\  
    \partial_i\phi^{(1)}\partial_j\phi^{(2)} - \partial_j\phi^{(1)}\partial_i\phi^{(2)}
    \ , \notag
  \end{equation}
  which allows us to write:
  \begin{equation}
    \mathcal{P}_j (\phi^{(1)} \nabla_x \phi^{(2)}) \ = \ 
    \Delta^{-1}\nabla^i \mathcal{N}_{ij}(\phi^{(1)},\phi^{(2)})
    \ . \label{A_null_formula}
  \end{equation}
  Then for the quadratic part of \eqref{AiS_est1} it suffices to show:
  \begin{equation}
    \lp{\nabla_x \Delta^{-1} \mathcal{N}(\phi^{(1)},\phi^{(2)})}{\ell^1 N} \ \lesssim \ 
    \lp{\phi^{(1)}}{S^1}\lp{\phi^{(2)}}{S^1} \ , \label{Ai_null_est}
  \end{equation}
  where $\mathcal{N}$ denotes any instance of the $\mathcal{N}_{ij}$.
  Finally, making an analogous definition to \eqref{HAa_def} for each
  term in \eqref{A_null_formula} we need to show:
  \begin{equation}
    \lp{(I-\mathcal{H})\nabla_x \Delta^{-1} \mathcal{N}(\phi^{(1)},\phi^{(2)})}{\Box Z^{hyp}}
    \ \lesssim \ 
    \lp{\phi^{(1)}}{S^1}\lp{\phi^{(2)}}{S^1} \ . \label{Ai_null_Best}
  \end{equation}

\end{proof}

% -------------------------------------------------------------------------

\subsection{Estimates for $A_0$ and $\partial_t A_0$}
The estimates for temporal potentials are analogous to those
above, but using instead the equations \eqref{MKG3-it} and \eqref{MKG4-it}.
In expanded form these are written as 
\[
\begin{split}
 \Delta A_0^{(m+1)} \ &=  - \Im(\phi^{(m)} \partial_t \bar \phi^{(m)}) - A_0^{(m)} |\phi^{(m)}|^2
\\
 \Delta \partial_t A_0^{(m+1)} \ &=  - \partial^i \Im(\phi^{(m)} \partial_i \bar \phi^{(m)}) - 
\partial^i (A_i^{(m)} |\phi^{(m)}|^2)
\end{split}
\]
In a first approximation, from $A_0^{(m)}$ we isolate to output 
of bilinear $high \times high \to low$ interactions, namely
\begin{equation}
  A_0^{(m+1),hh} \ = \ \sum_{\substack{k,k_i:\\
      k<\min\{k_1,k_2\}-C}} \Delta^{-1}  
P_k(\phi_{k_1}^{(m)} \partial_t \phi_{k_2}^{(m)}) 
\ . \notag
\end{equation}
To bring the analysis to the same level as in the case of the $A_j$'s, 
in analogy with \eqref{HA_def}, we also define the component $ \mathcal{H} A_0^{(m+1)}$ of 
$ A_0^{(m+1),hh}$, 
\begin{equation}
  \mathcal{H} A_0^{(m+1)} \ =\ -\sum_{\substack{k,k_i:\\
      k<\min\{k_1,k_2\}-C}} \Delta^{-1} \mathcal{H}_k
  (\phi^{(m)}_{k_1}{\partial_t \phi^{(m)}_{k_2}}) \ . \label{HA00_def}
\end{equation}
In addition to the estimates included in \eqref{picard}, we also need 
the some similar  bounds for $A^{(m+1),hh)}$ as well as bounds 
for the good differences for $3 \leq n \leq m+1$:
\begin{equation} \label{A0parts}
\begin{split}
 \| A_0^{(n),hh} - A_0^{(n-1),hh}\|_{\ell^1 L^2(\dot{H}^\frac{3}{2})}  \leq  & \ 
(C\epsilon)^{n-1}
\\
  \| (A_0^{(n)} - A_0^{(n),hh}) - (A_0^{(n-1)}- A_0^{(n-1),hh})\|_{\ell^1 L^1(L^\infty)}
  \leq &  \ 
(C\epsilon)^{n-1}
\\
  \| (  \mathcal{H} A_0^{(n)} - A_0^{(n),hh}) - (  \mathcal{H} A_0^{(n-1)}- A_0^{(n-1),hh})\|_{Z^{ell}}   \leq  & \ 
(C\epsilon)^{n-1}
\end{split}
\end{equation}

Passing to differences we arrive at a system of equations with multilinear
inhomogeneities of the form:
\begin{align}
  \Delta B_0 \ &= \ \phi^{(1)}\partial_t \phi^{(2)} + \phi^{(3)}\phi^{(4)} A^{(1)}_0 \ , \label{A0_it1}\\
  \Delta \partial_t B_0 \ &= \ \nabla_x \big(\phi^{(1)}\nabla_x
  \phi^{(2)} + \phi^{(3)} \phi^{(4)} A^{(1)}_i\big) \ . \label{A0_it2}
\end{align}
 As above  we write:
\begin{equation}
  B^{hh}_0 \ = \ \sum_{\substack{k,k_i:\\
      k<\min\{k_1,k_2\}-C}} P_k(\phi_{k_1}^{(1)} \partial_t \phi_{k_2}^{(2)}) \ . \notag
\end{equation}
respectively 
\begin{equation}
  \mathcal{H} B_0 \ =\ -\sum_{\substack{k,k_i:\\
      k<\min\{k_1,k_2\}-C}} \Delta^{-1} \mathcal{H}_k
  (\phi^{(1)}_{k_1}{\partial_t \phi^{(2)}_{k_2}}) \ . \label{HA0_def}
\end{equation}
Then we will show:
\begin{prop}
  One has the following  estimates for $B_0$ and $\partial_t
  B_0$ defined above:
  \begin{align}
    \lp{\nabla_{t,x}B_0 }{L^\infty(L^2)} \ &\lesssim \
    \lp{\phi^{(1)}}{S^1}\lp{\phi^{(2)}}{S^1} +
    \lp{A^{(1)}_\alpha}{L^\infty(\dot{H}^1)}
    \lp{\phi^{(3)}}{S^1}\lp{\phi^{(4)}}{S^1}
    \ , \label{A0_energy}\\
    \lp{(B_0,B_0^{hh})}{\ell^1 L^2(\dot{H}^\frac{3}{2})} \ &\lesssim \
    \lp{\phi^{(1)}}{S^1}\lp{\phi^{(2)}}{S^1} +
    \lp{A^{(1)}_0}{L^2(\dot{H}^\frac{3}{2})}
    \lp{\phi^{(3)}}{S^1}\lp{\phi^{(4)}}{S^1}
    \ , \label{A0_est1}\\
    \lp{\partial_t B_0}{\ell^1 L^2(\dot{H}^\frac{1}{2})} \ &\lesssim \
    \lp{\phi^{(1)}}{S^1}\lp{\phi^{(2)}}{S^1} + \lp{A^{(1)}_i}{S^1}
    \lp{\phi^{(3)}}{S^1}\lp{\phi^{(4)}}{S^1}
    \ , \label{A0_est2}\\
    \lp{B_0-B_0^{hh}}{\ell^1 L^1(L^\infty)} \ &\lesssim \
    \lp{\phi^{(1)}}{S^1}\lp{\phi^{(2)}}{S^1} +
    \lp{A^{(1)}_0}{L^2(\dot{H}^\frac{3}{2})}
    \lp{\phi^{(3)}}{S^1}\lp{\phi^{(4)}}{S^1}
    \ , \label{A0B_est1}\\
    \lp{B_0^{hh}-\mathcal{H} B_0}{Z^{ell}} \ &\lesssim \
    \lp{\phi^{(1)}}{S^1}\lp{\phi^{(2)}}{S^1} \ . \label{A0B_est2}
  \end{align}
\end{prop}
The bound for the last term in \eqref{picard} follows from 
\eqref{A0_energy}, \eqref{A0_est1} and \eqref{A0_est2}.
The three bounds in \eqref{A0parts} are consequences of
\eqref{A0_est1}, \eqref{A0B_est1} and \eqref{A0B_est2}.

\begin{proof}[Outline of proof of estimates \eqref{A0_est1}--\eqref{A0B_est1}]
  Estimate \eqref{A0_energy} follows from Sobolev's embedding, while
  the proofs of \eqref{A0_est1} and \eqref{A0_est2} either follow
  immediately from \eqref{hi_mod1} and \eqref{hi_mod2} above, or in
  the case of $A_0^{hh}$ from the estimates used to produce those
  bounds. Estimate \eqref{A0B_est1} follows from similar
  considerations. In detail, the estimates \eqref{A0_energy} - \eqref{A0B_est1} will be proved in subsection~\ref{subse:lp}.
 On the other hand \eqref{A0B_est2} is more
  microlocal in nature and follows from calculations similar to those
  in the proof of \eqref{Ai_null_Best} listed above. It will be proved as a consequence of Theorem~\ref{prod_thm} below. 
\end{proof}

% -------------------------------------------------------------------------

\subsection{Estimates for $\phi$}

We now turn to the heart of the matter, which is the covariant wave
equation \eqref{MKG4-it} for $\phi^{(m+1)}$. This cannot be viewed 
as an equation of type $\Box \phi = perturbative$, and therein lies
the difficulty. To address this issue we identify the nonperturbative 
part, and add it to the main operator $\Box$ to obtain a paradifferential
type magnetic d'Alembertian. 
Precisely, we define the leading order paradifferential approximation
to the covariant $\Box_A$ equation using only the free part
$A_i^{\text{free}}$ of $A_i^{(m)}$, which is independent of $n$:
\begin{equation}
  \Box_A^p \ = \ \Box - 2 i \sum_k P_{<k-C}A^{\text{free},j}P_k \partial_j \ , 
\label{para_wave}
\end{equation}
Fortunately, this operator does not depend on $m$.
Then the equation for $\phi^{(m+1)}$ takes the form
\begin{equation}\label{phieq-nonlin}
  \Box_A^p \phi^{(m+1)} = \mathcal M(A^{(m)}, \phi^{(m+1)})
\end{equation}
where $ \mathcal M(A^{(m)}, \phi^{(m+1)})$ is given by 
\begin{equation}\label{Mdef}
\begin{split}
\mathcal  M(A^{(m)}, \phi^{(m+1)})  = &   \ 2 i \big(
  A^{{(m)},\alpha}\partial_\alpha \phi^{(m+1)} - \sum_k
  P_{<k-C}A^{\text{free},j}\partial_j P_k \phi^{(m+1)}\big)
\\ & \ - \big( i \partial_t
  A_0^{(m)} \phi^{(m+1)}
  + A^{(m),\alpha} A_\alpha^{(m)} \phi^{(m+1)}\big)
\\ := & \  \mathcal  M^1(A^{(m)}, \phi^{(m+1)}) +  \mathcal  M^2(A^{(m)}, \phi^{(m+1)}) 
\end{split}
\end{equation}
We further write 
\[ 
\mathcal  M^1 =  \mathcal  N + \mathcal N_0
\]
where $\mathcal  N$ contains the terms of the form $A^j \partial_j \phi$ and 
$ \mathcal N_0$ contains the terms of the form $A^0 \partial_0 \phi$.
We  remark that  $ \mathcal  N$ exhibits a null structure.
Indeed, using the divergence free character of $A_i$ we write:
  \begin{equation}
    A_j \ = \ \nabla^i \Delta^{-1} F_{ij} \ , \qquad
    F_{ij} \ = \ \nabla_i A_j - \nabla_j A_i \ , \notag
  \end{equation}
  so that:
  \begin{equation}
 \mathcal  N (A_x,\phi) =    A^{i}\partial_i \phi \ = \ \frac{1}{2} \sum_{i<j}
    \mathcal{N}_{ij}(\nabla_i \Delta^{-1} A_j,\phi) \ . \notag
  \end{equation}
From here on we take the last expression as the definition of $\mathcal N$.
This will allow us to retain the null form while discarding the divergence
free condition on finer decompositions of $A_x$.

  Now the idea is to first produce a parametrix for $ \Box_A^p $, and
  then to estimate the right hand side of \eqref{phieq-nonlin}
  perturbatively. The linear estimate for the magnetic wave operator $
  \Box_A^p $ is one of the key points of the paper, and has the
  following form:

\begin{thm}[Linear estimates for $\phi$]\label{main_linear_thm}
  Let $\Box_A^p$ be the paradifferential gauge-covariant wave operator
  defined on line \eqref{para_wave}, and suppose that $\Box
  A^{free}=0$ with $\lp{A^{free}[0]}{\dot{H}^1\times L^2}\leqslant
  \epsilon$. If $\epsilon$ is sufficiently small then we have:
  \begin{equation}
    \lp{\phi}{S^1} \ \lesssim \ \lp{\phi[0]}{\dot{H}^1\times L^2} +
 \lp{\Box_A^p\phi}{N\cap
      l^1L^2(\dot{H}^{-\frac{1}{2}})} \ . \label{main_linear_est}
  \end{equation}
\end{thm}
Section ~\ref{s:para-lin} is devoted to the proof of this result.

\bigskip

In order to solve the 
equation \eqref{phieq-nonlin} it
 remains to estimate the the right hand side of
\eqref{phieq-nonlin}, 
\begin{equation}\label{mm}
 \| \mathcal  M(A^{(m)}, \phi^{(m+1)}) \|_{N \cap L^2(\dot H^{-\frac12})} 
\lesssim \epsilon_* \|  \phi^{(m+1)}\|_{S^1}
\end{equation}
which is applied with $\phi = \phi^{(m)}$.
In order to estimate the difference $\phi^{(m+1)}-\phi^{(m)}$ we need 
in addition to show that 
\begin{equation}\label{mm-diff}
 \| \mathcal  M(A^{(m)},\phi) -  \mathcal  M(A^{(m-1)}, \phi) 
\|_{N \cap L^2(\dot H^{-\frac12})} \lesssim  (C\epsilon)^{m-1} \|\phi\|_{S} 
\end{equation}
which is then applied to $\phi = \phi^{(m)}$. 
To prove \eqref{mm} and \eqref{mm-diff} we peel off some easier cases
before we arrive at the heart of the matter. 

 \step{1} {The $ \mathcal  M^2$ term.}
For this it suffices to have the following estimates:
  \begin{align}
    \lp{\partial_t A_0 \phi}{\ell^1
      L^2(\dot{H}^{-\frac{1}{2}})\cap L^1(L^2)} \ \lesssim \
    \lp{\partial_t A_0}{ L^2(\dot{H}^\frac{1}{2})\cap
      L^\infty(L^2)} \lp{\phi}{S^1}
    \ , \label{phi_cubic_est1}\\
    \lp{ A_0^{(1)} A_0^{(2)} \phi}{\ell^1
      L^2(\dot{H}^{-\frac{1}{2}})\cap L^1(L^2)} \ \lesssim \
    \prod_{i=1,2} \lp{A_0^{(i)}}{L^2(\dot{H}^\frac{3}{2})\cap
      L^\infty(\dot{H}^1)} \lp{\phi}{S^1} \
    . \label{phi_cubic_est2}
  \end{align}
which are proved using only global Strichartz type bounds
and Sobolev embeddings.

\step{2}{High modulation bounds  for $\mathcal M^1$} To establish the
$L^2(\dot{H}^{-\frac{1}{2}})$ bound for $\mathcal M^1$ we use
\eqref{hi_mod1} and the following two estimates:
  \begin{align}
    \sum_k \lp{P_{<k-C}A^{free}_i\, \nabla_x P_k
      \phi}{L^2(\dot{H}^{-\frac{1}{2}})}
    \ &\lesssim \ \lp{A^{free}_i}{S^1} \lp{\phi}{S^1} \ , \label{quadfree_phi_himod}\\
    \lp{A_0\partial_t \phi}{\ell^1
      L^2(\dot{H}^{-\frac{1}{2}})} \ &\lesssim \
    \lp{A_0}{L^2(\dot{H}^{\frac{3}{2}})}\lp{\phi}{S^1} \
    . \label{quad0_phi_himod}
  \end{align}
It remains to prove the $N$ bounds  for $\mathcal M^1$.
At first we separately consider $\mathcal N$ and  $\mathcal N_0$.

\step{3}{Peeling off the good parts of $\mathcal N$} The first step
will be to show that we can restrict our attention to $low \times
high$ interactions in $\mathcal N$.  For this we define the component 
$\mathcal{N}^{lowhi}$ by
\[
  \mathcal{N}^{lowhi}(A_x,\phi) = \ \sum_k
    \mathcal{N}( P_{<k-C}A_i, P_k \phi) \ ,
\]
To estimate  the difference $  \mathcal{N} -   \mathcal{N}^{lowhi}$
 we will  prove the null form estimate
 \begin{equation}
    \lp{ \mathcal{N}(A_x, \phi )
      -\mathcal{N}^{lowhi}(A_x, \phi ) }{N}
    \ \lesssim \ \lp{A_x}{S^1}\lp{\phi}{S^1} \ . \label{quadphi_HL_null}
  \end{equation}
  We remark that once this is done, the free part $A_x^{\text{free}}$
  has been taken care of and will no longer appear.

  For the low-high interactions in $\mathcal N$ we still have the null
  structure to use, but the balance of frequencies is
  unfavourable. For the most part we can still bound $\mathcal
  N^{lowhi}$ via a null form bilinear $S^1 \times S^1$
  estimate; the exception to this is the case of $high \times low \to low$ 
modulation interactions. We group these in an expression denoted 
by  $\mathcal H^* \mathcal N^{lowhi}$,  which is given by 
\begin{equation}
    \mathcal{H}^* \mathcal{N}^{lowhi}(A_j, \phi) \ = \ 
    \sum_{\substack{k,k':\\ k'<k-C}}\mathcal{H}_{k'}^*  
    \mathcal{N} (A_j, P_k\phi) \ , \notag
 \end{equation}
For the difference  we will prove the bilinear null form estimate
\begin{equation}
  \lp{\mathcal{N}^{lowhi}(A_x,\phi)-
\mathcal{H}^* \mathcal{N}^{lowhi}(A_x,\phi) }{N} \
  \lesssim \ \lp{A^{(1)}_i}{\ell^1 S^1}\lp{\phi^{(1)}}{S^1}
  \ , \label{gen_phi_lohi_quadest}
\end{equation}
in the last section of the paper.

The $high \times low \to low$ modulation interactions contained 
in $\mathcal{H}^* \mathcal{N}^{lowhi}$ can no longer be estimated 
using the $S^1$ norm of $A_x$, instead we need the stronger 
$Z^{hyp}$ norm.  This leads us to introduce the expression 
\[
\mathcal{H}^* \mathcal{N}^{lowhi} (\mathcal H A_x, \phi)
\]
In view of the estimates \eqref{AiB_est1}, \eqref{good-aj}
and \eqref{good-aj-diff},  the $N$ bound for the expression
\[
\mathcal{H}^* \mathcal{N}^{lowhi} ( A^{(m)}_x, \phi) - 
\mathcal{H}^* \mathcal{N}^{lowhi} (\mathcal H A^{(m-1)}_x, \phi),
\]
as well as the corresponding differences, is a consequence of the
following
\begin{equation}
  \lp{\mathcal{H}^*\mathcal{N}^{lowhi}( A_x,\phi)}{N} \ \lesssim \
    \lp{A_x}{Z^{hyp}}\lp{\phi}{S^1}
    \ . \label{gen_phi_Best}
\end{equation}
After peeling all the good contributions, the only part of $\mathcal N$
which has not been estimated so far is
\[
\mathcal{H}^*\mathcal{N}^{lowhi}( \mathcal H A^{(m)}_x,\phi^{(m+1)})
\]

\step{4}{Peeling off the good parts of $\mathcal N_0$} 
The arguments here follow the same lines as before. However, 
since $A_0$ solves an elliptic equation, a larger portion of 
all cases can be treated in a direct fashion via Strichartz estimates
and Sobolev embeddings, which leaves less in terms of bilinear 
estimates to be proved later. Recall that $\mathcal N_0(A_0,\phi) = 
A_0 \partial_t \phi$. Modifying slightly the setup  before we set 
\[
  \mathcal{N}_0^{lowhi}(A_0,\phi) = \ \sum_k
    \mathcal{N}_0( P_{<k-C} Q_{<k-C} A_0, P_k \phi) \ ,
\]
Then the difference is estimated via 
\begin{equation}\label{N0-lh}
\|  \mathcal{N}_0(A_0,\phi) - \mathcal{N}_0^{lowhi}(A_0,\phi)\|_{L^1 L^2}
\lesssim \|A_0\|_{Y^1} \|\phi\|_{S}
\end{equation}

The extra step we can take in the case of $ \mathcal{N}_0$ is to
replace $A_0^{(n)}$ by $A_0^{(n),hh}$.  By \eqref{A0_est1} and the second part 
of \eqref{A0parts} the difference is estimated using
\begin{equation}\label{N0-hh}
\|  \mathcal{N}_0^{lowhi}(A_0,\phi) - \mathcal{N}_0^{lowhi}(A_0^{hh},\phi)\|_{L^1 L^2}
\lesssim \|A_0-A_0^{hh}\|_{\ell^1 L^1 L^\infty} \|\phi\|_{S}
\end{equation}

The last two steps are similar to the case of $A_x$. First we introduce
\begin{equation}
    \mathcal{H}^* \mathcal{N}_0^{lowhi}(A_0^{hh},\phi) \ = \ 
    \sum_{\substack{k,k':\\ k'<k-C}}\mathcal{H}_{k'}^*  
    \mathcal{N} (A_0^{hh}, P_k\phi) \ , \notag
 \end{equation}
In view of \eqref{A0_est1} and the first part of \eqref{A0parts},
the corresponding differences are estimated using 
\begin{equation} 
 \lp{\mathcal{N}^{lowhi}_0(A_0^{hh},\phi) -\mathcal{H}^* 
\mathcal{N}^{lowhi}_0(A_0^{hh},\phi) }{N}
    \lesssim \lp{A^{hh}_0}{\ell^1 L^2(\dot{H}^{\frac{3}{2}})}
    \lp{\phi}{S^1} \ , \label{gen_phi_lohi_quadest0}
\end{equation}

Finally, replace this by $ \mathcal{H}^* \mathcal{N}_0^{lowhi}
(\mathcal H A_0,\phi)$. By \eqref{A0B_est2} and the third part of 
\eqref{A0parts} for the difference it suffices to have the estimate
\begin{equation} 
 \lp{\mathcal{H}^*\mathcal{N}^{lowhi}_0(A_0^{hh},\phi) -
      \mathcal{H}^*\mathcal{N}^{lowhi}_0( \mathcal{H} A_0,\phi)}{N} \lesssim
    \lp{A_0^{(1),hh}-\mathcal{H}
      A^{(1)}_0}{Z^{ell}}\lp{\phi^{(1)}}{S^1} \ .
    \label{quadphi_HL_Z0}
\end{equation}
After peeling all the good contributions, the only part of $\mathcal N_0$
which has not been estimated so far is
\[
\mathcal{H}^*\mathcal{N}^{lowhi}_0( \mathcal H A^{(m)}_0,\phi^{(m+1)})
\]

  \step{5}{Reduction of the remaining terms to quadrilinear null form
    bounds} So far we have one portion of $\mathcal N$ and one portion 
 of $\mathcal N_0$ left to estimate. We collect the two together 
in the expression 
\[
R^{(n+1)} =-    
\mathcal{H}^*\mathcal{N}^{lowhi}_0( \mathcal H A^{(m)}_x,\phi^{(m+1)})  +
\mathcal{H}^*\mathcal{N}^{lowhi}( \mathcal H A^{(m)}_0,\phi^{(m+1)})  
\]
Here we recall that $\mathcal H A^{(m)}_x$ and $\mathcal H A^{(m)}_0$ 
are bilinear expressions in $\phi^{(m-1)}$.

The key idea is that there is a cancellation that occurs between the
two terms above, which is why they need to be treated together rather
than separately. The trilinear estimate that needs to be proved in
this case for the trilinear expression 
\[
R(\phi^{(1)}, \phi^{(2)},\phi^{(3)}) 
 = -\mathcal{H}^*\mathcal{N}^{lowhi}_0( \mathcal H A_x(\phi^{(1)}, \phi^{(2)}),
\phi^{(3)})  +
\mathcal{H}^*\mathcal{N}^{lowhi}( \mathcal H A_0(\phi^{(1)}, \phi^{(2)}),\phi^{3})  
 \]
is 
 \begin{equation}
    \lp{R(\phi^{(1)}, \phi^{(2)},\phi^{(3)})  }{N} \ \lesssim \ 
    \lp{\phi^{(1)} }{S^1}\lp{\phi^{(2)} }{S^1}\lp{\phi^{(3)} }{S^1} \ , 
\label{quad_null_est}
\end{equation}
 Expanding everything into $\phi^{(i)}$, and
  performing some algebraic manipulations\footnote{See the Appendix
    for a calculation done without the clutter of abstract index
    notation.}  we have:
  \begin{equation}
R(\phi^{(1)}, \phi^{(2)},\phi^{(3)})    \ = \ 
    \mathcal{Q}_1- \mathcal{Q}_2 -\mathcal{Q}_3	\ , \notag
  \end{equation}
  where 
  \begin{align}
    \mathcal{Q}_1 \ &= \ \mathcal{H}^*\big(
    \Box^{-1}\mathcal{H}(\phi^{(1)}\partial_\alpha\phi^{(2)})\cdot
    \partial^\alpha \phi^{(3)}\big)
    \ , \label{quad_null1}\\
    \mathcal{Q}_2 \ &= \ \mathcal{H}^*\big(
    \Delta^{-1}\Box^{-1}\partial_t\partial_\alpha
    \mathcal{H}(\phi^{(1)}\partial^\alpha\phi^{(2)})\cdot
    \partial_t \phi^{(3)}\big)
    \ , \label{quad_null2}\\
    \mathcal{Q}_3 \ &= \ \mathcal{H}^*\big( \Delta^{-1}
    \Box^{-1}\partial_\alpha \partial^i
    \mathcal{H}(\phi^{(1)}\partial_i\phi^{(2)})\cdot
    \partial^\alpha \phi^{(3)}\big) \ . \label{quad_null3}
  \end{align}
  For each of these we prove \eqref{quad_null_est} separately, which
  concludes our estimates for $\mathcal{M}^{quad}$.

% -------------------------------------------------------------------------

\subsection{Proof of the $L^p$ product estimates}\label{subse:lp}

Before continuing with the main thrust of the paper, we pause here to
dispense with the easiest cases of the multilinear estimates listed
above.  These are \eqref{hi_mod1}, \eqref{hi_mod2}, \eqref{cubic_str},
\eqref{A0_energy}, \eqref{A0_est1}, \eqref{A0_est2}, \eqref{A0B_est1},
\eqref{quadfree_phi_himod}, \eqref{quad0_phi_himod},
\eqref{N0-hh}, \eqref{phi_cubic_est1}, and
\eqref{phi_cubic_est2}.  For the most part these can be broken down
into the following estimates, which are immediate consequences of
H\"older's and Bernstein's inequalitites:

\begin{lem}[Core generic product estimates]
  One has the following dyadic bounds:
  \begin{align}
    \lp{P_k (A_{k_1} \phi_{k_2})}{L^1(L^2)} \ &\lesssim \ 2^{\delta
      (k-\max\{k_i\})}2^{-\delta|k_1-k_2|}
    \lp{A_{k_1}}{L^2(\dot{H}^\frac{1}{2})}\lp{\phi_{k_2}}{L^2(\dot{W}^{6,\frac{1}{6}})} \ , \label{gen_prod1}\\
    \lp{P_k (\phi_{k_1}^{(1)}
      \phi_{k_2}^{(2)})}{L^2(\dot{H}^{-\frac{1}{2}})} \ &\lesssim \
    2^{\delta (k-\max\{k_i\})} 2^{-\delta|k_1-k_2|}
    \lp{\phi_{k_1}^{(1)}}{L^\infty(L^2)}\lp{\phi_{k_2}^{(2)}}{L^2(\dot{W}^{6,\frac{1}{6}})}
    \ , \label{gen_prod2}
  \end{align}
\end{lem}

\begin{proof}[Proof of estimate \eqref{hi_mod1}]
  This follows at once from summing over \eqref{gen_prod2} and the
  inclusion:
  \begin{equation}
    S^1 \ \subseteq \ L^2(\dot{W}^{6,\frac{1}{6}}) \ . \label{S1_L2L6}
  \end{equation}
\end{proof}

\begin{proof}[Proof of estimate \eqref{hi_mod2}]
  This follows from the inclusion $S^1\cdot S^1\subseteq
  L^\infty(L^2)$, a simple consequence of energy estimates and
  $\dot{H}^1\subseteq L^4$, and then summing over \eqref{gen_prod2}.
\end{proof}

\begin{proof}[Proof of estimate \eqref{cubic_str}]
  This follows from summing over \eqref{gen_prod1} after using
  \eqref{S1_L2L6} and the bilinear embedding:
  \begin{equation}
    B\cdot B\subseteq L^2(\dot{H}^\frac{1}{2}) \ , 
    \qquad B \ = \ L^2(\dot{W}^{6,\frac{1}{6}})\cap L^\infty(\dot{H}^1) \ . \label{B_bilin}
  \end{equation}
  This last estimate is a straightforward application of trichotomy, putting
  the high frequency term in the energy norm.  The standard details
  are left to the reader.
\end{proof}

\begin{proof}[Proof of estimate \eqref{A0_energy}]
  This is immediate from Sobolev embeddings involving
  $L^4,L^{\frac{4}{3}},L^2$.
\end{proof}

\begin{proof}[Proof of estimates \eqref{A0_est1}, \eqref{A0_est2}, \eqref{quadfree_phi_himod}, and
\eqref{quad0_phi_himod}]
These follow from \eqref{gen_prod2} as in estimates \eqref{hi_mod1}
and \eqref{hi_mod2} above. We use:
\begin{equation}
  L^2(\dot{H}^\frac{3}{2}) \ \subseteq \ L^2(\dot{W}^{6,\frac{1}{6}}) \ , \label{L2_L2L6}
\end{equation}
as a replacement for \eqref{S1_L2L6} when necessary.
\end{proof}

\begin{proof}[Proof of estimate \eqref{A0B_est1}]
  First notice that the desired bound for the cubic terms follows from
  \eqref{gen_prod1} and estimates similar to those above, and
  $\Delta^{-1} \ell^1L^1(L^2)\subseteq \ell^1 L^1(L^\infty)$. On the
  other hand, for the quadratic term one has:
  \begin{equation}
    \lp{P_k (\phi_{k_1}^{(1)} \partial_t \phi_{k_2}^{(2)})}{L^1(L^\infty)} \ \lesssim \ 
    2^{2k}2^{-\frac{1}{2}|k_1-k_2|}  \lp{\phi_{k_1}^{(1)}}{L^2(\dot{W}^{6,\frac{1}{6}})}
    \lp{\partial_t \phi_{k_1}^{(1)}}{L^2(\dot{W}^{6,\frac{1}{6}})} \ . \notag
  \end{equation}
  when $k = \max\{k_i\}+O(1)$.
\end{proof}

\begin{proof}[Proof of estimate \eqref{phi_cubic_est1}]
  This is again an immediate consequence of \eqref{gen_prod1},
  \eqref{gen_prod2}, and \eqref{S1_L2L6}
\end{proof}

\begin{proof}[Proof of estimate \eqref{phi_cubic_est2}]
  For the $L^2(\dot{H}^{-\frac{1}{2}})$ we can use \eqref{gen_prod2}
  once we place the product of $A_0$ in $L^\infty (L^2)$ via
  $\dot{H}^1\subseteq L^4$. On the other hand the $L^1(L^2)$ estimate
  comes from \eqref{gen_prod1} and using \eqref{B_bilin} for the
  produce of $A_0$ and \eqref{S1_L2L6} for $\phi$.
\end{proof}

\begin{proof}[Proof of estimate \eqref{N0-lh}]
  First decompose:
  \begin{equation}
    A_0\partial_t \phi -\mathcal{N}_0^{lowhi}(A_0,\phi) \ = \ T_1 + T_2  \ , \notag
  \end{equation}
  where:
  \begin{align}
    T_1 \ &= \ \sum_k P_{\geqslant k-C}A_0 \partial_t
    \phi_k
    \ , \notag\\
    T_2 \ &= \ P_{< k-C}Q_{>k-C}A_0 \partial_t \phi_k
  \end{align}
  We'll show each of $T_i\in L^1(L^2)$. For $T_1$ this follows from
  \eqref{gen_prod1} and using:
  \begin{equation}
    \lp{P_{\geqslant k-C}A_0}{L^2(\dot{H}^\frac{1}{2})} \ \lesssim \ 
    \sum_{k'} 2^{-k'} \lp{P_{k'}A_0}{L^2(\dot{H}^\frac{3}{2})} \ , \quad
    \lp{\partial_t \phi_k}{L^2(\dot{W}^{6,\frac{1}{6}})} \ \lesssim \ 
    2^{k}\lp{\phi_k }{S^1} \ . \notag
  \end{equation}
  The bound for $T_2$ is similar except we use:
  \begin{equation}
    \lp{P_{< k-C}Q_{\geqslant k-C}A_0}{L^2(\dot{H}^\frac{1}{2})} \ \lesssim \ 
    2^{-k} \lp{\partial_t A_0}{L^2(\dot{H}^\frac{1}{2})}  \ . \notag
  \end{equation}
  The bound for $\|T_2\|_{L^1 L^2}$ follows from \eqref{A0B_est1}, placing the second factor in $L^\infty(L^2)$.
\end{proof}

\begin{proof}[Proof of estimate \eqref{N0-hh}]
This follows from \eqref{A0B_est1} by placing again
  the second factor in $L^\infty(L^2)$.
\end{proof}

We remark that at this point we have reduced the proof of our main
Theorem \ref{thm:Main} to the demonstration of the linear estimates in Theorem
\ref{main_linear_thm} in Section~\ref{s:para-lin}, 
and then proving the multilinear estimates
\eqref{Ai_null_est}, \eqref{Ai_null_Best}, \eqref{A0B_est2},
\eqref{quadphi_HL_null}, \eqref{gen_phi_lohi_quadest},
\eqref{gen_phi_Best}, \eqref{gen_phi_lohi_quadest0},
\eqref{quadphi_HL_Z0}, and finally \eqref{quad_null_est} for each of
the expressions \eqref{quad_null1}--\eqref{quad_null3}. These bounds
are dealt with in the last section of the paper.

% -------------------------------------------------------------------------
%%%%%%%%%%%%%%%%%%%%%%%%%%%%%
% -------------------------------------------------------------------------

\section{Higher regularity and 
continuous dependence}\label{Sec:reg}

In the previous section we have constructed solutions which are small
in $S^1$ for all small energy data.  Here we complete the proof of
Theorem~\ref{thm:Main} by discusing the remaining issues, namely 
higher regularity and continuous dependence. The ideas here 
are fairly standard, and in particular closely resemble the similar
proof for wave maps, see \cite{Tat-wm}. For that reason we merely outline
the arguments that follow.

\subsection{Higher regularity}
The goal here is to establish the bound
\begin{equation}
\| A_x^{\text{nonlin}}\|_{l^1 S^N} +\|\phi\|_{S^N} + \| A_0\|_{Y^N} \lesssim 
\|(A_x[0],\phi[0])\|_{\dot H^N \times \dot H^{N-1}} 
\end{equation}
for all small energy data MKG solutions. This is done by a standard iterative 
procedure, and we explain here how to obtain bounds on $\nabla_x A_x$, 
$\nabla_x A_0$, $\nabla_x\phi$. In fact, this is accomplished by using the 
same estimates as before. Commencing with $A_x, A_0$, recall that we have 
\[
\Box A_i^{\text{nonlin},(m+1)} = - \mathcal{P}_i \big(\phi^{(m)}\nabla_x\phi^{(m)}+
  |\phi^{(m)}|^2 A^{(m)}_x \big)
 \]
 whence we see that $\nabla A_i^{\text{nonlin},(m+1)}$ satisfies a
 wave equation whose source term is a multilinear expression like the
 one for $\Box A_i^{\text{nonlin},(m+1)}$ but with one factor
 $\phi^{(m)}$ or $A^{(m)}_x$ replaced by $\nabla_x\phi^{(m)}$,
 respectively $\nabla_x A^{(m)}_x$. Using the same multilinear
 estimates as before, one then concludes that
\[
\big\|\nabla_x A_i^{\text{nonlin},(m+1)} - \nabla_x
A_i^{\text{nonlin},(m)} \big\|_{l^1S^1} \lesssim (C\epsilon_*)^{m},
\]
\[
\big\|(1-\mathcal{H})\nabla_x A_i^{\text{nonlin},(m+1)} - (1-\mathcal{H})\nabla_x A_i^{\text{nonlin},(m)} \big\|_{Z^{hyp}}\lesssim (C\epsilon_*)^{m}
\]
with implicit constant depending on $\|\nabla_x\phi[0,
\cdot]\|_{\dot{H}^1\times L^2} + \|\nabla_x A_i[0]\|_{\dot{H}^1\times
  L^2}$. A similar argument applies to $\nabla_x A_0$,
$\nabla_x\partial_t A_0$.  It remains to bound
$\nabla_x\phi^{(m+1)}$. Here we recall \eqref{phieq-nonlin}
\[
\Box_A^p \phi^{(m+1)} = \mathcal M(A^{(m)}, \phi^{(m+1)}),
\]
which leads to 
\[
\Box_A^p(\nabla_x\phi^{(m+1)}) = \nabla_x\mathcal M(A^{(m)}, \phi^{(m+1)}) + 2i\sum_k\nabla_x P_{<k-C}A^{free, j}P_k\partial_j\phi^{(m+1)}
\]
From \eqref{mult1}, we find with $k_1<k_2-C$
\[
\big\|\nabla_x P_{k_1}A^{free, j}\partial_jP_{k_2}\phi^{(m+1)}\big\|_{N}\lesssim 2^{-\delta|k_1-k_2|}\|P_{k_1}A^{free, j}\|_{S^1}\|\nabla_x P_{k_2}\phi^{(m+1)}\|_{S^1}, 
\]
from which we easily infer 
\[
\big\|2i\sum_k\nabla_x P_{<k-C}A^{free, j}P_k\partial_j\phi^{(m+1)}\big\|_{N}\lesssim \|A^{free}_x\|_{S^1}\|\nabla_x\phi\|_{S^1}
\]
On the other hand, the term 
\[
\nabla_x\mathcal M(A^{(m)}, \phi^{(m+1)})
\]
is again a multilinear expression in $\phi^{(k)}, A^{(k)},
\nabla_x\phi^{(k)}, \nabla_xA^{(k)}$, $k \in \{m, m-1\}$, with at most
one derivative factor in each monomial. Thus the multilinear estimates
from above apply.

\subsection{Frequency envelope bounds}

The $S^1$ bounds for the small data solutions can be supplemented with
corresponding frequency envelope bounds in a standard manner.
Precisely, suppose that $\{c_k\}$ is a $\dot H^1 \times L^2$ frequency
envelope for the data, in the sense that
\begin{equation}
 \|P_k (A_x[0],\phi[0])\|_{\dot H^1 \times L^2} \leq c_k, \qquad 
\|(A_x[0],\phi[0])\|_{\dot H^1 \times L^2} \approx \|c_k\|_{l^2} \leq \epsilon \ll 1
\label{fe-data}\end{equation}
Assume also that $\{c_k\}$ is slowly varying, 
\[
|c_k /c_j| \leq 2^{\delta|j-k|}, \qquad \delta \ll 1
\]

Then we have a similar bound for the solutions, 
\begin{equation}
\| P_k A_x^{\text{nonlin}}\|_{S^1} \lesssim c_k^2, \qquad 
\| P_k \phi\|_{S^1} + \| P_k A_0\|_{Y^1} \lesssim c_k
\label{fe-sln}\end{equation}
The proof of these bounds is a straightforward consequence 
of the estimates in the previous section, since we have 
off-diagonal decay in all of our multilinear estimates.

\subsection{Weak Lipschitz dependence on data}

While establishing Lipschitz dependence on data in the energy norm
would be desirable, this does not seem to hold because of the free 
wave component of the magnetic wave equation. Instead, we have a weaker
bound for the difference of two solutions $(A^{(1)}_x,\phi^{(1)})$ and 
 $(A^{(2)}_x,\phi^{(2)})$, of the form 
\begin{equation}
\|  A_x^{(1)} -A_x^{(2)} \|_{S^{1-\delta}} +
\|\phi^{(1)}-\phi^{(2)})\|_{S^{1-\delta}} + \| A_0^{(1)} - A_0^{(2)}\|_{Y^{1-\delta}}
 \lesssim \| (A^{(1)}_x-A^{(2)}_x,\phi^{(1)}-\phi^{(2)})[0]\|_{\dot H^{1-\delta} \times \dot H^{-\delta}} 
\label{wk-lip}\end{equation}
for small positive $\delta$. 

This is equivalent to a similar bound for the linearized equation. All
the components of the nonlinearity for which we have off-diagonal
decay in the multilinear estimates cause no difficulty, and impose no
sign restriction on $\delta$. The only difficulty arises in the paradifferential 
magnetic wave equation
\[
\Box_A^p \phi = 0
\]
Denoting by $(B,\psi)$ the corresponding linearized variables, we are 
led to consider the equation
\[
\Box_A^p \psi = - \sum_k B_{j,<k} \partial_j \phi_k
\]
where $B$ is an $H^{1-\delta}$ free wave. By the bilinear estimate
\eqref{gen_phi_lohi_quadest}, for the right hand side we have the bound
\[
\|\sum_k B_{j,<k} \partial_j \phi_k\|_{N^{-\delta}} \lesssim 
\|B[0]\|_{\dot H^{1-\delta} \times \dot H^{-\delta}} \| \phi\|_{S^1}, \qquad \delta > 0
\]
which leads to the desired linearized bound
\[
\|\psi\|_{S^{1-\delta}}  \lesssim \|\psi[0]\|_{\dot H^{1-\delta} \times \dot H^{-\delta}}
+\|B[0]\|_{\dot H^{1-\delta} \times \dot H^{-\delta}} \| \phi\|_{S^1}, \qquad \delta > 0
\]

\subsection{Approximation by smooth solutions}

Consider a small energy Coulomb initial data $(A_x, \psi)[0]$
and its regularizations $ (A_x^{(m)}, \psi^{(m)})[0] = P_{<m} (A_x, \psi)[0]$.
Denote by $(A_x,\phi)$, respectively $(A_x^{(m)}, \psi^{(m)})$ the corresponding 
solutions. Our aim here is to prove the following:

\begin{lemma}
Let $\{c_k\}$ be a slowly varying frequency envelope for  $(A_x, \psi)[0] $ in the energy norm $\dot H^1 \times L^2$. 
Then 
\begin{equation}\label{reg-app}
\| (A_x^{(m)} - A_x   , \psi^{(m)} -\psi)\|_{S^1}^2 \lesssim \sum _{k > m} c_k^2
\end{equation}
\end{lemma}

The proof of the lemma is straightforward, by using \eqref{wk-lip} to bound 
the fequencies less tham $m$ for the difference, and by using \eqref{fe-sln}
for the higher frequencies of each of the terms.

\subsection{Continuous dependence}

Given a convergent sequence of small data $(A_x^k, \psi^{k})[0] \to  
(A_x, \psi)[0]$ in $\dot H^1 \times L^2$, we consider their frequency 
envelopes $\{c^{k,m}\}$ respectively $c^k$ in $l^2$. Due to the 
above convergence, it follows that 
\begin{equation}\label{l2-small}
\lim_{m_0 \to \infty} \sum_{m > m_0} |c^{k,m}|^2 = 0, \qquad \text{ uniformly in $k$}
\end{equation}

We consider the corresponding regularized data $(A_x^{k,(m)}, \psi^{k,(m)})[0] $
and the corresponding solutions $(A^{k,(m)}, \psi^{k,(m)}) $.  By the relation 
\eqref{l2-small} and \eqref{reg-app} it follows that 
\[
(A^{k,(m)}, \psi^{k,(m)}) \to (A^{k}, \psi^{k}) \ \ \text{in} \ S^1, \qquad
\text{uniformly in k}
\]
On the other hand by the well-posedness theory for smooth data we have 
the convergence 
\[
(A^{k,(m)}, \psi^{k,(m)}) \to (A^{(m)}, \psi^{(m)}) \ \ \text{in} \ H^N_{loc}
\]
Combining the two we obtain the desired local in time convergence 
\[
(A^{k}, \psi^{k}) \to (A, \psi) \ \ \text{in} \ S^1_{loc}.
\]

% -------------------------------------------------------------------------
%%%%%%%%%%%%%%%%%%%%%%%%%%%%%
% -------------------------------------------------------------------------

\section{Renormalization}
\label{s:para-lin}
In what follows we work with the selfadjoint paradifferential
covariant wave operator
\begin{equation}\label{para}
  \Box_A^{p}=\Box - 2i \sum_{k}  A^j_{<k-C}\partial_j P_k
\end{equation}
Here $\nabla_x\cdot A=0$ are Coulomb gauge potentials which solve the
free wave equation $\Box A=0$.  Our first goal is to prove estimates
and construct a parametrix for the frequency localized evolution
\begin{equation}\label{para1}
  \Box^p_A \phi = f, \qquad (\phi(0),\phi_t(0)) = (g,h)
\end{equation}
with all functions $\phi$, $f$, $g$, $h$ localized at frequency $1$.

In general $A\neq 0$, so one cannot in addition have $dA=0$; in other
words the derivative interaction cannot be removed via a physical
space gauge transformation.  The situation changes drastically if one
views the gauge potentials as pseudo-differential operators. This
stems from the fact that when viewed microlocally all connections have
zero curvature, because they contain only one component for each fixed
direction in phase space.  To write the gauge interaction in terms of
a potential we need to solve:
\begin{equation}
  A\cdot \xi \ = \ - \tau \psi_t + \xi \psi_x . \notag
\end{equation}
This is impossible, but if we further restrict ourselves to the region
where $\tau\approx \pm |\xi|$, we are left with the exact solution:
\begin{equation}
  A\cdot \xi \ = \ d\psi_\pm\cdot (\mp|\xi|,\xi) \ , \qquad
  \psi_\pm(t,x,\xi) \ = \
  - L^\omega_{\pm} \Delta^{-1}_{\omega^\perp}(A\cdot \omega)
  \ . \notag
\end{equation}
Here $\psi_{\pm}$ are real valued and
\[
\omega = \frac{\xi}{|\xi|}, \qquad
L^\omega_{\pm}=\partial_t\pm\omega\cdot\nabla_x, \qquad
\Delta_{\omega^\perp} =\Delta-(\omega\cdot\nabla_x)^2,
\]
Also we have used the relation $L^\omega_+L^\omega_- =\Box+
\Delta_{\omega^\perp}$.  As defined $\psi_+$ is associated to the
upper cone $\{\tau = |\xi|\}$ and $\psi_-$ is associated to the lower
cone $\{\tau = -|\xi|\}$.

Quantizing this, one can write the reduced covariant wave equation
approximately as:
\begin{equation}
  \Box^p_A \ \approx \ \Box - 2i (\partial^\alpha \psi_\pm)(t,x,D)
  \partial_\alpha \ ,   \qquad \pm\tau > 0  \ .\notag
\end{equation}
which suggests that one should be be able to remove the gauge
interaction through the pseudodifferential conjugation:
\begin{equation}
  \Box^p_A \ \approx e^{-i\psi_\pm}(t,x,D)\Box e^{-i\psi_\pm}(D,y,s) \ 
  \qquad \pm\tau > 0 \ . \notag
\end{equation}

Applying this algorithm directly does not work well for two reasons.
First, the symbol we obtain is not localized at frequency $\ll 1$.
Secondly, the symbol $\Psi$ defined above is too singular due to
degeneracy of the operator $\Delta_{\omega^\perp}^{-1}$; this
corresponds exactly to parallel frequencies in $A$ and $\phi$.

To remedy the latter issue we take advantage of the fact that in the 
bilinear null form estimates there is a small angle gain.
This allows us to tightly cut off the small angle interactions
between $A$ and $\phi$ in the construction of $\Psi$, 
though not uniformly with respect to the $A$ frequencies.  Precisely, we
define the dyadic portions of $\psi$ by
\begin{equation}\label{defpsi}
  \psi_{k,\pm}(t,x,\xi) \ = \
  - L^\omega_{\pm} \Delta^{-1}_{\omega^\perp}(\Pi^{\omega}_{>\delta k}
  A_k\cdot \omega)
\end{equation}
Then the full $\psi$'s are defined by
\[
\psi_{\pm}: = \psi_{< 0, \pm}
\]
For the renormalization we will use frequency localized versions of
$e^{\pm i \psi_{\pm}}$, namely the operators
\[
e^{- i\psi_\pm}_{<0} (t,x,D), \qquad e^{ i\psi_\pm}_{<0} (D,y,s)
\]
Here we use $P(x,D)$ for the left quantization, and $P(D,y)$ for the
right quantization. Also the subscript $<0$ stands for the space-time
frequency localization at frequencies $\ll 1$.

Our main goal now is to show that the renormalizations are compatible
with the $S$ and $N$ spaces. Below we use the notation
\[
\Box^p_{A_{<0}}: = \Box - 2iA_{<-C}^j\partial_jP_0
\]
and analogously for $\Box^p_{A_{<k}}$.

\begin{theorem}\label{main_linear_ests}
  The frequency localized renormalization operators have the following
  mapping properties with $Z \in \{ N_0,L^2, N^*_0\}$:
  \begin{align}
    e^{\pm i\psi_{\pm}}_{<0}(t,x,D): \quad & Z \to Z \ , \label{N_mapping}\\
    \partial_t e^{\pm i\psi_\pm}_{<0}(t,x,D) : \quad & Z
    \to Z \ , \label{dtS_mapping} \\
    e^{-i\psi_{\pm}}_{<0}(t,x,D)e^{i\psi_{\pm}}_{<0}(D,y,s)-I: \quad &
    Z \to\epsilon Z,
    \  \label{S_mapping}\\
    e^{-i\psi_{\pm}}_{<0}(t,x,D) \Box -
    \Box^p_{A_{<0}}e^{-i\psi_{\pm}}_{<0}(t,x,D)
    : \quad& N^*_{0,\pm}\to \epsilon N_{0,\pm} \ . \label{error_ests}\\
    e^{- i\psi_{\pm}}_{<0}(t,x,D): \quad &S_0^\sharp \to S_0 \ ,
    \label{parametrix_bound}
  \end{align}

\end{theorem}

The above theorem allows us to construct an approximate solution for
\eqref{para} as follows:
\begin{equation}\label{eq:parametrix} \begin{split}
    \phi_{app} = & \ \frac12 \sum_{\pm} e^{- i\psi_\pm}_{<0} (t,x,D)
    \frac{1}{|D|}e^{\pm it|D|} e^{ i\psi_\pm}_{<0} (D,y,0) (|D| g \pm
    h) \\ & \ + e^{- i\psi_\pm}_{<0} (t,x,D) \frac{1}{|D|} K^\pm e^{
      \pm i\psi_\pm}_{<0} (D,y,s)f
  \end{split}
\end{equation}
where
\[
K^\pm f(t) = \int_{0}^t e^{\pm i(t-s)|D|} f(s) ds
\]
represents the solution to
\[
(\partial_t \mp i|D|) u = f, \qquad u(0) = 0
\]

Here if we drop the $e^{\pm i \phi_{\pm}}$ operators we simply have
the expression of the exact solution for the constant coefficient wave
equation. Thus the idea is that these operators approximately
conjugate the covariant wave flow to the constant coefficient wave
flow.

\begin{theorem}\label{thm:paraeqn_0}
  Assume that $f,g,h$ are localized at frequency $1$, and also that
  $f$ is localized at modulation $\lesssim 1$. Then $\phi_{app}$ is an
  approximate solution for $\Box^p_{A_{<0}}\phi = f,\,\phi(0) = (f, g),$ in
  the sense that
  \begin{equation}\label{Sbound}
    \| \phi_{app}\|_{S_0} \lesssim \|f\|_{N_0} + \|g\|_{L^2} + \|h\|_{L^2} 
  \end{equation}
  and
  \begin{equation}\label{small-err}
    \| \phi_{app}[0] - (g,h)\|_{L^2}+
    \|   \Box^p_{A_{<0}}\phi_{app}-f \|_{N_0} \ll \|f\|_{N_0} + \|g\|_{L^2} + \|h\|_{L^2}
  \end{equation}
\end{theorem}

We remark that $\phi_{app}$ constructed above has the same
localization at frequency $1$ and modulation $\lesssim 1$.  We also
remark that the actual parametrix construction only requires the
estimates \eqref{N_mapping}-\eqref{parametrix_bound}, but yields a weaker
form of \eqref{Sbound} with $S_0$ replaced by $N_0^*$.  Then the
estimate \eqref{parametrix_bound} serves to provide the additional
$S_0$ regularity. An easy consequence of this is the following

\begin{theorem}\label{thm:paraeqn}
  For all $f \in N$ and $(g,h) \in \dot H^1 \times L^2$ the solution
  to the paradifferential covariant wave equation \eqref{para1} is
  defined globally and satisfies
  \begin{equation}
    \| \phi\|_{S^1} \lesssim \|f\|_{N\cap l^1L^2\dot{H}^{-\frac{1}{2}}} + \|g\|_{\dot H^1} + \|h\|_{L^2}
  \end{equation}

\end{theorem}
\noindent
Again, using only \eqref{N_mapping}-\eqref{parametrix_bound} suffices but
yields a weaker form of \eqref{Sbound} with $S$ replaced by $N^*$.

For the remainder of the section we use Theorem~\ref{main_linear_ests}
to prove Theorems~\ref{thm:paraeqn_0},~\ref{thm:paraeqn}. The proof 
of Theorem~\ref{main_linear_ests} is completed in Sections~\ref{s:l2-gauge},
\ref{s:N-error}, after some preliminaries in Section~\ref{s:decomp}.

\begin{proof}[Proof of Theorem~\ref{thm:paraeqn_0}]

  The estimate \eqref{Sbound} follows directly by concatenating
  \eqref{N_mapping} and \eqref{parametrix_bound}.

  It remains to consider \eqref{small-err}. We begin with the initial
  data.  For the position we have
  \[
  \begin{split}
    \phi_{app}(0)-g = & \ \frac12 \sum_{\pm} e^{- i\psi_\pm}_{<0}
    (0,x,D) \frac{1}{|D|} e^{ i\psi_\pm}_{<0} (D,y,0) (|D| g \pm h) -
    g \\ = & \ \frac12 \sum_{\pm} \left[ e^{- i\psi_\pm}_{<0} (0,x,D)
      \frac{1}{|D|} e^{ i\psi_\pm}_{<0} (D,y,0)- \frac{1}{|D|}\right]
    (|D| g \pm h)
  \end{split}
  \]
  and we can apply the $L^2$ version of \eqref{S_mapping}.

  For the velocity we have
  \[
  \begin{split}
    \partial_t \phi_{app}(0)-h = & \ \frac12 \sum_{\pm} \pm e^{-
      i\psi_\pm}_{<0} (0,x,D) e^{ i\psi_\pm}_{<0} (D,y,0) (|D| g \pm
    h) - h \\ & + [\partial_t e^{- i\psi_\pm}_{<0}]
    (0,x,D)\frac{1}{|D|} e^{ i\psi_\pm}_{<0} (D,y,0) (|D| g \pm h)+ \\
    & + e^{- i\psi_\pm}_{<0} (0,x,D) \frac{1}{|D|} [\partial_t e^{
      i\psi_\pm}_{<0}] (D,y,0) (|D| g \pm h) \\ & \pm e^{-
      i\psi_\pm}_{<0} (0,x,D) e^{ i\psi_\pm}_{<0} (D,y,0) f(0)
  \end{split}
  \]
  The first line is rewritten as
  \[
  \frac12 \sum_{\pm} \left[ e^{- i\psi_\pm}_{<0} (0,x,D) e^{
      i\psi_\pm}_{<0} (D,y,0)- 1 \right] (\pm|D| g + h)
  \]
  and then we can use \eqref{S_mapping}. For the second and third
  lines we use \eqref{dtS_mapping}. Finally, for the last line we use
  \eqref{S_mapping} twice, along with the bound
  \[
  \|f(0)\|_{L^2} \lesssim \|f \|_{N_0}
  \]
  derived by Bernstein's inequality due to the unit modulation
  localization of $f$.

  Lastly, we consider the error estimate. We have
  \[
  \begin{split}
    \Box^p_{A_{<0}} \phi_{app} - f = & \sum_{\pm} [\Box^p_{A_{<0}} e^{-
      i\psi_\pm}_{<0} (t,x,D) - e^{- i\psi_\pm}_{<0} (t,x,D)\Box]
    \phi_{\pm} \\ & \pm \frac 12 e^{- i\psi_\pm}_{<0}
    (t,x,D)\frac{D_t\pm |D|}{|D|} e^{i\psi_\pm}_{<0}(D, y, s)f - f \\
    = &\ \sum_{\pm} [\Box^p_{A_{<0}} e^{- i\psi_\pm}_{<0} (t,x,D) - e^{-
      i\psi_\pm}_{<0} (t,x,D)\Box] \phi_{\pm}
    \\
    &+ \frac12 [ e^{- i\psi_\pm}_{<0} (t,x,D)e^{i\psi_\pm}_{<0}(D, y,
    s)- 1]f \\ & \ \pm \big[\frac12 [e^{- i\psi_\pm}_{<0}
    (t,x,D)\frac{1}{|D|} e^{i\psi_\pm}_{<0}(D, y, s)-
    \frac{1}{|D|}] \partial_t f \\ & \ + e^{- i\psi_\pm}_{<0}
    (t,x,D)\frac{1}{|D|} [\partial_t e^{i\psi_\pm}_{<0}](D, y, s))]f\big]
  \end{split}
  \]
  where
  \[
  \phi_{\pm} = \frac{1}{|D|}[ e^{\pm it|D|} e^{ i\psi_\pm}_{<0}
  (D,y,0) (|D| g \pm h) + K^\pm e^{ \pm i\psi_\pm}_{<0} (D,y,s)]
  \]
  For the first line we bound $\phi_{\pm}$ in $N^*_{\pm}$ via
  \eqref{N_mapping} and $\Box^{-1}$ estimates, and then use the
  conjugation bound \eqref{error_ests}.  For the second and third we
  use \eqref{S_mapping} in $N$. Finally for the fourth line we use
  \eqref{dtS_mapping} in $N$.
\end{proof}
\bigskip

\begin{proof}[Proof of Theorem~\ref{thm:paraeqn}] Consider first the
  frequency localized problems (with $k$ referring to spatial
  frequency)
  \[
  \Box^p_{A_{<k}} \phi_k = f_k,\,\phi_k(0) = (g_k, h_k),\,k\in
  \mathbf{Z}
  \]
  We solve each of these approximately and re-assemble the solutions
  $\phi_{app, k}$ to the full approximate solution $\phi_{app} =
  \sum_{k}\phi_{app, k}$. We cannot immediately apply the preceding
  theorem, since we make no further assumption on the space-time
  Fourier support of the source $f$. To remedy this, we split
  \[
  f_k = f_k^{hyp} + f_k^{ell},
  \]
  where $f_k^{hyp}$ is supported in the region $||\tau| -
  |\xi||\lesssim 2^k$.Then we solve the two problems
  \[
  \Box^p_{A_{<k}} \phi_k^1 = f_k^{hyp},\,\qquad \Box^p_{A_{<k}} \phi_k^2
  = f_k^{ell}
  \]
  approximately, the first by using the previous theorem, the second
  by neglecting the magnetic term $2iA_{<k}^j\partial_j\phi_k^2$. We
  then solve the second equation by 'division by the symbol', i. e. we
  set (from now on $k = 0$ by scaling invariance and we omit the
  subscripts)
  \[
  \phi^2 : = \Box^{-1}f^{ell}
  \]
  where the operator $\Box^{-1}$ is defined by multiplication with
  $\frac{1}{\tau^2 - |\xi|^2}$ on the Fourier side. Then the bound
  \[
  \|\phi^2\|_{S_0}\lesssim \|f^{ell}\|_{N_0}
  \]
  is immediate, and we reduce to solving the problem
  \[
  \Box^p_{A_{<0}} \phi^1 = f^{hyp},\qquad \phi^1(0) = (g, h) -
  \Box^{-1}f^{ell}(0)
  \]
  which we do by invoking Theorem~\ref{thm:paraeqn_0}.

  It remains to control the additional error generated by our
  approximate solution $\phi^2$, which is
  \[
  2iA^j_{<0}\partial_j\phi^2
  \]
  This is estimated by
  \[
  \|2iA^j_{<0}\partial_j\phi^2\|_{L_t^1 L_x^2\cap L^2\dot{H}^{-\frac{1}{2}}}\lesssim
  \|A^j_{<0}\|_{L_t^2 L_x^\infty}\|\phi^2\|_{L_{t,x}^2\cap L_t^\infty L_x^2}\ll
  \|f^{ell}\|_{N_0}
  \]
  It follows that our approximate solution $\phi_{app} =
  \sum_{k}\phi_{app, k}$ satisfies the conditions
  \[
  \|\Box^p_{A}\phi_{app} - f\|_{N\cap l^1L^2\dot{H}^{-\frac{1}{2}}} + \|\phi_{app}(0) - (g,
  h)\|_{\dot{H}^1\times L^2}\ll \|f\|_{N\cap l^1L^2\dot{H}^{-\frac{1}{2}}} + \|(g,
  h)\|_{\dot{H}^1\times L^2}
  \]
  Also, observe that the preceding Theorem~\ref{thm:paraeqn_0} implies
  \[
  \|\phi_{app}\|_{S^1}\lesssim \|f\|_{N\cap l^1L^2\dot{H}^{-\frac{1}{2}}} + \|(g, h)\|_{\dot{H}^1\times
    L^2}
  \]
  The proof is now completed by simple iterative application of the
  preceding to the successive errors.
\end{proof}

%%%%%%%%%%%%%%%%%%%%%%%%%%%%%
% ------------------------------------------------------------------------
% ------------------------------------------------------------------------
%%%%%%%%%%%%%%%%%%%%%%%%%%%%%

\section{Decomposable Spaces and Some Symbol Bounds}
\label{s:decomp}

\subsection{Review of the Basic Decomposable Calculus}

First we discuss the notion of decomposable function spaces and
estimates.  Recall that a zero homogeneous symbol $c(t,x;\xi)$ is said
to be in ``decomposable $L^q(L^r)$'' if $c=\sum_\theta c^{(\theta)}$,
$\theta\in 2^{-\mathbb{N}}$, and:
\begin{equation}
  \sum_\theta \lp{c^{(\theta)}}{D_\theta \big(L_t^{q}(L_x^{r})\big)} 
  \ < \ \infty \ , \label{decomp_sum}
\end{equation}
where:
\begin{equation}
  \lp{c^{(\theta)}}{D_\theta \big(L_t^{q}(L_x^{r})\big)}  \ = \ 
  \big\Vert \Big( \sum_{k=0}^{10n} \
  \sum_\phi \ \sup_\omega\ \lp{b^{\phi}_\theta\
    (\theta \nabla_\xi )^k \ c^{(\theta)} }
  {L_x^{r}}^2\Big)^\frac{1}{2}\big\Vert_{L_t^{q}}
  \ . \label{theta_decomp_norm}
\end{equation}
Here $b^{\phi}_\theta(\xi)$ denotes a cutoff on a solid angular sector
${\big|\xi|\xi|^{-1} - \phi\big|\leqslant \theta}$ for a fixed
$\phi\in\mathbb{S}^{n-1}$, and the sum is taken over a uniformly
finitely overlapping collection.  We define $\lp{b}{DL^q(L^r)}$ as the
infimum over all sums \eqref{decomp_sum}.  In \cite{KrSte} it is shown
that the following H\"older type inequality holds:
\begin{equation}
  \lp{\prod_{i=1}^m b_i}{DL^q(L^r)} \ \lesssim \ 
  \prod_{i=1}^m \lp{b_i}{DL^{q_i}(L^{r_i})} \ , \ \ 
  (q^{-1},r^{-1})=\sum_i (q_i^{-1},r_i^{-1}) 
  \ . \label{decomp_holder}
\end{equation}
In the sequel we only need a special case of decompositions provided
in terms of these norms:

\begin{lem}[Decomposability Lemma]
  Let $A(t,x;D)$ be any pseudodifferential operator with symbol
  $a(t,x;\xi)$.  Suppose $A$ satisfies the fixed time bound:
  \begin{equation}
    \sup_t \lp{A(t,x;D)}{L^2\to L^2} \ \lesssim \ 1 \ . \label{ab}
  \end{equation}
  Then for any symbol $c(t,x;\xi)\in DL^{q}(L^r)$ one has the
  space-time bounds:
  \begin{equation}\label{decomp1}
    \begin{split}
      \lp{(ac)(t,x;D)}{L^{q_1} L^2\to L^{q_2}(L^{r_2})} \lesssim &
      \lp{c}{DL^{q}(L^r)} \ , \\
      \frac{1}{q_1} + \frac{1}{q} = \frac{1}{q_2}, \qquad \frac12 +
      \frac1r = \frac{1}{r_2}, \qquad & 1 \leq q_1,q_2,q,r,r_2 \leq
      \infty
      % \lp{(ac)^{l,r}(t,x;D)}{L^{r}(L^2)\to L^q(L^2)} &\lesssim
      % \lp{c}{DL^{p}(L^\infty)} \ , \ \ q^{-1}=p^{-1}+r^{-1} \
      % , \label{decomp2}
    \end{split}
  \end{equation}
\end{lem}

\begin{proof}
  Due to the $l^1$ summation over $\theta$ in \eqref{decomp_sum} it
  suffices to consider the case $c = c_\theta$ for a fixed $\theta$.
  We further decompose
  \[
  c_\theta(t,x,\xi) = \sum_\phi c_{\theta}^{\phi}(t,x,\xi), \qquad
  c_{\theta}^{\phi}(t,x,\xi):=b_\theta^\phi(\xi) c^{(\theta)}(t,x,\xi)
  \]
  By \eqref{theta_decomp_norm} each $c_{\theta}^{\phi}$ is supported
  in an angle $\theta$ sector, and it is smooth on the scale of its
  support. Thus by a Fourier series decomposition we can separate
  variables and represent
  \[
  c_{\theta}^{\phi}(t,x,\xi) = \sum_{j >0} d_{\theta}^{\phi,j}(t,x)
  e_{\theta}^{\phi,j}(\xi),
  \]
  where
  \[
  \|d_{\theta}^{\phi,j}(t,x)\|_{L^r_x}\lesssim j^{-N} \sum_{k=0}^{10n}
  \lp{b^{\phi}_\theta\ (\theta \nabla_\xi )^k \ c^{(\theta)}}{L^r_x} ,
  \qquad |e_{\theta}^{\phi,j}| \leq 1
  \]
  Due to the rapid decay with respect to $j$ it suffices to consider
  the contribution to $c_{\theta}^{\phi}$ coming from a single $j$,
  say $j=1$.  Then $c_\theta$ has the form
  \[
  c_\theta = \sum_\phi d_{\theta}^{\phi}(t,x) e_{\theta}^{\phi}(\xi),
  \]
  where
  \begin{equation} \label{dtheta} \| d_{\theta}^{\phi}\|_{L^q_t
      l^2_{\phi} L^r_x} \lesssim 1, \qquad |e_{\theta}^{\phi}| \leq 1
  \end{equation}
  Then we can represent
  \[
  (ac)(t,x,D) u = \sum_\phi d_{\theta}^{\phi}(t,x) \cdot
  A(t,x,D)e_{\theta}^{\phi}(D) u
  \]
  The second factor above inherits the $L^2$ norm from $u$ due to
  \eqref{ab} and the square summation in $\omega$ due to the sector
  decomposition.  Thus
  \[
  \|A(t,x,D)e_{\theta}^{\phi}(D) u \|_{L^{q_1}_t l^2_\omega L^2_x}
  \lesssim \|u\|_{L^2}
  \]
  The estimate for $(ac)(t,x,D) u$ follows by combining the last bound
  with \eqref{dtheta}.
\end{proof}

% --------------------------------------------------------------------------------------------------------

\subsection{A Decomposable Calculus for Pseudodifferential Products}

In the sequel it will also be useful for us treat estimates for
products of operators in a modular way. Recall that if $a(x,\xi)$ and
$b(x,\xi)$ are symbols, then $a^rb^r- (ab)^r\approx i(\partial_x
a \partial_\xi b)^r$. This formula is not exact, but it leads to an
estimate:

\begin{lem}[Decomposable product calculus]
  Let $a(x,\xi)$ and $b(x,\xi)$ be smooth symbols. Then:
  \begin{equation}
    \lp{a^rb^r- (ab)^r}{L^r(L^2)
      \to L^q(L^2)}  \lesssim  \lp{(\nabla_x a)^r}{L^r(L^2)\to L^{p_1}(L^2)}
    \lp{\nabla_\xi b}{D_1L^{p_2}(L^\infty)}  \label{spec_decomp1}
  \end{equation}
  where $q^{-1}=\sum p_i^{-1}$. Furthermore, if $b=b(\xi)$ is a smooth
  compactly supported multiplier, then for any two translation
  invariant spaces $X,Y$ one has:
  \begin{equation}
    \lp{a^rb^r- (ab)^r}{X\to Y}  \lesssim  
    \lp{(\nabla_x a)^r}{X\to Y}
    \ . \label{spec_decomp2}
  \end{equation}
\end{lem}

\begin{proof}
  To prove the first estimate \eqref{spec_decomp1} we write the kernel
  $K(x,y)$ of the difference $a^rb^r- (ab)^r$ as follows:
  \begin{align}
    K(x,y) = & \ c_n\int_{\mathbb{R}^{3n}} e^{i(x-z)\cdot
      \xi}e^{i(z-y)\cdot\eta} a(z,\xi)\big(b(y,\eta)-
    b(y,\xi)\big)dzd\xi d\eta \  \notag\\
    = & \ c_n \int_0^1\int_{\mathbb{R}^{3n}} e^{i(x-z)\cdot
      \xi}e^{i(z-y)\cdot\eta} a(z,\xi) \nabla_\xi b\big(y,s\eta +
    (1-s)\xi\big)\cdot (\eta-\xi) dzd\xi d\eta ds
    \ , \notag\\
    = & \ c_n i \int_0^1\int_{\mathbb{R}^{3n}} e^{i(x-z)\cdot
      \xi}e^{i(z-y)\cdot\eta} \nabla_x a(z,\xi)\cdot \nabla_\xi
    b\big(y,s\eta + (1-s)\xi\big) dzd\xi d\eta ds
    \ , \notag\\
    = & \ i \frac{1}{(2\pi)^n}\int_0^1\int_{\mathbb{R}^n}
    T_{(s-1)\Xi}\nabla_x a^r(x,D)T_{-s\Xi} \widehat{\nabla_\xi
      b}(\cdot,\Xi) dsd\Xi \ , \label{b_iden}
  \end{align}
  where we have used Fourier inversion:
  \begin{equation}
    \nabla_\xi b\big(y,s\eta + (1-s)\xi\big) \ = \ \frac{1}{(2\pi)^n}\int_{\mathbb{R}^n}
    e^{i(s\eta + (1-s)\xi)\cdot\Xi}\
    \widehat{\nabla_\xi b}(y,\Xi)\, d\Xi \ . \notag
  \end{equation}
  From the smoothness of $b$ on the unit scale in $\xi$ we have that
  the weight function $\widehat{\nabla_\xi b}$ obeys the estimate:
  \begin{equation}
    \int_{\mathbb{R}^n} \lp{\widehat{\nabla_\xi b}(\cdot,\Xi)}{L^{p_2}(L^\infty)}d\Xi
    \ \lesssim \ \lp{b}{D_1L^{p_2}(L^\infty)} \ . \notag
  \end{equation}
  Then \eqref{spec_decomp1} is a direct consequence of H\"older's
  inequality and the the translation invariance of $L^p$ spaces.

  Note that the proof of \eqref{spec_decomp2} also follows from the
  identity \eqref{b_iden} directly.\, because in this case the weight
  function $\widehat{\nabla_\xi b}$ is a constant in $y$.
\end{proof}

% ------------------------------------------------------------------------------------------------------

\subsection{Some Symbols Bounds for Phases}

For its use in the sequel, we list out a number of decomposable
estimates for the phase $\psi(t,x;\xi)$ used to define our microlocal
gauge transformations:

\begin{lem}[Decomposable estimates for $\psi$]
  Let the phase $\psi(t,x;\xi)$ be defined as in \eqref{defpsi}, and
  its angular components $\psi^{(\theta)}=\Pi_\theta^\omega
  \psi(t,x;\xi)$, where $\omega=|\xi|^{-1}\xi$.  Then for $q \geq 2$
  and $2/q+3/r \leq 1$ one has:
  \begin{equation}
    \lp{(\psi_k^{(\theta)},2^{-k}\nabla_{t,x}\psi_k^{(\theta)})}{DL^q(L^r)} \ 
    \lesssim \ 2^{-(\frac{1}{q}+\frac{4}{r})k}\theta^{\frac{1}{2}-\frac{2}{q}-\frac{3}r}\epsilon
    , \label{psi_decomp2}
  \end{equation}
  In particular
  \begin{align}
    \lp{(\psi_k,2^{-k}\nabla_{t,x}\psi_k)} {DL^q(L^\infty)} \
    &\lesssim \ 2^{-\frac{1}{q}k}\epsilon
    \ ,  &q&>4 \ , \label{psi_decomp1}\\
    \lp{\nabla_{t,x}\psi_k} {DL^2(L^r)} \ &\lesssim \ 2^{(\frac12-
      \frac{4}{r} - \delta(\frac12+\frac3r))k}\epsilon \ , &r&\geq 6 \
    , \label{psi_decomp3}
  \end{align}
\end{lem}

\begin{proof}
  Notice that the last two estimates follow from the first by summing
  over dyadic $2^{-\delta k} \leq \theta\lesssim 1$. For the first
  bound we interchange the $t$ integration and the $\omega$ summation
  to obtain:
  \[
  \begin{split}
    \lp{(\psi_k^{(\theta)},2^{-k}\nabla_{t,x}\psi_k^{(\theta)})}{DL^q(L^r)}
    \ \lesssim \ & \theta^{-2}2^{-k} \big(\sum_{\omega}
    \lp{\Pi^\omega_\theta(D)A \cdot
      \omega}{L^q(L^r)}^2\big)^\frac{1}{2} \ \notag \\ \ \lesssim \ &
    \theta^{-1}2^{-k} \big(\sum_{\omega} \lp{\Pi^\omega_\theta(D)A
    }{L^q(L^r)}^2\big)^\frac{1}{2},
  \end{split}
  \]
  where at the second step we have used the Coulomb gauge to gain
  another factor of $\theta$.

  Now we conclude using the Strichartz estimates.  In four space
  dimensions the Strichartz sharp range is given by
  $\frac{2}{q}+\frac{3}{r_0}=\frac{3}{2}$.  Moreover, on an angular
  sector of size $\theta$ Bernstein's inequality gives the embedding
  $\Pi^\omega_\theta(D)P_k L^{r_0}\subseteq \theta^{3(\frac{1}{r_0} -
    \frac{1}{r})} 2^{4(\frac{1}{r_0}-\frac{1}{r})k}L^{r}$. Thus:
  \begin{equation}
    \big(\sum_{\omega}
    \lp{\Pi^\omega_\theta(D)A_k}{L^q(L^r)}^2\big)^\frac{1}{2} \ \lesssim \
    \theta^{\frac{3}{2}-\frac{2}{q}-\frac{3}r}2^{(1-\frac{1}{q}-\frac{4}r)k}\lp{A_k}{S_k} \ , \notag
  \end{equation}
  and the second estimate follows.
\end{proof}

We wrap this section up by proving some additional symbol type bounds
for the phases $\psi$. These involve the variation over the physical
space variables:

\begin{lem}[Additional symbol bounds for $\psi$]
  Let $\psi$ be as above. Then one has:
  \begin{align}
    |\psi_{<k}(t,x;\xi) -\psi_{<k}(s,y;\xi)|
    &\lesssim   \epsilon \log( 1+2^k (|t-s| + |x-y|)), \label{symbol1} \\
    |\psi(t,x;\xi) -\psi(s,y;\xi)|
    &\lesssim   \epsilon \log ( 1+ |t-s| + |x-y|) \label{symbol2} \\
    % |\partial_\xi(\psi(t,x;\xi) -\psi(s,y;\xi))|
    % &\lesssim \epsilon \log \la (t-s,x-y)\ra, \\
    |\partial_\xi^\alpha( \psi(t,x;\xi) -\psi(s,y;\xi))| &\lesssim \
    \epsilon \la (t-s,x-y)\ra^{|\alpha -\frac12| \sigma}, 1 \leqslant
    \alpha \leqslant \sigma^{-1} . \label{symbol3}
  \end{align}
\end{lem}

\begin{proof}
  We decompose as before
  \[
  \psi_{<k}(t,x;\xi) = \sum_{j < k} \sum_{\theta > 2^{\sigma j}}
  \psi^{(\theta)}_j(t,x,\xi)
  \]
  For each fixed $\theta$ and $j$ we have by the definition of $\psi$
  and the Coulomb gauge condition
  \[
  |\psi^{(\theta)}_j(t,x,\xi)| \lesssim \theta^{-1} 2^{-j}
  \sup_{\omega} \|\Pi_\theta^\omega A_j\|_{L^\infty}
  \]
  Then by energy estimates for $A$ and Bernstein's inequality we
  obtain
  \begin{equation} \label{psij} |\psi^{(\theta)}_j(t,x,\xi)| \lesssim
    \theta^\frac12 \| A_j[0]\|_{H^1 \times L^2}, \qquad
    |\psi_j(t,x,\xi)| \lesssim \| A_j[0]\|_{H^1 \times L^2}
  \end{equation}\label{psijx}
  A similar argument leads to
  \begin{equation}
    |\partial_{t,x} 
    \psi^{(\theta)}_j(t,x,\xi)| \lesssim  2^j \theta^\frac12 \| A_j[0]\|_{H^1 \times L^2},
    \qquad 
    |\partial_{t,x}  \psi_j(t,x,\xi)| \lesssim    2^j  \| A_j[0]\|_{H^1 \times L^2}
  \end{equation}
  Differentiating with respect to $\xi$ yields $\theta^{-1}$ factors,
  \[
  |\partial_\xi^\alpha \psi^{(\theta)}_j(t,x,\xi)| \lesssim
  \theta^{\frac12-|\alpha|} \| A_j[0]\|_{H^1 \times L^2}, \qquad
  |\partial_{x,t}
  \partial_\xi^\alpha \psi^{(\theta)}_j(t,x,\xi)| \lesssim 2^j
  \theta^{\frac12-|\alpha|} \| A_j[0]\|_{H^1 \times L^2}.
  \]
  For the bound \eqref{symbol1} we use both \eqref{psij} and
  \eqref{psijx} to write for $j \leq k$
  \[
  |\psi_{<k}(t,x;\xi) -\psi_{<k}(s,y;\xi)| \lesssim 2^j(|t-s| + |x-y|)
  + |k-j|
  \]
  and then optimize the choice of $j$.

  The proof of \eqref{symbol3} is similar.

\end{proof}

%%%%%%%%%%%%%%%%%%%%%%%%%%%%%
% ------------------------------------------------------------------------
% ------------------------------------------------------------------------
%%%%%%%%%%%%%%%%%%%%%%%%%%%%%

\section{ $L^2$ estimates for the  gauge transformations}
\label{s:l2-gauge}
In this section we prove three core $L^2$ based estimates for the
gauge transformations $e^{\pm i \psi_{\pm}}$. These will serve as
building blocks in later sections.

% ------------------------------------------------------------------------

\subsection{Oscillatory integral estimates}

In order to prove various estimates involving the operators
$e^{-i\psi_{\pm}}_{<0}(t,x,D)$ and $e^{i\psi_{\pm}}_{<0}(D,y,s)$ we
need to obtain pointwise kernel bounds for operators of the form
\[
T^a = e^{-i\psi_{\pm}}(t,x,D) a(D) e^{\pm i(t-s)|D|}
e^{i\psi_{\pm}}(D,y,s)
\]
where $a$ is localized at frequency $1$.  The kernel of the operator
$T_a$ is given by the oscillatory integral
\[
K^a(t,x;s,y) = \int e^{-i\psi_{\pm}}(t,x,\xi) a(\xi) e^{\pm
  i(t-s)|\xi|} e^{i \xi(x-y)} e^{i\psi_{\pm}}(\xi,y,s) d\xi
\]
Our main estimates for such kernels are as follows:

\begin{proposition}\label{kernel_prop}
  a) Assume that $a$ is a smooth bump on the unit scale. Then the
  kernel $K_a$ satisfies
  \begin{equation} \label{wave-decay} |K_a(t,x;s,y)| \lesssim \langle
    t-s \rangle^{-\frac32} \langle |t-s| - |x-y| \rangle^{-N}
  \end{equation}

  b) Let $a = a_{C}$ be a bump function on a rectangular region $C$ of
  size $2^{k} \times (2^{k+l})^3$ with $k \leq l \leq 0$. Then
  \begin{equation}\label{wave-decay-cube}
    |K_a(t,x;s,y)| \lesssim  2^{4k+3l} \langle 2^{2(k+l)} (t-s) \rangle^{-\frac32}
    \langle 2^k (|t-s| - |x-y|) \rangle^{-N}
  \end{equation}
  If in addition $x-y$ and $C$ have a $2^{k+l}$ angular separation
  then
  \begin{equation}\label{off-angle}
    |K_a(t,x;s,y)| \lesssim  2^{4k+3l} \langle 2^{2(k+l)} |t-s| \rangle^{-N}
    \langle 2^k (|t-s| - |x-y|) \rangle^{-N}
  \end{equation}

\end{proposition}

\begin{proof}
  a) Away from a conic neighbourhood of the cone $\{ |t-s| =
  \pm|x-y|\}$ the phase
  \[
  \Psi = \pm (t-s)|\xi| + \xi(x-y) - (\psi_{\pm}(t,x,\xi)-
  \psi_{\pm}(s,y,\xi))
  \]
  is nondegenerate due to \eqref{symbol3} with $|\alpha|=1$.  Hence
  repeated integration by parts yields
  \[
  | K^a(t,x,s,y) | \lesssim \la(t,x)-(s,y)\ra^{-N}, \qquad N \sim
  \sigma^{-1}
  \]
  Near the cone we need to be more careful. Denoting $T = |t-s|+|x-y|$
  and $R= |t-s|-|x-y|$, in suitable (polar) coordinates this takes the
  form
  \[
  K^a(t,x,s,y) = \int e^{-i( \psi_{\pm}(t,x,\xi') -
    \psi_\pm(s,y,\xi'))} e^{i R \xi_1} e^{i T \xi'^2} \tilde a(\xi)
  d\xi
  \]
  In $\xi_1$ (the former radial variable) this is a straight Fourier
  transform.  Given the bound \eqref{symbol3}, we can use stationary
  phase in $\xi'$.  While the $\xi$ derivatives of the $\psi_\pm$ part
  of the phase are not bounded, they only bring factors of $T^\sigma$,
  which is small enough not to affect the stationary phase ( this
  works up to $\sigma = \frac12$).  We obtain
  \[
  |K^a(t,x,s,y)| \lesssim T^{-\frac32} (1+ R)^{-N}
  \]

  b) Away from the cone the estimate follows easily as above since the
  phase is nondegenerate.  Near the cone we use again polar
  coordinates to express our oscillatory integral as above,
  \[
  K^C(t,x,s,y) = \int e^{-i( \psi_{\pm}(t,x,\xi') -
    \psi_\pm(s,y,\xi'))} e^{i R \xi_1} e^{i T \xi'^2} \tilde a_C(\xi)
  d\xi
  \]
  where $a_C$ is a bump function in a rectangle on the $2^k$ scale in
  the radial variable $\xi_1$ and on the $2^{k+l}$ scale in the
  angular variable $\xi'$.  Then we can separate variables in
  $(\xi_1,\xi')$. We note that this rectangle need not be centered at
  $\xi'=0$, though this is the worst case.  In $\xi_1$ this is again a
  Fourier transform, so we get the factor
  \[
  2^k \langle 2^k R \rangle^{-N}
  \]
  In $ \xi'$ we can use stationary phase to get the factor
  \[
  2^{3(k+l)} \langle 2^{2(k+l)T} \rangle^{-\frac32}
  \]
  The bound \eqref{wave-decay-cube} follows by multiplying these two
  factors.

  Finally, the estimate \eqref{off-angle} corresponds to the case when
  $a_C$ is supported in $|\xi'| > 2^l$ in the above representation.
  If $T < 2^{-2(k+l)}$ then there are no oscillations in $\xi'$ on the
  $2^{k+l}$ scale, and we just use the brute force estimate. For $T >
  2^{-2(k+l)}$ the phase is nonstationary in $\xi'$, and we obtain the
  factor
  \[
  2^{3(k+l)} (1+2^{2(k+l)} T)^{-N}
  \]
\end{proof}

While the above proposition contains all the oscillatory integral
estimates which are needed, it does not apply directly to the
frequency localized operators $e^{-i\psi_{\pm}}_{<0}(t,x,D)$ and
$e^{i\psi_{\pm}}_{<0}(D,y,s)$.  For that we need to produce similar
estimates for the kernels $K_{a,<0}$ of the operators
\[
T^a_{<0} = e^{-i\psi_{\pm}}_{<0}(t,x,D) a(D) e^{\pm i(t-s)|D|}
e^{i\psi_{\pm}}_{<0}(D,y,s)
\]
The transition to such operators is made in the next
\begin{proposition}\label{kernel_prop_loc}
  a) Assume that $a$ is a smooth bump on the unit scale. Then the
  kernel $K^a_{<0}$ satisfies
  \begin{equation} \label{wave-decay0} |K^a_{<0}(t,x;s,y)| \lesssim
    \langle t-s \rangle^{-\frac32} \langle |t-s| - |x-y| \rangle^{-N}
  \end{equation}
  In addition, the following fixed time bound holds:
  \begin{equation} \label{almost-ortho} |K^a_{<0}(t,x;t,y)- \check
    a(x-y)| \leq \epsilon |\log \epsilon|
  \end{equation}

  b) Let $a = a_{C}$ be a bump function on a rectangular region $C$ of
  size $2^{k} \times (2^{k+l})^3$ with $k \leq l \leq 0$. Then
  \begin{equation}\label{wave-decay-cube0}
    |K^a_{<0} (t,x;s,y)| \lesssim  2^{4k+3l} \langle 2^{2(k+l)} (t-s) \rangle^{-\frac32}
    \langle 2^k (|t-s| - |x-y|) \rangle^{-N}
  \end{equation}

  c) Let $a = a_{C}$ be a bump function on a rectangular region $C$ of
  size $1 \times (2^{l})^3$ with $ l \leq 0$.  Let $\omega \in \S^3$
  be at angle $l$ from $C$. Then we have the characteristic kernel
  bound
  \begin{equation}\label{off-angle0}
    \begin{split}
      |K^a_{<0}(t,x;s,y)| \lesssim 2^{3l} \langle 2^{2l} |t-s|
      \rangle^{-N} & \langle 2^l |x'-y'| \rangle^{-N}
      \\
      t-s = (x-y)\cdot \omega
    \end{split}
  \end{equation}

\end{proposition}
\begin{proof}
  a) We represent the symbol $e_{<0}^{\pm i\psi_{\pm}}$ as
  \begin{equation}
    e_{<0}^{\pm i\psi_{\pm}} = \int m(z)e^{\pm iT_z\psi_{\pm}}\,dz
  \end{equation}
  where $m(z)$ is an integrable bump function on the unit scale and
  $T_z$ denotes translation in the direction $z$, with $z$
  representing space-time coordinates. Since the wave equation is
  invariant to translations, the symbol $e^{\pm iT_z\psi_{\pm}}$ is of
  the same type as $e^{\pm i\psi_{\pm}}$.  Using this representation
  for both $\psi_{\pm}$ exponentials, the kernel $K^a_{<0}$ can be
  expressed in the form
  \[
  \begin{split}
    K^a_{<0}(t,x,s,y)= \ &\int \int e^{-i T_z \psi_{\pm}}(t,x,\xi)
    a(\xi) e^{i(\pm |\xi|, \xi) \cdot (t-s,x-y)}
    e^{i T_{w} \psi_{\pm}}(s,y,\xi)\  d\xi  \ m(z) m(w) \ dz dw \\
    = & \ \int T_z T_w \int e^{-i \psi_{\pm}}(t,x,\xi) a(\xi) e^{i(\pm
      |\xi|, \xi) \cdot (z-w)} e^{i \psi_{\pm}}(s,y,\xi)\ d\xi \ m(z)
    m(w) \ dz dw
  \end{split}
  \]
  Denoting $a(z,w) (\xi)= a(\xi) e^{i(\pm |\xi|, \xi) \cdot (z-w)}$,
  we can express the above kernel in terms of the kernels $K^a$ in the
  previous proposition, namely
  \begin{equation}\label{ka0}
    K^a_{<0}(t,x,s,y)=  \int T_z T_w  K^{a(z,w)}(t,x,s,y)   \ m(z) m(w) \ dz dw
  \end{equation}

  To prove the bound \eqref{wave-decay0} we use \eqref{wave-decay},
  together with the additional observation that the implicit constant
  in \eqref{wave-decay} depends on finitely many seminorms of $a$ (at
  most 8, to be precise) which we denote by $||| a|||$. Then
  \[
  ||| a(z,w)||| \lesssim (1+|z|+|w|)^N
  \]
  However, this growth is compensated by the rapid decay of $m$,
  therefore the bound \eqref{wave-decay} for $K^a$ transfers directly
  to $K^a_{<0}$ in \eqref{wave-decay0}.

  To prove \eqref{almost-ortho} we use the same representation as
  above to write
  \[
  K^a_{<0}(t,x,t,y) - \check a(x-y) = \int \int [e^{-i (T_z
    \psi_{\pm}(t,x,\xi) - T_{w} \psi_{\pm}(t,y,\xi))}-1] a(\xi)
  e^{i\xi(x-y)} \ d\xi \ m(z) m(w) \ dz dw
  \]
  By \eqref{symbol2} we have
  \[
  | T_z \psi_{\pm}(t,x,\xi) - T_{w} \psi_{\pm}(t,y,\xi))| \lesssim
  \epsilon \log (1+ |z|+|w|+|x-y|)
  \]
  which yields
  \[
  \begin{split}
    |K^a_{<0}(t,x,t,y) - \check a(x-y) | \lesssim &\ \epsilon \int
    \log (1+ |z|+|w|+|x-y|) |m(z)||m(w)| dz dw \\ \lesssim &\ \epsilon
    \log (2+|x-y|)
  \end{split}
  \]
  This suffices if $\log (2 + |x-y|) \lesssim |\log \epsilon|$. But
  for larger $|x-y|$ we can use \eqref{wave-decay0} directly.

  b) Using the representation \eqref{ka0}, the bound
  \eqref{wave-decay-cube0} follows from \eqref{wave-decay-cube}
  exactly by the same argument as in case (a).

  c) Using the representation \eqref{ka0}, the same argument also
  yields the bound \eqref{wave-decay} provided we have the following
  estimate for $K^a$:
  \[
  |K^a(t,x,s,y)| \lesssim 2^{3l} \langle 2^{2l} |t-s| \rangle^{-N}
  \langle 2^{l} |x'-y'| \rangle^{-N} (1+|(t-s)- (x-y) \cdot
  \omega|)^{10 N}
  \]
  To see that this is true, we consider three cases:

  (i) If $|t-s| \lesssim 2^{-2l}$ then \eqref{wave-decay-cube} applies
  directly.

  (ii) If $|t-s| \gg 2^{-2l}$ but $ ||x-y|-|t-s|| \gtrsim 2^{l}
  |x'-y'|+ 2^{2l}|t-s|$ then \eqref{wave-decay-cube} still suffices.

  (iii) If $|t-s| \gg 2^{-2l}$ and $|(t-s)- (x-y) \cdot \omega|)|
  \gtrsim 2^{l} |x'-y'|+ 2^{2l}|t-s| $ then \eqref{wave-decay-cube}
  also applies.

  (iv) Finally, if $|t-s| \gg 2^{-2l}$, but $ ||x-y|-|t-s|| \ll 2^{l}
  |x'-y'|+ 2^{2l}|t-s| $ and $|(t-s)- (x-y) \cdot \omega|)| \ll 2^{l}
  |x'-y'|+ 2^{2l}|t-s| $ then we must have $\angle (x-y,\omega) \ll
  2^{l}$, which implies that $\angle (x-y,C) \approx 2^{l}$.  Then
  \eqref{off-angle} applies.

\end{proof}

\subsection{Fixed-time $L^2$ estimates}

The following is an important application of the previous theorem
which is at the heart of our parametrix construction: it proves the
$L^2$-part of \eqref{N_mapping}, \eqref{dtS_mapping} as well as that
of \eqref{S_mapping}.

\begin{proposition}
  The following fixed time $L^2$ estimates hold for functions
  localized at frequency $1$:
  \begin{align}
    e_{<k}^{\pm i\psi_{\pm}}(t,x,D):\quad & L^2 \rightarrow L^2,\, \label{l2}\\
    e_{<0}^{-i\psi_{\pm}}(t,x,D) e_{<0}^{i\psi_{\pm}}(D,y,s) -I :
    \quad & L^2\rightarrow \epsilon^{\frac{N-4}{N}} \log \epsilon \
    L^2
    \label{l2-ortho}\\
    \partial_{x,t} e^{\pm i\psi_{\pm}}_{<0}(t,x,D) : \quad & L^2 \to
    L^2 \label{dtl2}
  \end{align}

\end{proposition}
\begin{proof}
  a) By the estimate \eqref{wave-decay} with $s=t$, the $TT^*$ type
  operator
  \[
  e^{\pm i\psi_{\pm}}(t,x,D) P_0^2 e^{\mp i\psi_{\pm}}(D,y,t)
  \]
  has an integrable kernel, so it is $L^2$ bounded.  Therefore $e^{\pm
    i\psi_{\pm}}(t,x,D) P_0 $ and its adjoint are $L^2$ bounded. To
  accomodate symbol localizations we observe that
  \[
  e_{<k}^{\pm i\psi_{\pm}} = \int m_k(z)e^{\pm iT_z\psi_{\pm}}_{<0}
  \,dz
  \]
  where $m(z)$ is an integrable bump function on the $2^{-k}$ scale
  and $T_z$ denotes translation in the direction $z$, with $z$
  representing space-time coordinates. Since the wave equation is
  invariant to translations, the symbol $e^{\pm iT_z\psi_{\pm}}$ is of
  the same type as $e^{\pm i\psi_{\pm}}$ and its left and right
  quantizations are also $L^2$ bounded.  Thus the bound \eqref{l2}
  follows by integration with respect to $z$.

  b) For the estimate \eqref{l2-ortho} we note that the kernel of
  $e_{<0}^{-i\psi_{\pm}}(t,x,D) a(D) e_{<0}^{i\psi_{\pm}}(D,y,t) -
  a(D)$ is given by $K^a_{<0}(t,x,t,y) - \check a(x-y)$. Combining
  \eqref{wave-decay} and \eqref{almost-ortho} we get
  \[
  |K^a_{<0}(t,x,t,y) - \check a(x-y)| \lesssim \min \{ \epsilon |\log
  \epsilon|, |x-y|^{-N}\}
  \]
  The integral of the expression on the right is about
  $\epsilon^{\frac{N-4}{N}} |\log \epsilon|$, therefore the conclusion
  follows.

  c) By translation invariance we discard the $<0$ symbol
  localization, and show that $\partial_{x,t} e^{\pm
    i\psi_{\pm}}(t,x,D) P_0$ is $L^2$ bounded.  We have
  \[
\partial_{x,t} e^{\pm i\psi_{\pm}} =\pm \partial_{x,t}\psi_{\pm}
  e^{\pm i\psi_{\pm}}.
\]
 By \eqref{psi_decomp3} we have
  $\partial_{x,t}\psi_{\pm} \in DL^\infty(L^\infty)$ therefore we can
  dispose of it and use the $L^2$ boundedness of $e^{\pm
    i\psi_{\pm}}(t,x,D) P_0$.

\end{proof}

% ------------------------------------------------------------------------

\subsection{Modulation localized estimates}

The next order of business is to show that the fixed time $L^2$ bounds
for $e^{\pm i\Psi}_{<0}$ drastically improve to space-times $L^2(L^2)$
bounds if one selects a fixed frequency in the symbol. Precisely,

% breaks off a low frequency piece. This comes in two flavors, one for
% spatial frequencies:

% \begin{proposition}\label{cube_loc_prop}
%   For any $k\leqslant 0$ one has:
%   \begin{equation}
% 	\lp{(e^{\pm i \Psi}_{<0}-e^{\pm i \Psi}_{<k})P_0}
% 	{L^\infty(L^2)\to L^2(L^2)} 
% 	\lesssim \ 2^{-\frac{1}{2}(1+\delta)k}\epsilon \ , \label{cube_loc}
% \end{equation}
% \end{proposition}
% \noindent and the other having to do with modulations:
\begin{proposition}\label{mod_loc_prop}
  For $l\leqslant k'\pm O(1)$ one has the fixed frequency estimate:
  \begin{equation}
    \lp{Q_{l} e^{\pm i\Psi}_{k'}Q_{<0}P_0}
    {N^*\to X^{0,\frac{1}{2}}_1} \ \lesssim \  2^{\delta(l-k')}\epsilon
    \ . \label{kk'_mod_loc}
  \end{equation}
  In particular summing over all $(l,k')$ with $l\leqslant k$ and
  $k-O(1)\leqslant k'$ for a fixed $k\leqslant 0$ yields:
  \begin{equation}
    \lp{Q_{<k} (e^{\pm i\Psi}_{<0} - e^{\pm i\Psi}_{<k-C})Q_{<0}P_0}
    {N^*\to X^{0,\frac{1}{2}}_1} \ \lesssim \ \epsilon \ . \label{mod_loc}
  \end{equation}
\end{proposition}

A key step in the proof of the proposition is the following result,
which we state separately since it is of independent interest:

\begin{lemma}
  Assume that $1\leqslant q \leqslant p \leqslant \infty$. Then for
  $k+C \leq l \leq 0$ we have :
  \begin{equation}
    \lp{( e^{\pm i\psi_{<k}}_{l})(t,x;D)}{L^{p}(L^2)\to L^q(L^2)} 
    \lesssim\ 
    \epsilon 2^{(\frac{1}{p}-\frac{1}{q})k}2^{5(k-l)}
    \ , \label{remainder_est}
  \end{equation}
  This holds for both left and right quantizations.
\end{lemma}
\begin{proof}

  For the symbol we iteratively write:
  \begin{align}
    S_{l} e^{\pm i \psi_{<k}} \ &= \ \pm i 2^{-l}S_{l}
    \big(\partial_t\psi_{<k}\cdot e^{\pm i \psi_{<k}}\big) \ , \notag\\
    &= \ \ldots \ = \ (\pm i )^5 2^{-5l}\prod_{j=1}^5
    [S^{(j)}_{l}\partial_t \psi_{<k}]\cdot e^{\pm i \psi_{<k}} \ ,
    \notag
  \end{align}
  where the product denotes a nested (repeated) application of
  multiplication by $S_{l}\partial_t \psi_{<k}$, for a series of
  frequency cutoffs $S^{(j+1)}_{l}S^{(j)}_{l}=S^{(j)}_{l} \approx
  S_{l}$ with expanding widths.  Disposing of these translation
  invariant cutoffs we see that \eqref{remainder_est} follows directly
  from \eqref{psi_decomp1}.

\end{proof}

\begin{proof}[Proof of Proposition \ref{mod_loc_prop}]
  We proceed in a series of steps.

  \step{1}{High modulation input} First we estimate the contribution
  of $Q_{k}e^{\pm i \Psi}_{k'}Q_{\geqslant k-C}P_0$ to line
  \eqref{kk'_mod_loc}. Using the $X_\infty^{0,\frac{1}{2}}$ bounds for
  the input, it suffices to prove the estimate:
  \begin{equation}
    \lp{Q_k e^{\pm i \Psi}_{k'}P_0}{L^2(L^2)\to L^2(L^2)} \ \lesssim \ 
    2^{\frac{1}{5}(k-k')}\epsilon \ . \notag
  \end{equation}
  By Sobolev estimates in $|\tau|\pm |\xi|$, this reduces to the
  bound:
  \begin{equation}
    \lp{e^{\pm i \Psi}_{k'}P_0}{L^2(L^2)\to L^\frac{10}{7}(L^2)} \ \lesssim \ 
    2^{-\frac{1}{5}k'}\epsilon \ . \notag
  \end{equation}
  Using continuous Littlewood-Paley resolutions to decompose the group
  element we have:
  \begin{align}
    e^{\pm i \Psi}_{k'} \ &= \ \pm i \int_{k''>k'-C} S_{k'}
    (\Psi_{k''}e^{\pm i \Psi_{<k''}})dk''
    + S_{k'}e^{\pm i\Psi_{<k'-C}} \ , \notag\\
    &= \ \mathcal{L} + \mathcal{R} \ . \label{one_iterate}
  \end{align}
  We'll treat these two terms separately.  In the case of
  $\mathcal{L}$ we use estimate \eqref{psi_decomp1} with $p=5$.  To
  estimate the contribution of $\mathcal{R}$ we use estimate
  \eqref{remainder_est}.
\bigskip

  \step{2}{Main decomposition for low modulation input} Now we
  estimate the expression $Q_{k}e^{\pm i \Psi}_{k'}Q_{< k-C}P_0u$.
  First expand the untruncated group elements as follows:
  \begin{align}
    e^{\pm i \Psi} \ &= \ e^{\pm i \Psi_{<k-C}} \pm i \int_{l>k-C}
    \Psi_l e^{\pm i \Psi_{<k-C}}dl -
    \dint_{l,l'>k-C}\hspace{-.25in}\Psi_l\Psi_{l'} e^{\pm i \Psi_{<k-C}}dl dl'\notag\\
    &\ \ \ \ \ \ \ \ \mp i \tint_{l,l',l''>k-C}\hspace{-.5in}
    \Psi_l\Psi_{l'} \Psi_{l''}
    e^{\pm i \Psi_{<l'''}}dl dl' dl'' \ , \notag\\
    &= \ \mathcal{Z} + \mathcal{L} + \mathcal{Q} + \mathcal{C} \
    . \notag
  \end{align}
  We will estimate the effect of each of these terms separately.

  \step{3}{Estimating the zero order term $\mathcal{Z}$} After
  localization via $S_{k'}$ the desired estimate follows directly from
  \eqref{remainder_est}.
\bigskip

  \step{4}{Estimating the linear term $\mathcal{L}$} First split the
  integral as follows:
  \begin{equation}
    \mathcal{L} \ = \ \int_{l>k-C}  \Psi_l S_{<k-\frac{1}{2}C}
    e^{\pm i \Psi_{<k-C}}dl + \int_{l>k-C}  \Psi_l S_{\geqslant k-\frac{1}{2}C}
    e^{\pm i \Psi_{<k-C}}dl \ . \label{Lsplit}
  \end{equation}
\bigskip

  \step{4a}{Estimating the principal linear term in $\mathcal{L}$} For
  the first term on RHS of line \eqref{Lsplit} it suffices to show the
  general estimate:
  \begin{multline}
    \lp{Q_{k}\cdot( \psi_{l}b_{<k-\frac{1}{2}C})(t,x;D)\cdot
      Q_{<k-C}P_0 }{L^\infty(L^2)\to L^2(L^2)} \\ \lesssim \
    \epsilon\, 2^{-\frac{1}{2}k}2^{\frac{1}{4}(k-l)} \sup_t
    \lp{B_{<k-\frac{1}{2}C}(t)}{L^2\to L^2} \ , \label{lin_psi_loc}
  \end{multline}
  for $l\geqslant k$, and for symbols $b(x,\xi)_{<k-\frac{1}{2}C}$
  with either the left or right quantization.  In this case the
  modulation of the output determines the angle between the spatial
  frequencies of $\psi_l(x,\xi)$ and the spatial frequency of the
  input, which is $\theta\sim 2^{\frac{1}{2}(k-l)}$. Since this is
  also the angle with $\xi$, we may restrict the symbol of $\Psi_l$ to
  $\sum_{\theta\gtrsim 2^{\frac{1}{2}(k-l)}} \psi_l^{(\theta)}$ for
  which the estimate \eqref{lin_psi_loc} follows immediately from
  \eqref{decomp1} and summing
  over \eqref{psi_decomp2}.
\bigskip
 
  \step{4b}{Estimating the frequency truncation error in
    $\mathcal{L}$} For the second term on RHS of line \eqref{Lsplit}
  we use \eqref{psi_decomp1} for $\psi_l$ with $p=6$ combined with
  \eqref{remainder_est} with $(p_2,q)=(6,3)$. If $l<k'-C$ the $S_{k'}$
  localization lands on the exponential and we gain
  $2^{-\frac{1}{2}k}2^{\frac{1}{6}(k-l)}2^{5(k-k')}$.  If $l\geqslant
  k'-C$ we can disregard $S_{k'}$ and directly get
  $2^{-\frac{1}{2}k}2^{\frac{1}{6}(k-l)}$.
  \bigskip

  \step{5}{Estimating the quadratic term $\mathcal{Q}$} We follow a
  similar procedure to \textbf{Step 4} above. First split $S_{<
    k-\frac{1}{2}C}e^{\pm i \psi_{<k-C}} + S_{\geqslant
    k-\frac{1}{2}C}e^{\pm i \psi_{<k-C}}$. For the second term one can
  proceed as in \textbf{Step 4b} above using \eqref{psi_decomp1},
  \eqref{remainder_est}, and \eqref{decomp_holder}.  Therefore we only
  need to consider the effect of the first term, for which we'll show
  the trilinear bound:
  \begin{multline}
    \lp{Q_{k}\cdot( \psi_{l}\psi_{l'}b_{<k-\frac{1}{2}C})(t,x;D)\cdot
      Q_{<k-C}P_0 }{L^\infty(L^2)\to L^2(L^2)} \\ \lesssim \
    \epsilon^2\,
    2^{-\frac{1}{2}k}2^{\frac{1}{4}(k-l)}2^{\frac{1}{6}(k-l')} \sup_t
    \lp{B_{<k-\frac{1}{2}C}(t)}{L^2\to L^2} \ , \label{quad_psi_loc}
  \end{multline}
  for $l'\geqslant l\geqslant k$. By the localization of the output,
  the symbol $\psi_l\psi_{l'}$ is localized to the sum:
  \begin{equation}
    \sum_{\theta\gtrsim 2^{\frac{1}{2}(k-l)}} 
    \psi_l^{(\theta)}\psi_{l'}
    + \sum_{\substack{\theta'\gtrsim 2^{\frac{1}{2}(k-l')}\\
        \theta\ll 2^{\frac{1}{2}(k-l')} }} 
    \psi_l^{(\theta)}\psi_{l'}^{(\theta')} \ = \ T_1 + T_2 \ . \notag
  \end{equation}
  For the term $T_1$ put the first factor in $DL^3(L^\infty)$ and the
  second in $DL^6(L^\infty)$. This gives us dyadic terms in
  LHS\eqref{quad_psi_loc}$(T_1)\sim
  2^{-\frac{1}{2}k}2^{\frac{1}{4}(k-l)}2^{\frac{1}{6}(k-l')}$.  For
  the term $T_2$ do the opposite, which yields
  a similar bound.
\bigskip

  \step{6}{Estimating the cubic term $\mathcal{C}$} In this case we
  can gain $2^{\frac{1}{6}(k-k')}$ directly through the use of
  \eqref{psi_decomp1} and three $DL^6(L^\infty)$ unless
  $l,l',l''<k'-C$. In the latter case $S_{k'}$ localizes the
  exponential and we can use a \eqref{remainder_est} instead. Further
  details are left to the reader.

\end{proof}

% ------------------------------------------------------------------------

\section{Proof of the $N_0\to N_0$ bounds in \eqref{N_mapping} and
  \eqref{dtS_mapping}}
\label{s:N-error}

Because of use in the sequel, we'll prove a somewhat more general and
symmetric version:

\begin{prop}[Symmetric $N$ Estimates]\label{symm_N_prop}
  One has the following operator bounds for either right or left
  quantizations:
  \begin{equation}
    \lp{e^{\pm i\Psi}_{<0}}{N\to N} \ \lesssim \ 1 \ . \label{symm_NN}
  \end{equation}
  as well as
  \begin{equation}
    \lp{\partial_t e^{\pm i\Psi}_{<0}}{N\to N} \ \lesssim \ 1 \ . \label{symm_NN-dt}
  \end{equation}
  In particular, by duality one also has:
  \begin{equation}
    \lp{e^{\pm i\Psi}_{<0}}{N^*\to N^*} \ \lesssim \ 1 \ , \label{dual_NN}
  \end{equation}
  for either right or left quantizations.
\end{prop}

\begin{proof}
  Applying $e^{\pm i\Psi}_{<0}$ to $F\in N_0$ we need to consider the
  following cases:

  \case{1}{$F$ is an $L^1(L^2)$ atom} This follows at once from the
  fixed
  time $L^2\to L^2$ mapping property of $e^{\pm i\Psi}_{<0}$.
\bigskip

  \case{2}{$F$ is an $X^{0,-\frac{1}{2}}_1$ atom} Fixing a modulation
  it suffices to consider $Q_k P_0 F$, and show that:
  \begin{equation}
    \lp{e^{ \pm i\Psi}_{<0}Q_k P_0 F}{N} \ \lesssim \ 2^{-\frac{1}{2}k}
    \lp{F}{L^2(L^2)} \ , \notag
  \end{equation} 
  which by duality reduces to showing:
  \begin{equation}
    \lp{Q_k e^{\pm i\Psi}_{<0} \td{P}_0 G}{L^2(L^2)} \ \lesssim \ 2^{-\frac{1}{2}k}
    \lp{G}{N^*} \ , \notag
  \end{equation} 
  for a slightly larger cutoff $\td{P}_0$. At this point the estimate
  follows the usual game of splitting $ e^{\pm i\Psi}_{<0}= e^{\pm
    i\Psi}_{<k-C} + (e^{\pm i\Psi}_{<0}- e^{\pm i\Psi}_{<k-C})$.  For
  the first term there is no modulation interference and the estimate
  is direct.  For the second use \eqref{mod_loc}.
\end{proof}

\section{Proof of the Renormalization Error Estimates}

% ------------------------------------------------------------------------

\subsection{The $N$ and $N^*$ Space Recovery Estimate
  \eqref{S_mapping}}

We prove \eqref{S_mapping} for $N^*$. The $L^\infty L^2$ part of the
estimate \eqref{S_mapping} for $N^*$ is a direct consequence of the
similar fixed time $L^2$ bound \eqref{l2-ortho}.  It remains to
consider the $X^{0,\frac12}_\infty$ estimate. Denote
\[
R_k = e^{- i\psi_\pm}_{<k}(x,D)e^{i\psi_\pm}_{<k}(D,y)
\]
Fixing the modulation we write:
\[
Q_k (R_0-I) = Q_k (R_0-I) Q_{>k-C} + Q_k R_0 Q_{<k-C}
\]
For the first term we bound $Q_{>k-C}$ in $L^2$ and use again
\eqref{l2-ortho}.  For the second we can freely subtract the low
frequencies in the symbol
\[
\begin{split}
  Q_k R_0 Q_{<k-C} = & \ Q_k (R_0 - R_{<k-C}) Q_{<k-C} \\ = & \
  Q_k(e^{- i\psi_\pm}_{<0}(x,D)- e^{i\psi_\pm}_{<k-C}(x,D))
  e^{i\Psi_\pm}_{< 0}(D,y) Q_{<k-C} \\ & \ + Q_k
  e^{-i\psi_\pm}_{<k-C}(x,D)
  (e^{i\psi_\pm}_{<0}(D,y)-e^{i\psi_\pm}_{<k-C}(D,y))Q_{<k-C}
\end{split}
\]
We estimate the two lines separately as operators from $N^*$ to
$X^{0,\frac12}_\infty$. In the first line we can use \eqref{dual_NN}
to discard $e^{i\Psi_\pm}_{< 0}(y,D) Q_{<k-C}$, and then
\eqref{mod_loc} for the remaining factor.  In the second line we first
write
\[
Q_k e^{-i\psi_\pm}_{<k-C}(x,D) = Q_k e^{-i\psi_\pm}_{<k-C}(x,D) \tilde
Q_k
\]
and use \eqref{l2} to discard the two left factors. Then we are
reduced again to \eqref{mod_loc}.

% ------------------------------------------------------------------------

\subsection{Proof of the conjugation estimate \eqref{error_ests}}

Recall the following product formula for left quantizations:
\begin{equation}
  a(x,D)b(x,D) \ = \ c(x,D) \ ,
  \quad c(x,\xi) \ \sim \ e^{-i\langle D_y,D_\xi\rangle }\big(
  a(x,\xi)b(y,\eta)\big)|_{\substack{x=y\\ \xi=\eta}} \ . \notag
\end{equation}
We are interested in the case where the first factor is the
d'Alembertian.  Then the above formula contains only finitely many
terms and is exact and we'll denote it by $\#_l$.  Using this to
compute both terms on line \eqref{error_ests} we have:
\begin{equation}
  \begin{split}
    (-\partial_t^2 - |\xi|^2)\#_l e^{-i\psi_\pm}_{<0} = &
    e^{-i\psi_\pm}_{<0} (-\partial_t^2 - |\xi|^2) + 2 [i\partial_t
    e^{-i\psi_\pm}_{<0}
    \partial_t - \partial_x e^{-i\psi_\pm}_{<0}\xi ] - \Box
    e^{-i\psi}_{<0} \\ = & e^{-i\psi_\pm}_{<0} (-\partial_t^2 -
    |\xi|^2) + 2 \partial_t e^{-i\psi_\pm}_{<0} (i \partial_t \pm
    |\xi|) \\ &- 2i [(\pm \partial_t \psi_\pm |\xi|-\partial_x\psi_\pm
    \xi)e^{-i\psi_\pm}]_{<0}
    \\
    & -[(|\partial_t \psi_{\pm}|^2 - |\partial_x
    \psi_{\pm}|^2)e^{-i\psi_\pm}]_{<0} \ , \notag
  \end{split}
\end{equation}
On the other hand we can write
\[
i A^j \xi_j \#_l e^{-i\psi_\pm}_{<0} = i [A^j
e^{-i\psi_\pm}]_{<0}\xi_j + A^j \partial_j e^{-i\psi_\pm}_{<0} + i
[A_j,S_{<0}] e^{-i\psi_\pm}\xi_j
\]
Hence combining all terms we can write down the exact symbol for the
conjugation operator
\[
\begin{split}
  \Diff := & \ (-\partial_t^2 - |\xi|^2 + 2 i A^j \xi_j)\#_l
  e^{-i\psi_\pm}_{<0} - e^{-i\psi_\pm}_{<0} (-\partial_t^2 - |\xi|^2)
  \\ = & \ - 2i [(\pm \partial_t \psi_\pm |\xi|-\partial_x\psi_\pm \xi
  + A^j \xi_j )e^{-i\psi_\pm}]_{<0} \\ & \ + 2 \partial_t
  e^{-i\psi_\pm}_{<0} (i \partial_t \pm |\xi|) \\ & \ -[(|\partial_t
  \psi_{\pm}|^2 - |\partial_x \psi_{\pm}|^2)e^{-i\psi_\pm}]_{<0} \\ &
  \ -2i A^j [\partial_j \psi_\pm e^{-i\psi_\pm}]_{<0} \\ &\ + 2i
  [A_j,S_{<0}] e^{-i\psi_\pm}\xi_j \\ := &\ \Diff_1 + \Diff_2 +
  \Diff_3 + \Diff_4 +\Diff_5
\end{split}
\]
It remains to estimate all five terms as pseudodifferential operators
from $S^\sharp_\pm \subset N^*_\pm$ to $N_\pm$.

{\bf The estimate for $\Diff_1$:} We recall that $\psi_{\pm}$ was
chosen exactly so that the high angle interaction part of $\Diff_1$
cancels. Using the definition of $\psi_{\pm}$ we rewrite this term as
\[
\Diff_1 = [(I-\Pi^\sigma)A \cdot \xi )e^{-i\psi_\pm}]_{<0}
\]
We can replace this by
\[
\Diff_1 = [(I-\Pi^\sigma)A \cdot \xi )e^{-i\psi_\pm}_{<0}]_{<0}
\]
where the inner $<0$ truncation is on a slightly larger ball.  By
translation invariance we can discard the outer $<0$ truncation.  Also
here we do not need the $S^\sharp$ structure, and it suffices to work
with $N^*$.  Then we can also use the bound \eqref{dual_NN} to discard
$e^{-i\psi_\pm}_{<0}$. Hence we are left with having to prove a
bilinear estimate, namely
\begin{equation}
  (I-\Pi^\sigma)A^j \partial_j : N^* \to N 
\end{equation}
After a frequency and angular expansion of $A_j$, we obtain a sum
\[
\sum_{k <0} \sum_{l < \sigma k} \sum_{\omega} P^\omega_l A^j_k
P^{\omega'}_l \partial_j u_0
\]
where the multipliers $ P^\omega_l$ and $P^{\omega'}_l$ select $2^{l}$
sectors which are $2^l$ separated. To estimate the summands in $N^*$
we split the into low and high modulation,
\[
P^\omega_l A^j_k P^{\omega'}_l u_0 = P^\omega_l A^j_k P^{\omega'}_l
Q_{>k+2l-C} \partial_j u_0 + P^\omega_l A^j_k P^{\omega'}_l
Q_{<k+2l-C} \partial_j u_0
\]
For the first term on the right, we place it into $L_t^1 L_x^2$ by
using $L_t^2 L_x^\infty$ for the factor $ P^\omega_l A^j_k$ and
$l_{t,x}^2$ for the second factor $Q_{>k+2l-C} \partial_j u_0$. The
Coulomb condition results in a gain $2^l$ (due to the operator
$\partial_j$), and Bernstein's inequality via $P_k
P^\omega_lL_t^2L_x^6\subset L_t^2L_x^\infty$ yields a gain
$2^{\frac{k+l}{2}}$, whence recalling the definition of $N^*$, we
obtain a net gain of $2^{\frac{l}{2}}$, which suffices for summation
both over $l$ and $k$ due to the condition $l<\sigma k$.

For the second term on the right, we place the output into
$\dot{X}_1^{0,-\frac12}$, which is possible since its modulation is of
size $\sim 2^{k+2l}$ from elementary geometry. Placing the second
input $\partial_j u_0$, the numerology is the same as for the first
term.

{\bf The estimate for $\Diff_2$:} By the definition of
$S^\sharp_{\pm}$ we have
\[
i \partial_t \pm |D|: S^\sharp_{\pm}: N
\]
Then we can use directly the bound \eqref{symm_NN-dt}.

{\bf The estimates for $\Diff_3$ and $\Diff_4$:} Here we use
decomposability. Precisely, by \eqref{psi_decomp2} we have
$\partial_{t,x} \psi_{\pm} \in DL^2(L^\infty)$. It is also easy to see
that $A$ belongs to the same space. Then by \eqref{decomp_holder} we
can place the output into the energy space $L_t^1 L_x^2$.

{\bf The estimate for $\Diff_5$:} This term is handled exactly like
$\Diff_1$; this time we do not have the small angular separation
condition between the factors, but summation over $k$ becomes possible
since we essentially gain an extra low-frequency derivative:
\[ [A_j,S_{<0}]\partial_j u_0 = L(\nabla_{t,x}A^j, \partial_j u_0)
\]
where $L$ is a translation invariant bilinear operator with integrable
kernel.

%%%%%%%%%%%%%%%%%%%%%%%%%%%%%
% ------------------------------------------------------------------------
% ------------------------------------------------------------------------
%%%%%%%%%%%%%%%%%%%%%%%%%%%%%

\section{Proof of the Dispersive Estimates}

In this section we prove the mapping property \eqref{parametrix_bound}
in a series of steps.

% ------------------------------------------------------------------------

\subsection{Proof of the basic Strichartz estimates}

First we show:
\begin{equation}
  e^{-i\psi_{\pm}}_{<0}(t,x,D): S_{0,\pm}^\sharp
  \  \to \ S^{str} \ . \label{str_ests}
\end{equation}
The first step is to reduce the problem to the case of homogeneous
waves. This is done by foliating with respect to $t$ for the $L^1 L^2$
part, and with respect to $\{\tau - |\xi|= \sigma\}$ slices for the
$X^{0,-\frac12}_1$ part,
\begin{equation}
  u^{\pm}(t) \ = \ \int e^{it\sigma} e^{\pm i t |D|}P_0 f^{\pm}(\tau)d\mu^\pm (\tau) \ ,
  \quad
  \int \lp{f^\pm(\tau)}{L^2}d\mu^\pm(\tau) \ \lesssim \ \lp{F}{X^{0,-\frac12}_1} \ . \notag  
\end{equation}
Then it remains to show that
\[
e^{-i\psi_{\pm}}_{<0}(t,x,D) e^{\pm it|D|} P_0 : L^2 \to S^{str}
\]
By a $TT^*$ argument this is equivalent to
\[
e^{-i\psi_{\pm}}_{<0}(t,x,D) e^{\pm i(t-s)|D|} P_0
e^{i\psi_{\pm}}_{<0}(D,y,s): L^{p'} L^{q'} \to L^p L^q
\]
for all pairs of Strichartz exponents $(p,q)$.  In the non-endpoint
case $p > 2$ his follows in a standard manner from an $L^2 \to L^2$
fixed time bound (see \eqref{l2})and and the dispersive estimate
\eqref{wave-decay0}. In the endpoint case $L^2 L^6$ we can use
Theorem~4 in \cite{Tat-c1}, which asserts that the endpoint case
follows from the non-endpoint case plus the dispersive estimate.

% ------------------------------------------------------------------------

\subsection{Proof of the square summed Strichartz
  estimates}\label{ang_str_sect}

Fixing an angular scale $2^l$ and the corresponding modulation scale
$2^{2l}$ as well as a rectangle scale $2^{k'} \times (2^{k'+l'})^3$
with $k'\leq k,l' \leq 0$ and $2l \leq k'+l' \leq l$ we need to show
that
\[
\sum_{C \in \mathcal C_{k',l'}} \|P_C Q_{<2l}
e^{-i\psi_{\pm}}_{<0}(t,x,D) u \|_{S^{str}}^2 \lesssim \| u
\|_{S^\sharp}^2
\]
For this we write
\[
Q_{<2l} e^{-i\psi_{\pm}}_{<0}(t,x,D) u \ = \
Q_{<2l}(e^{-i\psi_\pm}_{<0} - e^{-i\psi_{\pm}}_{<2l-C})u + Q_{<2l}
e^{-i\psi_\pm}_{<2l-C} Q_{<2l+C }u \
\]
The first term is estimated in $X^{0,\frac12}_1$ using
\eqref{mod_loc}.  For the second one, the symbol of
$e^{-i\psi_\pm}_{<2l-C}$ is strongly localized so that non adiacent
rectangles do not interact. Then the square summation with respect to
cubes is inherited from $S^\sharp$, and it remains to prove that for a
single cube $C$ we have
\begin{equation}\label{Cstr}
  \|  e^{-i\psi_{\pm}}_{<2l-C}(t,x,D)P_C u \|_{S^{str}} \lesssim 
  \| u \|_{S^\sharp_{\pm}}
\end{equation}
We can freely discard $P_C$. Then the last bound follows by
translation invariance from \eqref{str_ests}.

\subsection{Proof of the square summed $L^2 L^\infty$ estimates}

The setting and the argument is identical to the one above up to the
counterpart of \eqref{Cstr}, which now reads
\begin{equation}\label{Cstr2i}
  \|  e^{-i\psi_{\pm}}_{<2l-C}(t,x,D) Q_{<2l+C}  P_C u \|_{L^2 L^\infty} \lesssim 
  2^{k'} \| u \|_{S^\sharp_{\pm}}
\end{equation}
We can no longer discard $P_C$ due to the presence of the scale factor
$2^{k'}$.  Instead we repeat the argument in the proof of
\eqref{str_ests}.  irs the problem reduces to an estimate for free
waves,
\[
\| e^{-i\psi_{\pm}}_{<2l-C}(t,x,D) e^{\pm it|D|} P_C u \|_{L^2
  L^\infty} \lesssim 2^{k'} \| u_0 \|_{L^2_x}
\]
Then by a $TT^*$ argument this is equivalent to
\[
\| e^{-i\psi_{\pm}}_{<2l-C}(t,x,D) e^{\pm i(t-s)|D|}
P_Ce^{i\psi_{\pm}}_{<0}(D,y,s) f \|_{L^2 L^\infty} \lesssim 2^{2k'} \|
f \|_{L^2 L^1}
\]
This in turns follows from the dispersive estimate
\eqref{wave-decay-cube0}, which shows that kernel of the operator
above is bounded by
\[
2^{4k'+3l'} \langle 2^{2(k'+l')} (t-s) \rangle^{-\frac32}
\]
which integrates to $2^{2k'+l'} \leq 2^{2k'}$.

% ------------------------------------------------------------------------

\subsection{Proof of the plane wave norm estimates}

Again using the nesting \eqref{basic_norms}, at modulation $<2^{2l}$
it suffices to consider expressions involving $Q_{<2l}
e^{-i\psi_\pm}_{<2l-C}(t,x,D)$. By foliating as in the proof of
\eqref{str_ests}, we further reduce considerations to expressions $
e^{-i\psi_\pm}_{<2l-C}(t,x,D)e^{\pm it|D|}P_0 f$ for
$\lp{f}{L^2}\leqslant 1$. Localizing $e^{i\Psi}_{<k-C}e^{\pm it|D|}P_0
f$ by $P_\mathcal{C}$, we employ polar coordinates in Fourier space to
write:
\begin{equation}
  e^{-i\psi_\pm}_{<2l-C}(t,x,D)e^{\pm it|D|}P_{\mathcal{C}}f \ = \ 
  \frac{1}{(2\pi)^4}
  \int_{\mathbb{S}^3} e^{-i\psi(t,x;\omega)_\pm}_{<2l-C}f^\omega_\pm (t,x)d\omega
  \ , \notag
\end{equation}
where:
\begin{equation}
  f^\omega_\pm(t,x) \ = \ 
  \int e^{i(\pm t + x\cdot\omega)\lambda} p_{\mathcal{C}}
  (\lambda\omega )\widehat{f}(\lambda \omega) 
  \lambda ^3 d\lambda \ . \notag
\end{equation}
Then computation of the square sum $P\!W^\mp$ norms follows
immediately from these expressions, H\"older's inequality, and the
Plancherel's theorem.

\subsection{Proof of the null energy estimates}

As in the case of the square summed Strichartz estimates, the problem
reduces to proving that
\[
e^{-i\psi_{\pm}}_{<0}(t,x,D) : S^\sharp \to NE
\]
We fix a null direction $\omega$ and prove the null energy estimate in
its associated frame,
\[
\snabla e^{-i\psi_{\pm}}_{<0}(t,x,D) : S^\sharp \to L^\infty_\omega
L^2_{\omega^\perp}
\]
For simplicity we take $\omega = (1,1,0,0,0)$. The symbol of its
associated tangential gradient $\snabla$ then consists of $\tau -
\xi_1$ and $\xi'$.  Depending on whether we are near the $+$ cone or
near the $-$ cone we can replace
\[
\tau-\xi_1 = (\tau \pm |\xi|) - (\xi_1 \pm |\xi|)
\]
The contribution of the first term to the output is easily estimated
via the $N^*$ norm in $L^2$ and then in $L^\infty_\omega
L^2_{\omega^\perp}$ by Bernstein. It remains to bound
\[
(D_1 \pm |D|, D') e^{-i\psi_{\pm}}_{<0}(t,x,D) : S^\sharp_{\pm} \to
L^\infty_\omega L^2_{\omega^\perp}
\]
We fix the signs to $+$.  Since no time multipliers are involved any
more in the above operator, the reduction to the case of homogeneous
waves still applies, and we are left with proving that
\[
(D_1 \pm |D|, D') e^{-i\psi_{\pm}}_{<0}(t,x,D) e^{it|D|}: L^2 \to
L^\infty_\omega L^2_{\omega^\perp}
\]
The obvious course of action at this point is to use disposability
arguments to commute the derivatives inside, and then use a $TT^*$
argument. This unfortunately borderline fails due to the limited decay
of the wave kernel along null cones. Instead we take an extra step,
and perform a dyadic decomposition with respect to the angle between
$\xi$ and $\omega'= (1,0,0,0)$.  There are two cases we need to
consider separately.

{\bf For large $O(1)$ angles} we can simply discard the multipliers,
and show that
\[
P_{1}^{\omega'} e^{-i\psi_{\pm}}_{<0}(t,x,D) e^{it|D|}: L^2 \to
L^\infty_\omega L^2_{\omega^\perp}
\]
For this we commute $P_{1}^{\omega'}$ inside, using disposability to
estimate the error in $L^2$. Then, by translation invariance, it
remains to show that
\[
e^{-i\psi_{\pm}}_{<0}(t,x,D) e^{it|D|}P_{1}^{\omega'} : L^2 \to L^2(
\Sigma), \qquad \Sigma = \{ t = \omega' \cdot x\}
\]
By a $TT^*$ argument this is reduced to an estimate of the form
\[
e^{-i\psi_{\pm}}_{<0}(t,x,D) e^{i(t-s)|D|}P_{1}^{\omega'}
e^{i\psi_{\pm}}_{<0}(D,y,s): L^2( \Sigma) \to L^2(\Sigma)
\]
Due to the angular separation, we can use \eqref{off-angle} with $l=0$
to conclude that the kernel of the above operator decays rapidly off
diagonal, and the $L^2$ boundedness easily follows.

{\bf For small $\ll 1$ angles} we can write $D_1 \pm |D| = a(D) D'$
and then discard $a$, and we are left with proving that
\[
D' e^{-i\psi_{\pm}}_{<0}(t,x,D) e^{it|D|}: L^2 \to L^\infty_\omega
L^2_{\omega^\perp}
\]
The advantage of this is that multipliers which depend only on $D'$
are compatible with the $L^2_{\omega^\perp}$. Hence via a dyadic
decomposition in $D'$, we need to show that
\[
2^k P_k (D') e^{-i\psi_{\pm}}_{<0}(t,x,D) e^{it|D|}: L^2 \to l^2_k
L^2( \Sigma), \qquad k < -C
\]
We commute the multiplier inside,
\[
\begin{split}
2^k P_k (D') e^{-i\psi_{\pm}}_{<0}(t,x,D) e^{it|D|} = &\  \tilde P_k(D')
[2^k P_k (D'), e^{-i\psi_{\pm}}_{<0}(t,x,D)] e^{it|D|}  \\ &+  \tilde
P_k(D') e^{-i\psi_{\pm}}_{<0}(t,x,D) e^{it|D|}P_k (D')
\end{split}
\]
By the commutator estimate \eqref{spec_decomp1} the first term
satisfies the same bounds as the operator
\[
\tilde P_k(D') (\nabla' e^{-i\psi_{\pm}}_{<0})(t,x,D) e^{it|D|} = - i
\tilde P_k(D') (\nabla' \psi_{\pm} e^{-i\psi_{\pm}})_{<0}(t,x,D)
e^{it|D|}
\]
Since $\nabla \psi_{\pm} \in DL^2 L^r$ for some $r < \infty$, by
disposability we have
\[
(\nabla' \psi_{\pm} e^{-i\psi_{\pm}})_{<0}(t,x,D): L^\infty L^2 \to
L^2 L^{\frac{2r}{r+2}}
\]
Hence by Bernstein,
\[
\| \tilde P_k(D') (\nabla' e^{-i\psi_{\pm}}_{<0})(t,x,D)
e^{it|D|}\|_{L^2_x to L^2_{x,t} } \lesssim 2^{-\frac{4}{r} k}
\]
where the point is that there is some gain in $k$ in order to
guarantee summability.

It remains to show that
\[
\tilde P_k(D') e^{-i\psi_{\pm}}_{<0}(t,x,D) e^{it|D|}P_k (D'): L^2 \to
l^2_k L^2(\Sigma)
\]
We harmlessly discard $\tilde P_k(D')$. The square summability with
respect to $k$ is now provided by the multiplier on the right, so it
suffices to fix $k$.  By a $TT^*$ argument this reduces to
\[
e^{-i\psi_{\pm}}_{<0}(t,x,D) e^{i(t-s)|D|} 2^{2k} P_k^2
e^{i\psi_{\pm}}_{<0}(D,y,s): L^2( \Sigma) \to L^2(\Sigma)
\]
For the above operator we use the kernel bound provided by
\eqref{off-angle0} with $l = k$, namely
\[
2^{5l} \langle 2^{2l} |t-s| \rangle^{-N} \langle 2^l |x-y|
\rangle^{-N}
\]
This is an integrable kernel, so the $L^2 \to L^2$ bound easily
follows.

% -------------------------------------------------------------------------
%%%%%%%%%%%%%%%%%%%%%%%%%%%%%
% -------------------------------------------------------------------------

\section{Statements and proofs
 of the core multilinear estimates}

It remains to prove \eqref{Ai_null_est}, \eqref{Ai_null_Best},
\eqref{A0B_est2}, \eqref{quadphi_HL_null},
\eqref{gen_phi_lohi_quadest}, \eqref{gen_phi_Best},
\eqref{gen_phi_lohi_quadest0}, \eqref{quadphi_HL_Z0}, and finally
\eqref{quad_null_est} for each of the expressions
\eqref{quad_null1}--\eqref{quad_null3}.  This set of estimates can be
roughly grouped into two rough categories. The first involves
``null-form'' type estimates, and the second involves product
estimates without a microlocal gain from angular separation. The first
category of bounds boils down to:

\begin{thm}[Core null form  estimates]\label{null_thm}
  The following hold:
  \begin{align}
    &\lp{ P_{k_1}\mathcal{N}(\phi^{(2)}_{k_2}, \phi^{(3)}_{k_3}) }{N}
    \ \lesssim \ 2^{k_1}2^{\delta(k_1-\max\{k_2,k_3\})}
    2^{-\delta|k_2-k_3|}
    \lp{ \phi^{(2)}_{k_2} }{S^1}\lp{ \phi^{(3)}_{k_3} }{S^1} \ , \label{mult1}\\
    &\lp{
      (I-\mathcal{H}^*_{k_1})\mathcal{N}(\phi^{(1)}_{k_1},\phi^{(2)}_{k_2})
    }{N} \lesssim 2^{k_1}
    \lp{\phi^{(1)}_{k_1}}{S^1}\lp{\phi^{(2)}_{k_2}}{S^1} , \ k_1<k_2-C
    \ ,
    \  \label{mult2}\\
    &\lp{ \mathcal{H}^*_{k_1}
      \mathcal{N}(\phi^{(1)}_{k_1},\phi^{(2)}_{k_2}) }{L^1(L^2)}
    \lesssim 2^{k_1}
    \lp{\phi^{(1)}_{k_1}}{Z^{hyp}}\lp{\phi^{(2)}_{k_2}}{S^1} , \
    k_1<k_2-C \ ,
    \  \label{mult3}\\
    &\lp{(I-\mathcal{H}_{k_1})P_{k_1}
      \mathcal{N}(\phi^{(2)}_{k_2},\phi^{(3)}_{k_3}) }{\Box Z^{hyp}}
    \label{mult4}\\
    &\hspace{.4in} \lesssim 2^{k_1}2^{\delta(k_1-\max\{k_2,k_3\})}
    2^{-\delta|k_2-k_3|}
    \lp{\phi^{(2)}_{k_2}}{S^1}\lp{\phi^{(3)}_{k_3}}{S^1} \ ,  \notag\\
    &\lp{\mathcal{H}_{k_1}\mathcal{N}
      (\phi^{(2)}_{k_2},\phi^{(3)}_{k_3}) }{\Box Z^{hyp}}\label{mult5}\\
    &\hspace{.4in}\lesssim \ 2^{k_1}2^{-\delta|k_2-k_3|}
    \lp{\phi^{(2)}_{k_2}}{S^1}\lp{\phi^{(3)}_{k_3}}{S^1} , \
    k_1>\max\{k_2,k_3\}-C \ . \notag
  \end{align}
  In addition one has the following quadrilinear form bounds, which
  hold under the condition $ k< k_i-C$:
  \begin{align}
    \big|\big\langle \Box^{-1} \mathcal{H}_k
    (&\phi^{(1)}_{k_1}\cdot \partial_\alpha \phi^{(2)}_{k_2}) ,
    \mathcal{H}_k (\partial^\alpha \phi^{(3)}_{k_3}\cdot \psi_{k_4})
    \big\rangle\big| \label{mult6}\\
    \ &\lesssim \ 2^{ \delta(k-\min\{k_i\}) }
    \lp{\phi^{(1)}_{k_1}}{S^1}\lp{\phi^{(2)}_{k_2}}{S^1}\lp{\phi^{(3)}_{k_3}}{S^1}
    \lp{\psi_{k_4}}{N^*} \ ,  \notag\\
    \big|\big\langle (\Box\Delta)^{-1} &\mathcal{H}_k \partial_\alpha
    (\phi^{(1)}_{k_1}\cdot \partial^\alpha \phi^{(2)}_{k_2}) ,
    \partial_t \mathcal{H}_k (\partial_t \phi^{(3)}_{k_3}\cdot
    \psi_{k_4})
    \big\rangle\big| \label{mult7}\\
    \ &\lesssim \ 2^{ \delta(k-\min\{k_i\}) }
    \lp{\phi^{(1)}_{k_1}}{S^1}\lp{\phi^{(2)}_{k_2}}{S^1}\lp{\phi^{(3)}_{k_3}}{S^1}
    \lp{\psi_{k_4}}{N^*} \ ,    \notag\\
    \big|\big\langle (\Box\Delta)^{-1} &\nabla_x \mathcal{H}_k
    (\phi^{(1)}_{k_1}\cdot \nabla_x \phi^{(2)}_{k_2}) ,
    \mathcal{H}_k \partial_\alpha (\partial^\alpha
    \phi^{(3)}_{k_3}\cdot \psi_{k_4})
    \big\rangle\big| \label{mult8}\\
    \ &\lesssim \ 2^{ \delta(k-\min\{k_i\}) }
    \lp{\phi^{(1)}_{k_1}}{S^1}\lp{\phi^{(2)}_{k_2}}{S^1}\lp{\phi^{(3)}_{k_3}}{S^1}
    \lp{\psi_{k_4}}{N^*} \ .  \notag
  \end{align}
\end{thm}

The second category of bounds is covered by:

\begin{thm}[Additional core product  estimates]\label{prod_thm}
  The following hold:
  \begin{align}
    &\lp{ (I-\mathcal{H}^*_{k_1})(Q_{<k_2-C}A_{k_1}\partial_t
      \phi_{k_2}) }{N}\ \lesssim\ 2^{\frac{3}{2}k_1}
    \lp{A_{k_1}}{L^2(L^2)}
    \lp{\phi_{k_2}}{S^1} ,  \quad k_1<k_2-C \ , \label{mult9}\\
    &\lp{ \mathcal{H}^*_{k_1} (A_{k_1}\partial_t \phi_{k_2})
    }{L^1(L^2)} \lesssim \lp{A_{k_1}}{Z^{ell}}\lp{\phi_{k_2}}{S^1} ,
    \quad k_1<k_2-C \ ,
    \  \label{mult10}\\
    &\lp{(I-\mathcal{H}_{k_1})P_{k_1}
      (\phi^{(2)}_{k_2}\partial_t \phi^{(3)}_{k_3}) }{\Delta Z^{ell}}
\lesssim 2^{\delta(k_1-k_2)}
    \lp{\phi^{(2)}_{k_2}}{S^1}\lp{\phi^{(3)}_{k_3}}{S^1} \ , \quad
    k_1\leqslant k_2-C \ , \label{mult11}
    % &\lp{ (\phi^{(2)}_{k_2}\partial_t \phi^{(3)}_{k_3}) }{\Delta
    % Z^{ell}}\lesssim \ 2^{-\delta|k_2-k_3|}
    % \lp{\phi^{(2)}_{k_2}}{S^1}\lp{\phi^{(3)}_{k_3}}{S^1} , \
    % k_1>\max\{k_2,k_3\}-C \ . \label{mult12}
  \end{align}
\end{thm}

We conclude this Section with the application of Theorems
\ref{null_thm} and \ref{prod_thm} to the estimates listed above. This
is for the most part straight forward and left to the reader.

\begin{proof}[Proof that Theorem~\ref{null_thm} implies estimates
  \eqref{Ai_null_est}, \eqref{Ai_null_Best}, \eqref{quadphi_HL_null},
  \eqref{gen_phi_lohi_quadest}, \eqref{gen_phi_Best}, and
  \eqref{quad_null_est}]
  For \eqref{Ai_null_est} we use \eqref{mult1}, for \eqref{Ai_null_Best}
  use \eqref{mult4}, for \eqref{quadphi_HL_null} use \eqref{mult1}
  with $k_2\geq k_1+O(1)$, for \eqref{gen_phi_lohi_quadest} use
  \eqref{mult2}, for \eqref{gen_phi_Best} use \eqref{mult3}, and for
  \eqref{quad_null_est} use \eqref{mult6} - \eqref{mult8}.
\end{proof}

\begin{proof}[Proof that Theorem \ref{prod_thm} implies estimates
  \eqref{A0B_est2}, \eqref{gen_phi_lohi_quadest0}, and
  \eqref{quadphi_HL_Z0}]

For \eqref{A0B_est2} use \eqref{mult11}, for \eqref{gen_phi_lohi_quadest0} use \eqref{mult9}, and for \eqref{quadphi_HL_Z0} use \eqref{mult10}. 
 
\end{proof}

%\section{Proof of the core multilinear estimates}

For the remainder of  this Section we prove Theorems \ref{null_thm} and
\ref{prod_thm}. We begin with a simple calculation that will be used a
number of times:

\begin{lem}[Square summed $L^2(L^6)\to L^2(L^\infty)$ estimate]\label{L2L6_lem}
  Let $j-C\leqslant k'+l'\leqslant \frac{1}{2}(j+k)+C$ with $l' <C$
  and $k'\leqslant k+C$. Then one has the following uniform estimate:
  \begin{equation}
    \big(\sum_{\mathcal{C}_{k'}(l')} 
    \lp{P_{\mathcal{C}_{k'}(l')}Q_{<j}\phi_k }{L^2(L^\infty)}^2
    \big)^\frac{1}{2} \ \lesssim \ 2^{\frac{2}{3}k'}2^{\frac{1}{2}l'}
    2^{\frac{5}{6}k}\lp{\phi_k}{S_k[L^2(L^6)]} 
    \ , \label{L2L6_bernstein}
  \end{equation}
  where $\mathcal{C}_{k'}(l')$ is a finitely overlapping collection of
  radially oriented rectangles of dimensions $(2^{k'+l'})^3\times
  2^{k'}$. Here $S_{k_2}[L^2(L^6)]$ refers to the $L^2(L^6)$ portion
  (including square sums) of the norm from lines \eqref{str_and_defn}
  and \eqref{Sl_def}.
\end{lem}

\begin{proof}[Proof of estimate \eqref{L2L6_bernstein}]
 This follows immediately from the
  $L^2(L^6)$ estimate on line \eqref{Sl_def} and Bernstein's
  inequality in the form $P_{\mathcal{C}_{k'}(l')}L^6\subseteq
  2^{\frac{2}{3}k'}2^{\frac{1}{2}l'} L^\infty$.
\end{proof}

\subsection{Proof of the null form estimates}

We begin with estimates \eqref{mult1}--\eqref{mult3}. These can be
boiled down to an even more atomic form which is the following:

\begin{lem}[Core modulation estimates]
  The following estimate holds uniformly in the indices $j_i,k_i$, where
  $j_2,j_3=j_1+ O(1) $:
  \begin{multline}
    \big|\big\langle Q_{j_1} \phi_{k_1}^{(1)} , \mathcal{N}(Q_{<j_2}
    \phi_{k_2}^{(2)} , Q_{<j_3}\phi_{k_3}^{(3)})
    \big\rangle\big|\\
    \lesssim 2^{-\delta|j_1-k_2|}2^{-\delta|k_1-k_3|}
    2^{\min\{k_1,k_3\}}2^{2k_2}
    \lp{\phi^{(1)}_{k_1}}{X_\infty^{0,\frac{1}{2}}}
    \lp{\phi^{(2)}_{k_2}}{S_{k_2}}
    \lp{\phi^{(3)}_{k_3}}{L^\infty(L^2)} \ , \label{mod1}
  \end{multline}
    In addition, when $j>k_{min}+C$ one has the improved bound:
  \begin{equation}
    \big|\big\langle Q_j \phi_{k_1}^{(1)} , 
    \mathcal{N}( \phi_{k_2}^{(2)} , \phi_{k_3}^{(3)})
    \big\rangle\big| 
    \!\lesssim\!   2^{-\delta (j-k_{min})}
    2^{2k_{min}}2^{k_{max}}
    \lp{\!\phi^{(1)\!}_{k_1}}{X_\infty^{0,\frac{1}{2}}} 
    \lp{\!\phi^{(2)\!}_{k_2}}{S_{k_2}} 
    \lp{\!\phi^{(3)\!}_{k_3}}{N^*} \ . \label{mod2}
  \end{equation}
\end{lem}

%%%%

\begin{proof}[Proof of estimate \eqref{mod1}]
  There are three cases depending on the relative separation of
  spatial frequencies. Each of these is further split into low and
  high modulation subcases.

  \case{1a}{$k_2=k_3+O(1)$ and $j_1<k_1$} Here the frequency angle
  separation between the second two factors is $\angle(\phi^2,\phi^3)
  \lesssim 2^{-k_2}2^{\frac{1}{2}(k_1+j_1)}:=2^l$.  The null form
  $\mathcal{N}$ saves one power of this angle.  On the other hand, the
  angle of interaction with the output is $\angle(\phi^1,\phi^2)
  \lesssim 2^{\frac{1}{2}(j_1-k_1)}:=2^{l'}$.  Breaking up the high
  frequency factors in the null form with respect to the $2^{k_1}$
  radial scale and the $2^l$ sector scale it remains to estimate the
  expression $Q_{j_1} F$ in $L^2$ where
\[
 F :=  P_{k_1}
      \sum_{\mathcal{C}_{k_1}(l')} \mathcal{N} \big(
      P_{\mathcal{C}_{k_1}(l')}Q_{<j_2} \phi^{(2)}_{k_2} ,
      P_{-\mathcal{C}_{k_1}(l')}
      Q_{<j_3} \phi^{(3)}_{k_3})
\]
Disposing of $\mathcal{N}$ we start with  the fixed time estimate
\[
\begin{split}
\| F(t)\|_{L^2}^2 \lesssim & \  2^{2(k_2+k_3+l)} 
  \big(\sum_{\mathcal{C}_{k_1}(l')}
\|  P_{\mathcal{C}_{k_1}(l')}Q_{<j_2} \phi^{(2)}_{k_2}(t)\|_{L^\infty} 
\|  P_{-\mathcal{C}_{k_1}(l')}  Q_{<j_3} \phi^{(3)}_{k_3}(t) \|_{L^2}\big)^2
\\
\lesssim & \ 2^{2k_3 + k_1 + j_1}  \left(  \sum_{\mathcal{C}_{k_1}(l')}
\|  P_{\mathcal{C}_{k_1}(l')}Q_{<j_2} \phi^{(2)}_{k_2}(t)\|_{L^\infty}^2 \right)
 \left(  \sum_{\mathcal{C}_{k_1}(l')}\|  P_{-\mathcal{C}_{k_1}(l')}  Q_{<j_3} \phi^{(3)}_{k_3}(t)\|^2_{L^2}\right)
\\
\lesssim & \ 2^{2k_3 + k_1 + j_1}  \left(  \sum_{\mathcal{C}_{k_1}(l')}
\|  P_{\mathcal{C}_{k_1}(l')}Q_{<j_2} \phi^{(2)}_{k_2}(t)\|_{L^\infty}^2 \right)
\| \phi^{(3)}_{k_3}\|^2_{L^\infty L^2}
\end{split}
\]
where at the last step we have used orthogonality in frequency and then
the $L^\infty L^2$ boundedness of $Q_{<j_3}$. Hence integrating in time 
we arrive at
\begin{equation}
\|F\|_{L^2}^2 \lesssim  \ 2^{2k_3 + k_1 + j_1} \| \phi^{(3)}_{k_3}\|^2_{L^\infty L^2} 
  \sum_{\mathcal{C}_{k_1}(l)}
\|  P_{\mathcal{C}_{k_1}(l')}Q_{<j_2} \phi^{(2)}_{k_2}(t)\|_{L^2 L^\infty}^2 
\end{equation}
Then applying   Lemma \ref{L2L6_lem} to put $\phi_{k_2}^{(2)}$ in $L^2(L^\infty)$,
  we see that:
\[
\begin{split}
\|Q_{j_1} F\|_{L^2} \lesssim & 2^{\frac{2}{3}k_1}2^{\frac{1}{2}l'}
    2^{\frac{5}{6}k_3}  2^{k_3 + \frac12(k_1 + j_1)}  \lp{\phi^{(2)}_{k_2}}{S_{k_2}[L^2(L^6)]}
    \lp{\phi^{(3)}_{k_3}}{L^\infty(L^2)}
\\ = & \ 2^{\frac{1}{2}j_1}2^{\frac{1}{4}(j_1-k_1)}2^{\frac{1}{6}
      (k_1-k_2)}2^{k_1}2^{k_2}2^{k_3} \lp{\phi^{(2)}_{k_2}}{S_{k_2}[L^2(L^6)]}
    \lp{\phi^{(3)}_{k_3}}{L^\infty(L^2)}
\end{split}
\]
concluding the proof in this case.

\case{1b}{$k_2=k_3+O(1)$ and $j_1\geqslant k_1$} Here integrating by
parts we get a factor of $2^{k_1+k_2}$ from the null form.  Then we
localize $\phi_{k_2}^{(2)}$ and $\phi_{k_3}^{(3)}$ with respect to
$2^{k_1}$ sized frequency cubes, but without any modulation
localization. For $ \phi_{k_2}^{(2)}$ we use the bound
\eqref{L2Linfty_sqsum}, and for $\phi_{k_3}^{(3)}$ we use the
$L^\infty L^2$ norm. Then the same computation as above for
\[
F = 
   \sum_{\mathcal{C}_{k_1}} \mathcal{N} \big(
      P_{\mathcal{C}_{k_1}}Q_{<j_2} \phi^{(2)}_{k_2} ,
      P_{-\mathcal{C}_{k_1}}
      Q_{<j_3} \phi^{(3)}_{k_3})
\]
yields
\[
\|F\|_{L^2} \lesssim 2^{2k_1+ \frac32 k_2} \|   \phi^{(2)}_{k_2}\|_{S_{k_2}} 
\|  \phi^{(3)}_{k_3}\|_{L^\infty L^2}
\]
which suffices.

  \case{2a}{$k_1=k_2+O(1)$ and $j_1 < k_3$} This case is mostly
  analogous to \textbf{Case 1a}. Integrating by parts one may place the
  null form between $\phi_{k_1}^{(1)}$ and $\phi_{k_2}^{(2)}$.  Then
  the angular decompositions between the inputs of these two terms,
  the angle of the output (frequency $2^{k_3}$ now) is the same as
  \textbf{Case 1} but with $k_1$ and $k_3$ transposed. 
With 
\[
 F :=  P_{k_3}
      \sum_{\mathcal{C}_{k_3}(l')} \mathcal{N} \big(
      P_{\mathcal{C}_{k_3}(l')}Q_{j_1} \phi^{(1)}_{k_1} ,
      P_{-\mathcal{C}_{k_3}(l')}
      Q_{<j_2} \phi^{(2)}_{k_2})
\]
the relevant
  computation becomes:
\[
\begin{split}
\| F\|_{L^1 L^2} \lesssim &  \  2^{k_1+\frac12(k_3+j_1)} 
\sum_{\mathcal{C}_{k_3}(l')} \|  P_{\mathcal{C}_{k_3}(l')}Q_{j_1} \phi^{(1)}_{k_1}\|_{L^2} \|   P_{-\mathcal{C}_{k_3}(l')}
      Q_{<j_2} \phi^{(2)}_{k_2}\|_{L^2 L^\infty}
\\
 \lesssim &  \  2^{k_1+\frac12(k_3+j_1)} 
\left(\sum_{\mathcal{C}_{k_3}(l')}
 \|  P_{\mathcal{C}_{k_3}(l')}Q_{j_1} \phi^{(1)}_{k_1}\|_{L^2}
\right)^\frac12\left(
\sum_{\mathcal{C}_{k_3}(l')} \|   P_{-\mathcal{C}_{k_3}(l')}
      Q_{<j_2} \phi^{(2)}_{k_2}\|_{L^2 L^\infty}\right)^\frac12
\\
\lesssim & \ 2^{\frac{1}{2}j_1}2^{\frac{1}{4}(j_1-k_3)}2^{\frac{1}{6}
      (k_3-k_2)}2^{k_1}2^{k_2}2^{k_3} \lp{Q_{j_1}\phi^{(1)}_{k_1}}{L^2(L^2)}
    \lp{\phi^{(2)}_{k_2}} {S_{k_2}[L^2(L^6)]} ,
\end{split}
\]
  which suffices.

  \case{2b}{$k_1=k_2+O(1)$ and $j_1 > k_3$} The same argument as in
  {\bf Case 2a} applies, with the only difference that
  $\phi^{(1)}_{k_1}$ and $\phi^{(2)}_{k_2}$ are now decomposed on the
  frequency scale $2^{k_3}$, and there is no modulation cutoff. Thus 
the bound \eqref{L2Linfty_sqsum} has to be  used for the latter.

 \case{3a}{$k_1=k_3+O(1)$ and $j < k_2$} The computation here is
  essentially the same as above, but with slightly different
  numerology. The main difference is now that $\angle(\phi^2,\phi^3)
  \lesssim 2^{\frac{1}{2}(j_1-k_2)}:=2^l$, and we sum over sectors of
  size $\sim 2^{k_2}\times (2^{k_2+l})^3$.  Then a computation similar 
to {\bf Case 1} gives us:
\[
\begin{split}
    \hbox{LHS}\eqref{mod1} \ \lesssim & \ 2^{-\frac{1}{2}j_1}
    2^{\frac{3}{2}l}2^{\frac{5}{2}k_2}2^{k_3}
    \lp{\phi^{(1)}_{k_1}}{X_\infty^{0,\frac{1}{2}}}
    \lp{\phi^{(2)}_{k_2}}{S_{k_2}[L^2(L^6)]}
    \lp{\phi^{(3)}_{k_3}}{L^\infty(L^2)} \ , \\
    = & \ 2^{\frac{1}{4}(j_1-k_2)}2^{2k_2}2^{k_3}
    \lp{\phi^{(1)}_{k_1}}{X_\infty^{0,\frac{1}{2}}}
    \lp{\phi^{(2)}_{k_2}}{S_{k_2}[L^2(L^6)]}
    \lp{\phi^{(3)}_{k_3}}{L^\infty(L^2)} \ . \notag
  \end{split}
\]

 \case{3b}{$k_1=k_3+O(1)$ and $j > k_2$} Here we directly
  use the easy product estimate:
  \begin{equation}
    \hbox{LHS}\eqref{mod1} \ \lesssim \ 2^{-\frac{1}{2}j_1}
    2^{\frac{5}{2}k_2}2^{k_3}
    \lp{\phi^{(1)}_{k_1}}{X_\infty^{0,\frac{1}{2}}} 
    \lp{\phi^{(2)}_{k_2}}{S_{k_2}} 
    \lp{\phi^{(3)}_{k_3}}{L^\infty(L^2)} 
    \ . \notag
  \end{equation}
\end{proof}

%%%%

\begin{proof}[Proof of estimate \eqref{mod2}]
  There are two cases depending on the relation of the modulations of
  $\phi^{(2)}_{k_2}$ and $\phi^{(3)}_{k_3}$ to $2^j$.

  \case{1}{One of $\phi^{(2)}_{k_2}$ or $\phi^{(3)}_{k_3}$ has
    modulation comparable to $2^j$} This situation is symmetric. It
  suffices to show the single estimate:
  \begin{equation}
    \lp{P_{k_1}\mathcal{N}(Q_{>j-C}\phi^{(2)}_{k_2},
      \phi^{(3)}_{k_3})}{L^2(L^2)} \ \lesssim \
    2^{-\frac{1}{2}j }2^{3\min\{k_i\}}2^{\max\{k_i\}} 
    \lp{\phi^{(2)}_{k_2}}{X_\infty^{0,\frac{1}{2}}}
    \lp{\phi^{(3)}_{k_3}}{L^\infty(L^2)} \ , \notag
  \end{equation}
  which follows immediately from putting a derivative of $\mathcal{N}$
  on the lowest frequency and $L^2\to L^\infty$ Bernstein's inequality
  for the lowest frequency as well.

  \case{2}{Both $\phi^{(2)}_{k_2}$ and $\phi^{(3)}_{k_3}$ have
    modulation $\ll 2^j$} This can happen only when $k_1< k_2-C$, in
  which case it is a $(++)$ or $(--)$ type $High\times High\Rightarrow
  Low$ interaction between $\phi^{(2)}_{k_2}$ and
  $\phi^{(3)}_{k_3}$. Thus, $j=k_2+O(1)$ and it suffices to prove:
  \begin{equation}
    \lp{P_{k_1}\mathcal{N}(Q_{<j-C}\phi^{(2)}_{k_2},
      Q_{<j-C}\phi^{(3)}_{k_3})}{L^2(L^2)} \ \lesssim \
    2^{2k_1}2^{\frac{3}{2}k_2} 
    \lp{\phi^{(2)}_{k_2}}{S_{k_2}}\lp{\phi^{(3)}_{k_3}}{L^\infty(L^2)} \ . \notag
  \end{equation}
  This follows at once by putting a derivative on the low frequency
  output, and breaking the product into antipodal blocks
  $\mathcal{C}_{k_1}$ of scale $2^{k_1}$ while using the special
  $L^2(L^\infty)$ square sum bound \eqref{L2Linfty_sqsum} for the
  first factor. Note that the multiplier
  $Q_{<j-C}P_{\mathcal{C}_{k_1}}$ is disposable thanks to $j>k_1+C$.
\end{proof}\ret

%%%%

\begin{proof}[Proof of estimate \eqref{mult1}]
 This follows  from  the next two bounds:
  \begin{align}
    \lp{ (I-\mathcal{H}_{k_1}- Q_{>k_1+C})P_{k_1}\mathcal{N}
      (\phi^{(2)}_{k_2}, \phi^{(3)}_{k_3}) }{L^1(L^2)} \ \lesssim \
    \hbox{LHS}\eqref{mult1}
    \ , \label{mult111}\\
    \lp{(Q_{>k_1+C}+ \mathcal{H}_{k_1})P_{k_1}\mathcal{N}
      (\phi^{(2)}_{k_2}, \phi^{(3)}_{k_3}) }{X_1^{0,-\frac{1}{2}}} \
    \lesssim \ \hbox{LHS}\eqref{mult1} \ . \label{mult12}
    % \lp{ Q_{>k_1+C} P_{k_1}\mathcal{N} (\phi^2_{k_2}, \phi^3_{k_3})
    % }{X_1^{0,\infty}} \ \lesssim \ \hbox{LHS}\eqref{mult1} \
    %		 . \label{mult13}
  \end{align}

  \case{1}{Estimate \eqref{mult111}} The restriction induced by $
  (I-\mathcal{H}_{k_1}- Q_{>k_1+C})$ means that one of the two input
  factors always has the leading modulation.  First permute notation
  so the output frequency is $k_3$ and the two inputs are
  $k_1,k_2$. Then by duality (note that $\mathcal{N}$ is skew-adjoint
  as a form) and symmetry of the estimate, it suffices to show bounds
  of the form:
  \begin{multline}
    \big|\big\langle Q_j \phi_{k_1}^{(1)} , \mathcal{N}(Q_{<j+O(1)}
    \phi_{k_2}^{(2)} , Q_{<j+O(1)}\psi_{k_3})
    \big\rangle\big|\\
    \lesssim \
    2^{-\delta|j-k_2|}2^{k_3}2^{\delta(k_3-\max\{k_1,k_2\})}
    2^{-\delta|k_1-k_2|} 2^{k_1}2^{k_2} \lp{\! \phi^{(1)}_{k_1}\!
    }{S_{k_1}}\lp{\! \phi^{(2)}_{k_2}\! }{S_{k_2}}
    \lp{\psi_{k_3}}{L^\infty(L^2)} \ , \notag
  \end{multline}
  which follows directly from \eqref{mod1}.

  \case{2a}{Estimate \eqref{mult12} for low modulations} Freezing the
  output modulation it suffices to show:
  \begin{multline}
    \lp{Q_j S_{k_1}\mathcal{N}
      (Q_{<j-C}\phi^{(2)}_{k_2}, Q_{<j-C}\phi^{(3)}_{k_3})}{X_1^{0,-\frac{1}{2}}}\\
    \lesssim \ 2^{-\delta
      |j-k_2|}2^{k_1}2^{\delta(k_1-\max\{k_2,k_3\})}
    2^{-\delta|k_2-k_3|} 2^{k_2}2^{k_3} \lp{\! \phi^{(2)}_{k_2}\!
    }{S_{k_2}}\lp{\! \phi^{(3)}_{k_3}\! }{S_{k_3}} \ , \notag
  \end{multline}
  which again is a direct consequence of \eqref{mod1}.

  \case{2b}{Estimate \eqref{mult12} for high modulations} In this case
  we use:
  \begin{equation}
    \lp{Q_j S_{k_1}\mathcal{N}
      (\phi^{(2)}_{k_2}, \phi^{(3)}_{k_3})}{X_1^{0,-\frac{1}{2}}}\\
    \lesssim \ 2^{-\delta |j-k_2|}2^{k_1}2^{\delta(k_1-\max\{k_2,k_3\})}
    2^{-\delta|k_2-k_3|} 2^{k_2}2^{k_3}
    \lp{\! \phi^{(2)}_{k_2}\! }{S_{k_2}}\lp{\! \phi^{(3)}_{k_3}\! }{S_{k_3}}
    \ , \notag
  \end{equation}
  follows immediately from \eqref{mod2} in the case $j>k_1+C$.
\end{proof}\ret

%%%%

\begin{proof}[Proof of estimate \eqref{mult2}]
  There are two main cases:

  \case{1}{$\phi_{k_1}^{(1)}$ with highest modulation} Due to the
  restriction of $\mathcal{H}^*$, this can only occur when the second
  factor is $Q_{>k_1+C}\phi_{k_1}^{(1)}$. Then we use \eqref{mod2}
  which directly implies for this case:
  \begin{equation}
    \lp{\mathcal{N}(Q_j\phi_{k_1}^{(1)},\phi_{k_2}^{(2)} ) }{N} \ \lesssim \ 
    (2^{\delta(k_1-j)}+2^{-\delta|j-k_2|}) 2^{2k_1} 2^{k_2}
    \lp{\phi_{k_1}^{(1)}}{S_{k_1}}\lp{\phi_{k_2}^{(2)}}{S_{k_2}} \ . \notag
  \end{equation}

  \case{2}{Output or $\phi_{k_2}^{(2)}$ have the leading modulation}
  In this case we end up needing bounds of the form:
  \begin{align}
    \lp{Q_j \mathcal{N}(Q_{<j+O(1)}\phi_{k_1}^{(1)},
      Q_{<j+O(1)}\phi_{k_2}^{(2)} ) }{X_1^{0,-\frac{1}{2}}} \
    &\lesssim \ 2^{-\delta|j-k_1|} 2^{2k_1} 2^{k_2}
    \lp{\phi_{k_1}^{(1)}}{S_{k_1}}\lp{\phi_{k_2}^{(2)}}{S_{k_2}} \ . \notag\\
    \lp{Q_{<j+O(1)}\mathcal{N}(Q_{<j+O(1)}\phi_{k_1}^{(1)},
      Q_j\phi_{k_2}^{(2)} ) }{L^1(L^2)} \ &\lesssim \
    2^{-\delta|j-k_1|} 2^{2k_1} 2^{k_2}
    \lp{\phi_{k_1}^{(1)}}{S_{k_1}}\lp{\phi_{k_2}^{(2)}}{S_{k_2}} \ ,
    \notag
  \end{align}
  which are both immediate from \eqref{mod1}.
\end{proof}\ret

%%%%

\begin{proof}[Proof of estimate \eqref{mult3}]
  This is a direct computation using angular decompositions. At a
  fixed modulation $2^j$ the angle of interaction is
  $\angle(\phi^1,\phi^2)\lesssim 2^{\frac{1}{2}(j-k_1)}:=2^l$, and
  disposing of the null form we have:
  \begin{multline}
    \lp{Q_{<j-C}\mathcal{N}(Q_j\phi^{(1)}_{k_1}, Q_{<j-C}\phi^{(2)}_{k_2})  }{L^1(L^2)}\\
    \lesssim \ 2^l 2^{k_1}2^{k_2}\sum_\omega \lp{P^\omega_l
      Q_{k_1+2l}\phi^{(1)}_{k_1}}{L^1(L^\infty)}\cdot \lp{P_l^\omega
      Q_{<j-C}\phi^{(2)}_{k_2} }{L^\infty(L^2)} \ . \notag
  \end{multline}
  Using Cauchy-Schwarz and:
  \begin{equation}
    \lp{P_l^\omega
      Q_{<j-C}\phi^{(2)}_{k_2}  }{L^\infty(L^2)}^2  \ \lesssim \ 
    \sum_{\omega'\subseteq \omega }\lp{P_{2^{\frac{1}{2}(j-k_2)}}^{\omega'}
      Q_{<j-C}\phi^{(2)}_{k_2}  }{L^\infty(L^2)}^2 \ , \notag
  \end{equation}
  we have:
  \begin{equation}
    \lp{Q_{<j-C}\mathcal{N}(Q_j\phi^{(1)}_{k_1}, Q_{<j-C}\phi^{(2)}_{k_2})  }{L^1(L^2)}\
    \lesssim \ 2^{\frac{1}{4}(j-k_1)} 2^{k_1}\lp{\phi_{k_1}^{(1)}}{Z_{k_1}^{hyp}}
    \lp{\phi_{k_2}^{(2)}}{S^1_{k_2}} \ , \notag
  \end{equation}
  which suffices.
\end{proof}\ret

\begin{proof}[Proof of estimate \eqref{mult4}]
  This follows immediately from \eqref{mult11} and
  \eqref{B_embed}. Notice that the lack of an $L^1(L^2)$ estimate for
  $Q_{>k_1+C}P_{k_1}\mathcal{N}(\phi^{(2)}_{k_2},\phi^{(3)}_{k_3})$ is
  irrelevant because definition of $Z^{hyp}$ limits modulations to
  $Q_{<k_1+C}P_{k_1}$.
\end{proof}\ret

%%%%

\begin{proof}[Proof of estimate \eqref{mult5}]
  By symmetry we may assume $k_2<k_3+O(1)$, in which case
  $k_1=k_3+O(1)$. There are two subcases depending on the size of the
  output modulation $2^j$:

  \case{1}{$k_2>j-C$} In this case the angle of interaction between
  the two inputs is $\angle(\phi^2,\phi^3)\lesssim
  2^{\frac{1}{2}(j-k_2)}:=2^{l'}$.  On the other hand, the output
  sector localization of $2^l:=2^{\frac{1}{2}(j-k_1)}\approx
  2^{\frac{1}{2}(j-k_3)}$ is passed to the high frequency factor, so
  using Lemma \ref{L2L6_lem} we have:
  \begin{align}
    &\sum_\omega 2^l \lp{ P^{\omega}_l Q_j S_{k_1} \mathcal{N}
      (Q_{<j-C}\phi^{(2)}_{k_2},
      Q_{<j-C}\phi^{(3)}_{k_3})}{L^1(L^\infty)}^2
    \notag \\
    \ \lesssim \ &2^l 2^{2l'} 2^{2k_2}2^{2k_3}
    \sum_{\omega'}\sum_{\substack{\omega:\\
        \omega\subseteq \omega'}} \lp{P^{\omega'}_{l'}
      Q_{<j-C}\phi^{(2)}_{k_2}}{L^2(L^\infty)}^2
    \lp{P^\omega_l Q_{<j-C}\phi^{(3)}_{k_3}}{L^2(L^\infty)}^2 \ , \notag\\
    \lesssim \ &2^{2l}2^{3l'}2^{5k_2}2^{5k_3}
    \lp{\phi^{(2)}_{k_2}}{S_{k_2}[L^2(L^6)]}^2
    \lp{\phi^{(3)}_{k_3}}{S_{k_3}[L^2(L^6)]}^2 \ ,\notag\\
    \lesssim \ &(2^{k_1+j})^2\cdot 2^{2k_1} 2^{2(k_2-k_3)}
    \lp{\phi^{(2)}_{k_2}}{S_{k_2}^1}^2
    \lp{\phi^{(3)}_{k_3}}{S_{k_3}^1}^2 \ . \notag
  \end{align}

  \case{2}{$k_2<j-C$} Here the calculation is essentially the same as
  above, except that $l'=0$. This again gives:
  \begin{equation}
    \lp{ Q_j S_{k_1}
      \mathcal{N} (Q_{<j-C}\phi^{(2)}_{k_2}, Q_{<j-C}\phi^{(3)}_{k_3})}{Z^{hyp}}\
    \lesssim \ (2^{j+k_1})\cdot 2^{k_1}2^{(k_2-k_3)}
    \lp{\phi^{(2)}_{k_2}}{S_{k_2}^1}
    \lp{\phi^{(3)}_{k_3}}{S_{k_3}^1}
    \ . \notag
  \end{equation}
\end{proof}\ret

%%%%

\begin{proof}[Proof of estimate \eqref{mult6}]
  Freezing the output modulation, we will show:
  \begin{multline}
    \big|\big\langle \Box^{-1} Q_j P_{k} (Q_{<j-C}
    \phi^{(1)}_{k_1}\cdot \partial_\alpha Q_{<j-C}\phi^{(2)}_{k_2}) ,
    Q_j P_{k} (\partial^\alpha Q_{<j-C}\phi^{(3)}_{k_3}\cdot
    Q_{<j-C}\psi_{k_4})
    \big\rangle\big| \\
    \lesssim \ 2^{\frac{1}{4}(j-k)}2^{\frac{1}{2}(k-\min\{k_i\})}
    \lp{\phi^{(1)}_{k_1}}{S^1}\lp{\phi^{(2)}_{k_2}}{S^1}\lp{\phi^{(3)}_{k_3}}{S^1}
    \lp{\psi_{k_4}}{N^*} \ . \label{main_null_dyadic}
  \end{multline}
  Here we are in the configuration $k_1=k_2+O(1)$, $k_3=k_4+O(1)$, and
  $k<k_i-C$.  Thus, the left and right products are summed over
  diametrically opposite angular sectors of size $2^k\times
  (2^{k+l})^3$, where $2^l:= 2^{\frac{1}{2}(j-k)}$. On the other hand,
  the null form between the second and third terms gains us
  $\angle(\phi^2,\phi^3)^2$, where this angle mod $\pi$ cannot exceed
  $2^{l+C}$.  Therefore we group the product of the two diagonal sums
  into dyadic values of $\angle(\phi^2,\phi^3)\!\!\!\mod \pi$ and
  break into three cases:

  \case{1}{$\angle(\phi^2,\phi^3)\!\!\!\mod \pi \lesssim 2^l
    2^{k-k_2}$} A little care is needed to use orthogonality in
  space. To gain this, at first keep the second diagonal sum under the
  time integral as follows:
  \begin{multline}
    \hbox{LHS}\eqref{main_null_dyadic}|_{\angle(\phi^2,\phi^3)\!\!\!\!\!\!\mod\!\pi
      \lesssim 2^l 2^{k-k_2}}\\
    \lesssim \ 2^{-k_2}2^{k_3} \sum_{\mathcal{C}_k(l)}
    \lp{P_{\mathcal{C}_k(l)}Q_{<j-C} \phi^{(1)}_{k_1}}{L^2(L^\infty)}
    \lp{P_{-\mathcal{C}_k(l)}Q_{<j-C} \phi^{(2)}_{k_2}}{L^2(L^\infty)}\\
    \times \sup_t \sum_{\mathcal{C}'_k(l)}
    \lp{P_{\mathcal{C}'_k(l)}Q_{<j-C} \phi^{(3)}_{k_3}(t)}{L^2_x}
    \lp{P_{-\mathcal{C}'_k(l)}Q_{<j-C} \psi_{k_4}(t)}{L^2_x} \
    . \notag
  \end{multline}
  The inner sum of the second factor on the RHS can easily be
  reconstructed after Cauchy-Schwarz by spatial orthogonality. On the
  other hand, for the two $L^2(L^\infty)$ norms we use Lemma
  \ref{L2L6_lem}. This gives us:
  \begin{align}
    &\hbox{LHS}\eqref{main_null_dyadic}
    |_{\angle(\phi^2,\phi^3)\!\!\!\!\!\!\mod\!\pi \lesssim 2^l
      2^{k-k_2}}
    \notag\\
    \lesssim \ &2^{\frac{4}{3}k}2^l
    2^{\frac{5}{6}k_1}2^{-\frac{1}{6}k_2}2^{k_3}
    \lp{\phi^{(1)}_{k_1}}{S_{k_1}[L^2(L^6)]}
    \lp{\phi^{(2)}_{k_2}}{S_{k_2}[L^2(L^6)]}
    \lp{\phi^{(3)}_{k_3}}{S_{k_3}}
    \lp{\psi_{k_4}}{N^*} \ , \notag\\
    \lesssim \ &2^{\frac{1}{2}(j-k)}2^{\frac{4}{3}(k-\min\{k_i\})}
    \lp{\phi^{(1)}_{k_1}}{S^1}\lp{\phi^{(2)}_{k_2}}{S^1}\lp{\phi^{(3)}_{k_3}}{S^1}
    \lp{\psi_{k_4}}{N^*} \ , \notag
  \end{align}
  which is even better than \eqref{main_null_dyadic}.

  \case{2}{$\angle(\phi^2,\phi^3)\!\!\!\mod \pi \lesssim 2^l
    2^{k-k_3}$} This is essentially the same as \textbf{Case 1} above,
  but since the angular gain is in frequency $2^{k_3}$ we put
  $\phi^{(3)}_{k_3}$ in a square summed $L^2(L^\infty)$ via Lemma
  \ref{L2L6_lem} and $\phi^{(1)}_{k_1}$ in $L^\infty(L^2)$
  instead. This gives:
  \begin{align}
    &\hbox{LHS}\eqref{main_null_dyadic}
    |_{\angle(\phi^2,\phi^3)\!\!\!\!\!\!\mod\!\pi \lesssim 2^l
      2^{k-k_3}}
    \notag\\
    \lesssim \ &2^{\frac{4}{3}k}2^l 2^{k_1} 2^{\frac{5}{6}k_2}
    2^{-\frac{1}{6}k_3} \lp{\phi^{(1)}_{k_1}}{L^\infty(L^2)}
    \lp{\phi^{(2)}_{k_2}}{S_{k_2}[L^2(L^6)]}
    \lp{\phi^{(3)}_{k_3}}{S_{k_3}[L^2(L^6)]}
    \lp{\psi_{k_4}}{N^*} \ , \notag\\
    \lesssim \ &2^{\frac{1}{2}(j-k)}2^{\frac{4}{3}(k-\min\{k_i\})}
    \lp{\phi^{(1)}_{k_1}}{S^1}\lp{\phi^{(2)}_{k_2}}{S^1}\lp{\phi^{(3)}_{k_3}}{S^1}
    \lp{\psi_{k_4}}{N^*} \ , \notag
  \end{align}
  which suffices. Note that a similar square summation for fixed time
  procedure as in \textbf{Case 1} was used here.

  \case{3}{$2^l 2^{k-\min\{k_2,k_3\}} \ll
    \angle(\phi^2,\phi^3)\!\!\!\mod\pi \lesssim 2^l$} In this case
  there is a definite angle between spatial frequencies in the blocks
  $\mathcal{C}_k(l),\mathcal{C}'_k(l)$ for fixed dyadic $2^{l'}=
  \angle(\phi^2,\phi^3)$.  Thus, we can use one power of
  $\angle(\phi^2,\phi^3)$ to put the product of $\phi^{(2)}_{k_2}$ and
  $\phi^{(3)}_{k_3}$ in $L^2(L^2)$ using null frames. A little more
  care is needed here to gain spatial orthogonality for
  $\psi_{k_4}$. Thus, we first fix time and compute:
  \begin{align}
    &\hbox{LHS}\eqref{main_null_dyadic}|_{\angle(\phi^2,\phi^3)
      \!\!\!\!\!\!\mod\!\pi\sim 2^{l'},t=const}\notag \\
    \lesssim \ &2^{-2(k+l)}2^{k_2}2^{k_3}2^{2l'} \hspace{-.3in}
    \sum_{\substack{ \mathcal{C}_k(l),\mathcal{C}'_k(l) :\\
        \angle(\mathcal{C}_k(l),\pm \mathcal{C}_k(l)')\sim 2^{l'} }}
    \hspace{-.2in} \lp{P_{\mathcal{C}_k(l)}Q_{<j-C}
      \phi^{(2)}_{k_2}(t)\cdot P_{\pm \mathcal{C}'_k(l)}Q_{<j-C}
      \phi^{(3)}_{k_3}(t)
    }{L^2_x}\notag\\
    &\ \ \ \ \times \lp{P_{\mathcal{C}_k(l)}Q_{<j-C}
      \phi^{(1)}_{k_1}(t)}{L^\infty}
    \lp{P_{\mp\mathcal{C}'_k(l)}Q_{<j-C} \psi_{k_4}(t)}{L^2_x}
    \ , \notag\\
    \lesssim \ &2^{-2(k+l)}2^{k_2}2^{k_3}2^{2l'} \lp{Q_{<j-C}
      \psi_{k_4}(t)}{L^2_x}\cdot \big(\sum_{ \mathcal{C}_k(l)}
    \lp{P_{\mathcal{C}_k(l)}Q_{<j-C} \phi^{(1)}_{k_1}(t)}{L^\infty}^2\big)^{\frac{1}{2}}\notag\\
    &\ \ \ \  \times \big(\hspace{-.3in}\sum_{\substack{ \mathcal{C}'_k(l) :\\
        \angle(\mathcal{C}_k(l),\pm \mathcal{C}_k(l)')\sim 2^{l'} }}
    \hspace{-.2in} \lp{P_{\mathcal{C}_k(l)}Q_{<j-C}
      \phi^{(2)}_{k_2}(t)\cdot P_{\pm \mathcal{C}'_k(l)}Q_{<j-C}
      \phi^{(3)}_{k_3}(t) }{L^2_x}^2\big)^\frac{1}{2} \ . \notag
  \end{align}
  Integrating and using Cauchy-Schwarz in time, and then using the
  special microlocalized $L^2(L^\infty)$ block norm from line
  \eqref{Sl_def} for $\phi^{(1)}_{k_1}$, we get:
  \begin{equation}
    \hbox{LHS}\eqref{main_null_dyadic}|_{\angle(\phi^2,\phi^3)
      \!\!\!\!\!\!\mod\!\pi\sim 2^{l'}}\
    \lesssim \ 2^{-k}2^{-2l}2^{l'} 2^{\frac{1}{2}k_1}2^{k_2}2^{k_3} I_{23}(l') 
    \lp{\phi^{(1)}_{k_1}}{S_{k_1}}
    \lp{\psi_{k_4}}{N^*} \ . \notag
  \end{equation}
  where:
  \begin{equation}
    I_{23}(l')^2 \ = \ 
    \hspace{-.3in}
    \sum_{\substack{ \mathcal{C}_k(l),\mathcal{C}'_k(l) :\\
        \angle(\mathcal{C}_k(l),\pm \mathcal{C}_k(l)')\sim 2^{l'} }}
    \hspace{-.2in}
    2^{2l'}\lp{P_{\mathcal{C}_k(l)}Q_{<j-C} \phi^{(2)}_{k_2}\cdot
      P_{\pm \mathcal{C}'_k(l)}Q_{<j-C} \phi^{(3)}_{k_3}
    }{L^2(L^2)}^2
    \ . \notag
  \end{equation}
  Since we already spent $2^{k_3}$ we put the second factor in
  $L^\infty_\omega(L^2_{\omega^\perp})$ and use
  $L^2_{\omega}(L^\infty_{\omega^\perp})$ for the first. This gives
  us:
  \begin{equation}
    I_{23}(l') \ \lesssim \ 2^{\frac{3}{2}(k+l)}
    \lp{\phi^{(2)}_{k_2}}{S_{k_2}}\lp{\phi^{(3)}_{k_3}}{S_{k_3}} \ , \notag
  \end{equation}
  and so:
  \begin{equation}
    \hbox{LHS}\eqref{main_null_dyadic}|_{\angle(\phi^2,\phi^3)
      \!\!\!\!\!\!\mod\!\pi\sim 2^{l'}}\
    \lesssim \ 2^{\frac{1}{2}l'}2^{\frac{1}{2}(k-\min\{k_i\})}
    \lp{\phi^{(1)}_{k_1}}{S^1}
    \lp{\phi^{(2)}_{k_2}}{S^1}\lp{\phi^{(3)}_{k_3}}{S^1}
    \lp{\psi_{k_4}}{N^*} \ . \notag
  \end{equation}
  Summing this over all $l'<l+C$ gives \eqref{main_null_dyadic} for
  this case.
\end{proof}\ret

%%%%

\begin{proof}[Proof of estimate \eqref{mult7}]
  Freezing the output modulation, we'll show:
  \begin{multline}
    \big|\big\langle (\Box\Delta)^{-1} Q_j
    P_k \partial_t \partial_\alpha
    (Q_{<j-C}\phi^{(1)}_{k_1}\cdot \partial^\alpha
    Q_{<j-C}\phi^{(2)}_{k_2}) , (\partial_t
    Q_{<j-C}\phi^{(3)}_{k_3}\cdot Q_{<j-C}\psi_{k_4})
    \big\rangle\big| \\
    \ \lesssim \ 2^{\frac{1}{2}(j-k)}2^{ \frac{1}{6}(k-k_1) }
    \lp{\phi^{(1)}_{k_1}}{S^1}
    \lp{\phi^{(2)}_{k_2}}{S^1}\lp{\phi^{(3)}_{k_3}}{S^1}
    \lp{\psi_{k_4}}{N^*} \ . \notag
  \end{multline}
  We will use an $L^\infty(L^2)$ estimate for both the third and
  fourth factors. Thus, via H\"older's inequality, a simple weight
  calculation, and the Leibniz rule we have reduced matters to:
  \begin{align}
    \lp{Q_j P_k (\partial_\alpha
      Q_{<j-C}\phi^{(1)}_{k_1}\cdot \partial^\alpha
      Q_{<j-C}\phi^{(2)}_{k_2})}{L^1(L^\infty)} &\lesssim
    2^{j+2k}2^{\frac{1}{2}(j-k)}2^{ \frac{1}{3}(k-k_1) }
    \lp{\phi^{(1)}_{k_1}}{S^1}
    \lp{\phi^{(2)}_{k_2}}{S^1} \ , \label{Q_0_L1Linfty}\\
    \lp{Q_j P_k (Q_{<j-C}\phi^{(1)}_{k_1}\cdot \Box
      Q_{<j-C}\phi^{(2)}_{k_2})}{L^1(L^\infty)} &\lesssim
    2^{j+2k}2^{\frac{3}{2}(j-k)}2^{ \frac{1}{6}(k-k_1) }
    \lp{\phi^{(1)}_{k_1}}{S^1} \lp{\phi^{(2)}_{k_2}}{S^1} \
    . \label{Box_L1Linfty}
  \end{align}

  To prove \eqref{Q_0_L1Linfty}, note that the null form gains two
  powers of the angle $\angle(\phi^1,\phi^2)\lesssim
  2^{\frac{1}{2}(k+j)}2^{-k_1}:=2^l$. Therefore, breaking the product
  into a sum over antipodal radially directed blocks of dimension
  $2^{k}\times (2^{k+l'})^3$, where $l'=\frac{1}{2}(j-k)$, and using
  \eqref{L2L6_bernstein} for both factors we have:
  \begin{align}
    \hbox{LHS}\eqref{Q_0_L1Linfty} \ &\lesssim \
    2^{2l}2^{l'}2^{\frac{4}{3}k}
    2^{\frac{11}{6}k_1}2^{\frac{11}{6}k_2}
    \lp{\phi^{(1)}_{k_1}}{S_{k_1}[L^2(L^6)]}\lp{\phi^{(2)}_{k_2}}{S_{k_2}[L^2(L^6)]} \ , \notag\\
    &\lesssim \ 2^{j+2k} 2^{\frac{1}{2}(j-k)} 2^{\frac{1}{3}(k-k_1)}
    2^{k_1}2^{k_2}\lp{\phi^{(1)}_{k_1}}{S_{k_1}}\lp{\phi^{(2)}_{k_2}}{S_{k_2}}
    \ . \notag
  \end{align}

  To prove \eqref{Box_L1Linfty} we use the same calculations as above,
  except for the second factor we trade \eqref{L2L6_bernstein} for:
  \begin{equation}
    \big(\sum_{\mathcal{C}_{k}(l')} 
    \lp{ P_{\mathcal{C}_{k}(l')} \Box Q_{<j-C}\phi_{k_2}^{(2)} }{L^2(L^\infty)}^2
    \big)^\frac{1}{2} \ \lesssim \ 2^{2k}2^{\frac{3}{2}l'} 2^{\frac{1}{2}j}2^{k_2}
    \lp{\phi_{k_2}^{(2)}}{X_\infty^{0,\frac{1}{2}}} 
    \ , \label{Xsb_bernstein}
  \end{equation}
  which is an immediate consequence or Bernstein's inequality and
  $\Box Q_{<j-C}P_{k_2} X_\infty^{0,\frac{1}{2}} \subseteq
  2^{\frac{1}{2}j}2^{k_2}L^2(L^2)$.  This gives:
  \begin{equation}
    \hbox{LHS}\eqref{Box_L1Linfty} \ \lesssim \ 
    2^{j+2k} 2^{\frac{3}{2}(j-k)} 2^{\frac{1}{6}(k-k_1)}
    2^{k_1}2^{k_2}\lp{\phi^{(1)}_{k_1}}{S_{k_1}}\lp{\phi^{(2)}_{k_2}}{S_{k_2}} \ . \notag
  \end{equation}
\end{proof}

%%%%

\begin{proof}[Proof of estimate \eqref{mult8}]
  Freezing the output modulation, our goal here is to show:
  \begin{multline}
    \big|\big\langle (\Box\Delta)^{-1} \nabla_x Q_j P_k
    (Q_{<j-C}\phi^{(1)}_{k_1}\cdot \nabla_x Q_{<j-C} \phi^{(2)}_{k_2})
    , Q_j P_k \partial_\alpha (\partial^\alpha
    Q_{<j-C}\phi^{(3)}_{k_3}\cdot Q_{<j-C} \psi_{k_4})
    \big\rangle\big| \\
    \ \lesssim \ 2^{\frac{1}{2}(j-k)}2^{ \frac{1}{6}(k-\min\{k_i\}) }
    \lp{\phi^{(1)}_{k_1}}{S^1}
    \lp{\phi^{(2)}_{k_2}}{S^1}\lp{\phi^{(3)}_{k_3}}{S^1}
    \lp{\psi_{k_4}}{N^*} \ . \notag
  \end{multline}
  Expanding the null form and computing the weights, it suffices to
  have:
  \begin{align}
    \lp{Q_j P_k (Q_{<j-C}\phi^{(1)}_{k_1}\cdot \nabla_x Q_{<j-C}
      \phi^{(2)}_{k_2})}{L^2(L^2)} &\lesssim
    2^{\frac{1}{2}k}2^{\frac{1}{4}(j-k)}2^{ \frac{1}{6}(k-k_1) }
    \lp{\phi^{(1)}_{k_1}}{S^1}
    \lp{\phi^{(2)}_{k_2}}{S^1} \ , \label{nabla_L2}\\
    \lp{Q_j P_k (\partial_\alpha Q_{<j-C}\phi^{(3)}_{k_3}\cdot
      \partial^\alpha Q_{<j-C}\psi_{k_4})}{L^2(L^2)} &\lesssim \
    2^j2^{\frac{3}{2} k}2^{\frac{1}{4}(j-k)}2^{\frac{1}{6}(k-k_3)}
    \lp{\phi^{(3)}_{k_3}}{S^1}\lp{\psi_{k_4}}{N^*} \ , \label{Q_0_L2}\\
    \lp{Q_j P_k (\Box Q_{<j-C}\phi^{(3)}_{k_3}\cdot Q_{<j-C}
      \psi_{k_4})}{L^2(L^2)} &\lesssim \ 2^j
    2^{\frac{3}{2}k}2^{\frac{1}{4}(j-k)}
    \lp{\phi^{(3)}_{k_3}}{S^1}\lp{\psi_{k_4}}{N^*} \ . \label{Box_L2}
  \end{align}
  The proof of these three estimates follows from the same angular
  decompositions and antipodal block sums as in the last
  proof. Estimate \eqref{nabla_L2} follows by using
  \eqref{L2L6_bernstein} for the first factor and $L^\infty(L^2)$ for
  the second.  Estimate \eqref{Q_0_L2} follows similarly once once the
  null form is taken into account. Note that here the antipodal block
  sum for $\psi_{k_4}$ needs to be reconstructed for fixed time using
  orthogonality.  Finally, estimate \eqref{Box_L2} is the same as
  \eqref{Q_0_L2} but uses \eqref{Xsb_bernstein} instead for
  $\phi_{k_3}^{(3)}$.  Further details are left to the reader.
\end{proof}

% -------------------------------------------------------------------------

\subsection{Proof of the additional product estimates}

These estimates use the same kind of modulation/angular-sum
decompositions as in previous proofs, so we leave more of the details
to the reader.

\begin{proof}[Proof of estimate \eqref{mult9}]
  There are three main cases:

  \case{1}{$A_{k_1}$ with highest modulation} Due to the restriction
  of $\mathcal{H}^*$, this can only occur when the second factor is
  $Q_{k_1+C<\cdot < k_2-C}A_{k_1}$. Then there are two subcases:

  \case{1a}{Contribution of $\partial_tQ_{>k_1-C} \phi_{k_2}$} Here we
  use a product of the two bounds:
  \begin{equation}
    \lp{Q_{k_1+C<\cdot < k_2-C} A_{k_1}}{L^2(L^\infty)} \ \lesssim \ 
    2^{2k_1}\lp{A_{k_1}}{L^2(L^2)} \ , \ \
    \lp{\partial_tQ_{>k_1-C}\phi_{k_2}}{L^2(L^2)} \ \lesssim \ 2^{-\frac{1}{2}k_1}
    \lp{\partial_t \phi_{k_2}}{X_\infty^{0,\frac{1}{2}}} \ . \notag
  \end{equation}

  \case{1b}{Contribution of $\partial_tQ_{<k_1-C} \phi_{k_2}$, and
    output modulation $\gtrsim 2^{k_1}$} Note that this is the
  remaining case because an output modulation of $\ll 2^{k_1}$ is
  impossible due to the restrictions on $Q_{k_1+C<\cdot < k_2-C}
  A_{k_1}$ and $\partial_tQ_{<k_1-C} \phi_{k_2}$.  Here we use a
  product of the previous $L^2(L^\infty)$ estimate for the first and
  $L^\infty(L^2)$ for the latter.

  \case{2}{Output or $\phi_{k_2}$ have the leading modulation and
    $j<k_1+C$} In this case we can reduce things to the bounds:
  \begin{align}
    \lp{Q_j (Q_{<j+O(1)}A_{k_1}
      \partial_t Q_{<j+O(1)}\phi_{k_2} ) }{X_1^{0,-\frac{1}{2}}} \
    &\lesssim \ 2^{\frac{1}{4}(j-k_1)} 2^{\frac{3}{2}k_1} 2^{k_2}
    \lp{A_{k_1}}{L^2(L^2)}\lp{\phi_{k_2}}{S_{k_2}} \ . \notag\\
    \lp{Q_{<j+O(1)} (Q_{<j+O(1)}\phi_{k_1}^{(1)}\partial_t
      Q_j\phi_{k_2}^{(2)} ) }{L^1(L^2)} \ &\lesssim \
    2^{\frac{1}{4}(j-k_1)} 2^{\frac{3}{2} k_1} 2^{k_2}
    \lp{A_{k_1}}{L^2(L^2)}\lp{\phi_{k_2}}{S_{k_2}} \ . \notag
  \end{align}
  Both of these bounds follow from the usual angular sum
  decompositions and the Bernstein type estimate:
  \begin{equation}
    \big(\sum_{\mathcal{C}_{k_1}(\frac{1}{2}(j-k_1))} 
    \lp{ P_{\mathcal{C}_{k_1}(\frac{1}{2}(j-k_1))} Q_{<j-C}A_{k_1} }{L^2(L^\infty)}^2
    \big)^\frac{1}{2} \ \lesssim \  2^{\frac{3}{4}j}2^{\frac{5}{4}k_1}
    \lp{A_{k_1}}{L^2(L^2)} 
    \ . \notag
  \end{equation}

  \case{3}{Output or $\phi_{k_2}$ have the leading modulation and
    $j\geqslant k_1+C$} This is analogous to the last case except
  there are no angular decompositions and instead we simply use:
  \begin{equation}
    \lp{ Q_{<j-C}A_{k_1} }{L^2(L^\infty)}
    \ \lesssim \  2^{2k_1}
    \lp{A_{k_1}}{L^2(L^2)} 
    \ , \notag
  \end{equation}
  instead of the previous square sum bound.
\end{proof}\ret

%%%%

\begin{proof}[Proof of estimate \eqref{mult10}]
  This follows at one from the definition of $Z^{ell}$, which allows
  one to sum over the modulation localizations enforced by
  $\mathcal{H}^*_{k_k}$.
\end{proof}

%%%%

\begin{proof}[Proof of estimate \eqref{mult11}]
  Freezing the output modulation of the product, and getting rid of
  the contribution of $\mathcal{H}_{k_1}$, it suffices to show:
  \begin{equation}
    \lp{Q_j P_{k_1} T_a}{L^1(L^\infty)} \ \lesssim \ 2^{2k_1}2^{\frac{1}{2}(j-k_1)}2^{\frac{1}{6}(k_1-k_2)}
    \lp{\phi^{(2)}_{k_2}}{S^1}\lp{\phi^{(3)}_{k_3}}{S^1} \ , \notag
  \end{equation}
  for either one of:
  \begin{equation}
    T_1 \ = \ Q_{\geqslant j-C}\phi^{(2)}_{k_2} \partial_t \phi^{(3)}_{k_3} \ , 
    \qquad T_2 \ = \ Q_{< j-C}\phi^{(2)}_{k_2} Q_{\geqslant j-C} \partial_t \phi^{(3)}_{k_3}
    \ . \notag
  \end{equation}
  Both cases are essentially the same due to the matching
  $k_2=k_3+O(1)$. After breaking the sum into angular sectors one
  uses:
  \begin{equation}
    \big(\sum_{\mathcal{C}_{k_1}(l')} 
    \lp{ P_{\mathcal{C}_{k_1}(l')}  Q_{\geqslant j-C}\phi_{k_2}^{(2)} }{L^2(L^\infty)}^2
    \big)^\frac{1}{2} \ \lesssim \ 2^{2k_1}2^{\frac{3}{2}l'} 2^{-\frac{1}{2}j}
    \lp{\phi_{k_2}^{(2)}}{X_\infty^{0,\frac{1}{2}}} 
    \ , \notag
  \end{equation}
  for the high modulation factor and \eqref{L2L6_bernstein} for the
  low modulation factor, both with $l'=\frac{1}{2}(j-k_1)$. Note that
  $\partial_t \phi^{(3)}_{k_3}$ can be further broken into high and
  low modulation pieces, and the product of two high modulations is
  even more favorable.
\end{proof}

% -------------------------------------------------------------------------
%%%%%%%%%%%%%%%%%%%%%%%%%%%%%
% -------------------------------------------------------------------------

\section*{Appendix: A multilinear null form for MKG-CG}\label{Ap1}

Schematically we have:
\begin{equation}
  \Box A_i \ = \ -\mathcal{P}_i \Im(\phi \overline{\nabla\phi}) \ , 
  \qquad \Delta A_0 \ = \ -\Im(\phi \overline{\partial_t\phi}) \ , \notag
\end{equation}
where $\mathcal{P}$ is the Hodge projection. In 3D one can write this
operator conveniently in terms of the vector product
$\nabla\times$. In higher dimensions it is easier to use the simple
formula:
\begin{equation}
  \mathcal{P} \ = \ I - \nabla\Delta^{-1}\nabla \ , \notag
\end{equation}
where the first $\nabla$ is grad and the second is div.

Now investigate the expression:
\begin{equation}
  \mathcal{N}(A,\phi) \ = \ A^\alpha\partial_\alpha \phi \ . \notag
\end{equation}
Directly plugging things in we have:
\begin{equation}
  \mathcal{N}(A,\phi) \ = \ \Delta^{-1}\Im(\phi\overline{\partial_t\phi})\cdot \partial_t\phi -
  \Box^{-1}\Im(\phi\overline{\partial_i \phi})\cdot \partial^i\phi + \frac{\partial^i\partial^j}{\Delta \Box}
  \Im(\phi\overline{\partial_i\phi})\cdot \partial_j\phi \ . \notag
\end{equation}
To uncover the null structure, we add and subtract the expression
$\Box^{-1}\Im(\phi\overline{\partial_t \phi})\partial_t\phi$. This
gives us:
\begin{equation}
  \mathcal{N}(A,\phi) \ = \ -\mathcal{Q}_1(A,\phi) + \mathcal{N}_2(A,\phi) \ , \notag
\end{equation}
where:
\begin{align}
  \mathcal{Q}_1(A,\phi) \ &= \
  \Box^{-1}\Im(\phi\overline{\partial_\alpha\phi})\cdot
  \partial^\alpha\phi
  \ , \notag\\
  \mathcal{N}_2(A,\phi) \ &= \
  \Delta^{-1}\Im(\phi\overline{\partial_t\phi})\cdot \partial_t\phi -
  \Box^{-1}\Im(\phi\overline{\partial_t \phi})\cdot \partial_t\phi
  +\frac{\partial^i\partial^j}{\Delta\Box}
  \Im(\phi\overline{\partial_i\phi})\cdot \partial_j\phi \ . \notag
\end{align}
To continue the computation we use
$\Delta^{-1}-\Box^{-1}=-\partial_t^2 \Delta^{-1}\Box^{-1}$ so the
second term simplifies to:
\begin{equation}
  \mathcal{N}_2(A,\phi) \ = \ -\Delta^{-1}\Box^{-1} \partial_t^2 \Im(\phi\overline{\partial_t\phi}) 
  \cdot \partial_t\phi
  + \Delta^{-1}\Box^{-1} 
  \partial^i\partial^j\Im(\phi\overline{\partial_i\phi})\cdot \partial_j\phi  \ . \notag
\end{equation}
This can be further reduced to the sum of $Q_0$ null structures by
adding and subtracting (for instance)
$\partial^i\partial_t\Im(\phi\overline{\partial_i\phi})\partial_t\phi$,
which gives:
\begin{equation}
  \mathcal{N}_1(A,\phi) \ = \ \mathcal{Q}_2(A,\phi) + \mathcal{Q}_3(A,\phi) \ , \notag
\end{equation}
where:
\begin{align}
  \mathcal{Q}_{2}(A,\phi) \ &= \
  \Delta^{-1}\Box^{-1} \partial_t\partial_\alpha
  \Im(\phi\overline{\partial^\alpha\phi}) \cdot \partial_t \phi
  \ , \notag\\
  \mathcal{Q}_{3}(A,\phi) \ &= \ \Delta^{-1}\Box^{-1}
  \partial_\alpha \partial^i\Im(\phi\overline{\partial_i\phi})\cdot \partial^\alpha\phi
  \ . \notag
\end{align}

There is a conceptual way to visualize this by duality. The expression
$A^\alpha\partial_\alpha \phi$ is a sum of three $Q_0$ null
structures: The interaction is between the first two factors, the
second two factors, or mixed between $\phi^2$ and $\phi^3$. The latter
is the main one which involves null frames.

\end{document}